\documentclass[jmp,amsart,amsmath,reqno,nofootinbib]{revtex4-1}
\usepackage{amssymb}
\usepackage{graphicx}
\usepackage{setspace}
\usepackage{hyperref}
\usepackage{xcolor}
\definecolor{dark-blue}{rgb}{0.15,0.15,0.4}
\hypersetup{
    colorlinks=true,       
    linkcolor=black,          
    citecolor=black,        
    urlcolor=dark-blue
}

\makeatletter
\def\thm@space@setup{
\thm@preskip=4pt
\thm@postskip=4pt}
\makeatother

\newtheorem{thm}{Theorem}[section]
\newtheorem{prop}[thm]{Proposition}
\newtheorem{lem}[thm]{Lemma}
\newtheorem{cor}[thm]{Corollary}

\newtheorem{remark}[thm]{Remark}

\def\^{{\wedge}}
\def\*{{\star}}
\def\bar{\overline}

\def\e#1{{\rm e}^{\, #1}}

\def\ha{\frac{1}{2}}

\def\BH{{\mathbb H}}

\def\BR{{\mathbb R}}

\def\BZ{{\mathbb Z}}

\def\CN{{\mathcal N}}
\def\CO{{\mathcal O}}

\def\CQ{{\mathcal Q}}

\def\CU{{\mathcal U}}
\def\CV{{\mathcal V}}

\def\Rc{{\mathsf c}}
\def\Rh{{\mathsf h}}

\def\RJ{{\mathsf J}}
\def\Rr{{\mathsf r}}
\def\RR{{\mathsf R}}
\def\Rw{{\mathsf w}}
\def\RZ{{\mathsf Z}}

\DeclareMathOperator{\lk}{\mathsf{lk}}
\DeclareMathOperator{\slk}{\mathsf{slk}}
\DeclareMathOperator{\w}{\mathsf w}
\DeclareMathOperator{\tb}{\mathsf{tb}}
\DeclareMathOperator{\rot}{\mathsf{rot}}
\DeclareMathOperator{\sgn}{\mathrm{sgn}}

\begin{document}
\title{Remarks on Legendrian Self-Linking}
\author{Chris Beasley,\\ {\sl Department of Mathematics, Northeastern University, 
Boston MA 02115, USA}\\[2 ex]
Brendan McLellan,\\ {\sl Department of Mathematics \& Statistics, Utah State University, 
Logan UT 84322, USA}\\[2 ex]
and \\[2 ex]
Ruoran Zhang\\ {\sl Department of Mathematics, Northeastern University, 
Boston MA 02115, USA}}

\begin{abstract}
\bigskip
\baselineskip=15pt
The Thurston-Bennequin invariant provides one notion of self-linking for
any homologically-trivial Legendrian curve in a contact
three-manifold.  Here we discuss related analytic notions of self-linking for
Legendrian knots in $\BR^3$.  Our definition is based
upon a reformulation of the elementary Gauss linking integral and
is motivated by ideas from supersymmetric gauge theory.  We recover
the Thurston-Bennequin invariant as a special case.
\end{abstract}

\maketitle 

\section{Introduction}
The linking number $\lk(C_1,C_2)$ of disjoint, oriented curves ${C_1, C_2
  \subset \BR^3}$ is among the most basic invariants in knot theory, and
as such it admits many different descriptions.  We begin with
three.

Let $(x,y,z)$ be Euclidean coordinates on $\BR^3$ and consider the
global angular form \cite{BottTu:82} 
\begin{equation}\label{Angular}
\psi \,=\, \frac{1}{4\pi}\frac{x \, dy\^dz \,-\, y \, dx\^dz \,+\, z
  \, dx\^dy}{\left[x^2 \,+\, y^2 \,+\,
    z^2\right]^{3/2}}\,\in\,\Omega^2\big(\BR^3-\{0\}\big)\,.
\end{equation}  
With the given normalization, ${\psi=\varrho^*\omega}$ is the pullback
of an \mbox{$SO(3)$-invariant}, unit-volume form $\omega$ on ${S^2 \subset
  \BR^3}$ under the retraction 
\begin{equation}\label{Retrac}
\begin{aligned}
&\varrho:\BR^3-\{0\} \,\longrightarrow\,S^2\,,\\
&\varrho(x,y,z) \,=\,
\frac{\left(x,y,z\right)}{\left[x^2\,+\,y^2\,+\,z^2\right]^{1/2}}\,.
\end{aligned}
\end{equation}
Without loss, we take the embedded curves ${C_i}$ for ${i=1,2}$ to be
parametrized by smooth maps 
\begin{equation}
X_i:S^1\,\longrightarrow\,\BR^3\,,\qquad i=1,2\,,
\end{equation}
in terms of which we write the difference
\begin{equation}\label{DiffG}
\begin{aligned}
&\Gamma:\BR^3\times\BR^3\,\longrightarrow\,\BR^3\,,\\
&\Gamma(X_1,X_2)\,=\, X_2 - X_1\,.
\end{aligned}
\end{equation}
The Gauss formula for the linking number of $C_1$
and $C_2$ is then given by an integral over the torus ${T^2 = S^1
  \times S^1}$,
\begin{equation}\label{Gauss}
\lk(C_1,C_2) \,=\, \int_{T^2} \left(X_1\times
  X_2\right)^*\!\Gamma^*\psi\,.
\end{equation}
Essential here, since $C_1$ and $C_2$ are disjoint space curves, the
singularity of $\psi$ at the origin is avoided, and the
Gauss integrand is everywhere smooth and bounded on $T^2$.

Alternatively, because $\omega$ represents a generator for the cohomology
$H^2(S^2;\BZ)$, the Gauss linking integral computes the topological degree
\begin{equation}\label{Degree}
\lk(C_1,C_2) \,=\, \deg\varphi_{12} \,\in\,\BZ\,,
\end{equation}
where $\varphi_{12}$ is the composition
\begin{equation}\label{LkPhi}
\varphi_{12}:T^2\,\longrightarrow\,S^2\,,\qquad\qquad\varphi_{12} \,=\,
\varrho\circ\Gamma\circ(X_1\times X_2)\,.
\end{equation}
In its second description as the degree of $\varphi_{12}$, the linking
number is clearly an integer and invariant under smooth isotopies of
the curves $C_{1,2}$.
The overall sign of the linking number depends upon the orientations for
$T^2$ and $S^2$.  Throughout, if $(\theta_1,\theta_2)$ are angular
coordinates on $T^2$, we orient the torus by $d\theta_1\^d\theta_2$.
We similarly give ${S^2\subset\BR^3}$ the orientation induced from the standard
orientation $dx\^dy\^dz$, for which the cohomology generator ${\omega >
  0}$ is positive.

\begin{figure}[t]
$$\begin{matrix}
&\includegraphics[scale=0.50]{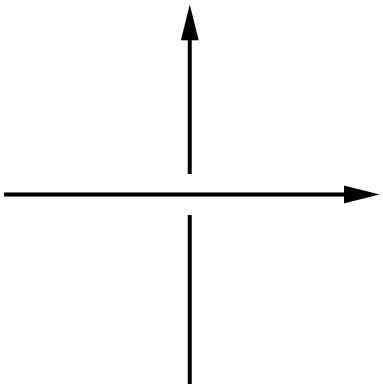}\qquad & \qquad 
&\includegraphics[scale=0.50]{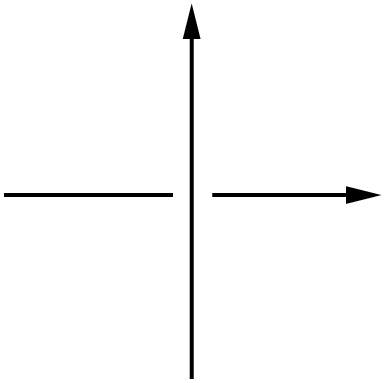}\\
&(a.)\quad \w=+1 \qquad & \qquad &(b.)\quad \w=-1
\end{matrix}$$
\caption{Writhe at right- and left-handed crossings.}\label{Writhe}
\end{figure}
Famously the linking number admits a third, diagrammatic description,
convenient for computations.  Again without loss, we consider a
generic plane projection for $C_1$ and $C_2$ in which only simple,
double-point crossings are present.   The index set $\textbf{I}$ of all
crossings in the planar diagram for the link divides into subsets
${\textbf{I} =
  \textbf{I}^1_1\cup\textbf{I}^1_2\cup\textbf{I}^2_1\cup\textbf{I}^2_2}$,   
depending upon whether the upper and lower strands at each crossing
belong respectively to $C_1$ or $C_2$.  To
each crossing ${a\in\textbf{I}}$ we attach a local writhe
${\w_a=\pm 1}$ according to the handedness, as in Figure
\ref{Writhe}.
In terms of these data, the linking number is computed by any of
the following sums,
\begin{equation}\label{LkDiag}
\lk(C_1,C_2) \,=\, \sum_{a \in \textbf{I}^2_1} \w_a \,=\,
\sum_{b \in \textbf{I}^1_2} \w_b \,=\, \ha \sum_{c \in
  \textbf{I}^2_1 \cup \textbf{I}^1_2} \w_c\,.
\end{equation}
The signed sum of crossings by $C_2$ over $C_1$ can be interpreted
as a signed count of preimages $\varphi_{12}^{-1}(p)$ for ${p\in S^2}$ at the
North Pole, so the first equality in \eqref{LkDiag} 
follows directly from the topological description of $\lk(C_1,C_2)$ as
the degree of $\varphi_{12}$.  One can  
check that our orientation conventions for $T^2$ and $S^2$ are
consistent with the assignment of signs for $\w$ in Figure
\ref{Writhe}.  The second equality in \eqref{LkDiag} follows similarly
with ${p\in S^2}$ at the South Pole, and the third equality is the symmetric
combination of the preceding two.

\subsection{Perspectives on self-linking}

The present article is concerned not with linking but with self-linking, for
which one would like to make sense of $\lk(C_1,C_2)$ as a knot
invariant in the degenerate
case ${C_1 = C_2}$.  This problem does not have a unique solution, and
several notions of self-linking already exist.  Our
purpose is to propose another, motivated by gauge theory
\cite{Beasley:2015,ZhangR:2016} and with an
eye towards higher-order self-linking invariants, defined in the
geometric style of \cite{Bott:95g} as integrals over configuration
spaces of points associated to the knot.

The most naive attempt to define a self-linking number for an
embedded, oriented curve ${C \subset \BR^3}$ is based upon the Gauss
integral in \eqref{Gauss}.  Again parametrizing $C$ as the
image of a smooth map
\begin{equation}
X:S^1\,\longrightarrow\,\BR^3\,,
\end{equation} 
we follow our nose to set 
\begin{equation}\label{SLKZero}
\slk_0(C) \,:=\, \underset{\varepsilon\to 0}{\lim}\int_{T^2 -
  \Delta(\varepsilon)} \left(X \times
  X\right)^*\!\Gamma^*\psi\,.
\end{equation}
Unlike \eqref{Gauss}, the self-linking integrand is now singular along
the diagonal ${\Delta \subset T^2}$, due to the divergence of $\psi$
at the origin in $\BR^3$.  To deal with the singularity, we excise a tubular
neighborhood $\Delta(\varepsilon)$ of width ${0<\varepsilon\ll 1}$
about the diagonal, depicted as the shaded region in Figure
\ref{Deleps}, and we integrate only over the resulting cylinder ${T^2
  - \Delta(\varepsilon)}$.  We finally take the limit ${\varepsilon\to
  0}$ to remove any dependence on the auxiliary parameter.

\begin{figure}[t]
$$\begin{matrix}
\raisebox{-60pt}{\includegraphics[scale=0.60]{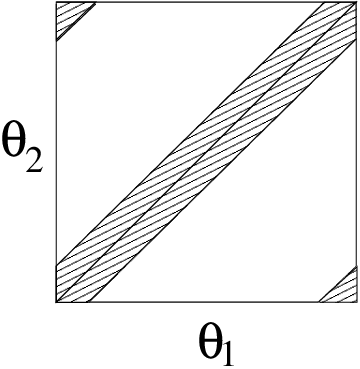}}
\,&\simeq\,\,\,\raisebox{-42pt}{\includegraphics[scale=0.30]{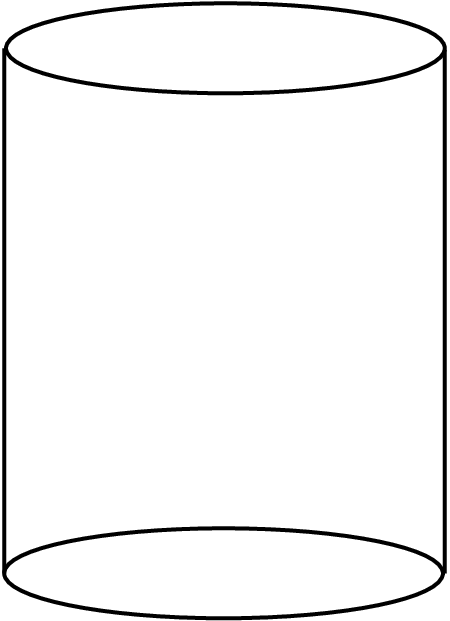}}
\end{matrix}$$
\caption{Neighborhood $\Delta(\varepsilon)$ of the diagonal in
  ${T^2}$.}\label{Deleps}
\end{figure}

There are two essential statements to make about the naive self-linking
integral in \eqref{SLKZero}, both of which go back to the classic
works by C\u{a}lug\u{a}reanu \cite{Calugareanu:59} and Pohl
\cite{Pohl:67}.  First, the singularity along the diagonal $\Delta$ is
integrable, so the limit defining ${\slk_0(C)\in\BR}$ does exist.

Second, $\slk_0(C)$ is {\sl not} a deformation-invariant of $C$, but rather
varies to first-order under a change $\delta X$ in the embedding, 
\begin{equation}\label{VarSLK}
\delta\slk_0(C) \,=\, -\frac{1}{2\pi}\oint_{S^1}\!\!ds\,\,
\epsilon_{\mu\nu\rho}\, \dot{X}^\mu\,\delta
X^\nu\,\raisebox{1.25pt}{$\dddot{X}{}^{\!\rho}$}\,,\qquad\qquad
||\dot{X}(s)||^2 \,=\, 1\,.
\end{equation}
To simplify the expression on the right of \eqref{VarSLK}, we have taken
$X^\mu(s)$ to be a regular, unit-speed parametrization, ${\dot{X}^\mu
  \equiv dX^\mu/ds}$, for which $ds$ is the arc-length
measure on $C$.  Also, $\epsilon_{\mu\nu\rho}$ for ${\mu,\nu,\rho = 1,2,3}$
is the fully anti-symmetric tensor, normalized so that
${\epsilon_{123} = +1}$.  As a corollary of \eqref{VarSLK}, the value
of the naive self-linking integral depends non-trivially on the geometry of ${C
  \subset \BR^3}$, and it cannot generically be an integer,
${\slk_0(C) \notin \BZ}$.

The naive attempt to define a self-linking invariant fails for two
reasons, one topological and one analytic.  

From the topological
perspective, the primary difficulty is that the domain of integration
${T^2 -\Delta(\varepsilon)}$ has a boundary, illustrated in Figure
\ref{Deleps}.  When the embedding map $X$
varies in \eqref{SLKZero}, the integrand changes by a
cohomologically-trivial two-form, whose integral over the cylinder ${T^2
  -\Delta(\varepsilon)}$ receives a non-trivial boundary contribution for
${\varepsilon>0}$.

The analytic aspect of the self-linking anomaly is not as obvious, but
no less important.  To wit, in the detailed computation leading to
\eqref{VarSLK}, the boundary contribution to $\delta\slk_0(C)$ remains non-vanishing even
in the limit ${\varepsilon\to 0}$.  See the Ph.D. thesis of Bar-Natan
\cite{BarNatan:91} for a very nice presentation of the anomaly calculation.  This phenomenon 
depends very much on the analytic behavior of the angular form $\psi$ near the origin,
and need not have occurred had the divergence of $\psi$ been less
rapid.  This observation will be central to our work and is exemplified by the
Fundamental Lemma in Section \ref{Fundamental}.

Because of the anomaly, all notions of self-linking involve a
modification to the naive Gauss integral in \eqref{SLKZero}, as well as
a refinement of the equivalence relation by smooth isotopy for the
curve ${C\subset \BR^3}$.  

For instance, in the original approach of
\cite{Calugareanu:59,Pohl:67},
one adds to $\slk_0(C)$ a counterterm $T(C)$ whose
variation with respect to $X$ is exactly opposite the variation in
\eqref{VarSLK}, 
\begin{equation}\label{VarTC}
\delta T(C) \,=\, -\delta\slk_0(C)\,.
\end{equation}
The required counterterm is nothing more than the total torsion of
$C$ measured with respect to arc-length,
\begin{equation}
T(C) \,=\, \frac{1}{2\pi}\oint_{S^1} \!\!ds\,\,\tau\,,
\end{equation}
where the local Frenet-Serret \cite{doCarmo:76} torsion $\tau$ is given in
terms of $X$ by 
\begin{equation}\label{Tors}
\tau \,=\, \frac{\epsilon_{\mu \nu \rho} \, \dot{X}^\mu \,
  \ddot{X}^\nu \, \raisebox{1.25pt}{$\dddot{X}{}^{\!\rho}$}}{||\dot{X}
  \times \ddot{X}||^2}\,,\qquad\qquad \ddot{X}(s) \,\neq\, 0\,.
\end{equation}
For sake of brevity, we do not review the proof of \eqref{VarTC}
here.  A direct calculation of $\delta T(C)$ can be found in 
Appendix B of \cite{Beasley:2015}.

From the relation in \eqref{VarTC}, the sum 
\begin{equation}\label{SLKTau}
\slk_\tau(C) \,=\, \slk_0(C) \,+\, T(C)\,
\end{equation}
does not change under any small deformation of $C$,
so it provides a reasonable notion of self-linking.  However,
$\slk_\tau(C)$ is also not invariant under arbitrary smooth isotopies
of $C$.  Due to the denominator in \eqref{Tors}, the torsion $\tau$ is
only defined at points of $C$ where ${\ddot{X}(s) \neq 0}$, or
equivalently, where the Frenet-Serret curvature is non-zero.  To make
sense of $T(C)$, we assume that $C$ has everywhere
non-vanishing curvature.  The latter condition is open and so holds
for the generic space curve, but of course it does not hold for all
space curves.  As a result, $\slk_\tau(C)$ is only invariant under
those `non-degenerate' isotopies which preserve the condition of
non-vanishing curvature along $C$.

\begin{figure}[t]
\begin{center}
\includegraphics[scale=0.50]{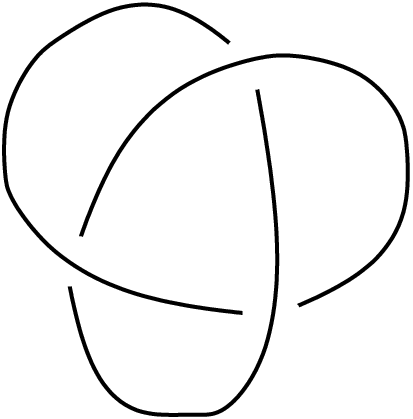}
\caption{Trefoil knot with ${\w(C) = -3}$.}\label{Trefoil}
\end{center}
\end{figure}

Like the linking number, the self-linking number in \eqref{SLKTau} has
a simple diagrammatic description.  To derive it, consider the smooth isotopy
\begin{equation}\label{FLambda}
F_\Lambda:\BR^3\,\longrightarrow\,\BR^3\,,\qquad\qquad F_\Lambda(x,y,z) \,=\,
\left(x,y,\frac{z}{\Lambda}\right),
\end{equation}
where ${\Lambda > 0}$ is a positive real parameter.  Applied to any
curve ${C \subset \BR^3}$, this isotopy flattens $C$ to 
the $xy$-plane as $\Lambda$ grows to infinity.  Rotating $C$ if
necessary, we assume that both $C$ and its projection to the
$xy$-plane have everywhere non-vanishing curvature.  See Figure \ref{Trefoil}
for a diagram of the trefoil knot which satisfies this condition.
Then $F_\Lambda$ is non-degenerate for all values of $\Lambda$, and
$\slk_\tau(C)$ can be evaluated in the limit ${\Lambda\to\infty}$.
For any plane curve, ${\tau = 0}$ identically, so $\slk_\tau(C)$
reduces to $\slk_0(C)$ in this limit.  Otherwise, when $\Lambda$ is
large and $C$ is nearly planar, the naive self-linking integral in
\eqref{SLKZero} reduces to a sum over the local writhe at each
self-crossing of $C$.  We omit a proof of the latter claim, which is
hopefully plausible on the basis of the similar description
\eqref{LkDiag} for $\lk(C_1,C_2)$.  In Section \ref{Local},
we will establish a closely related result as part of our Main Theorem.

Altogether, $\slk_\tau(C)$ is given by the writhe $\w(C)$ of the
planar diagram for $C$,
\begin{equation}
\slk_\tau(C) \,=\, \w(C) \,:=\, \sum_{a \in \textbf{I}} \w_a\,.
\end{equation}
In particular, ${\slk_\tau(C) \in \BZ}$ is an integer, which is not
manifest from the analytic description.  Also, because the assignment
of the writhe ${\w=\pm 1}$ in Figure \ref{Writhe} is
invariant under a reversal of orientation for both strands,
$\slk_\tau(C)$ does not depend upon the orientation of $C$.  Finally,
the various versions of the unknot in Figure \ref{WUnknot} illustrate that
$\slk_\tau(C)$ cannot be a full isotopy invariant of $C$.

\begin{figure}[t]
$$\begin{matrix}
&\includegraphics[scale=0.50]{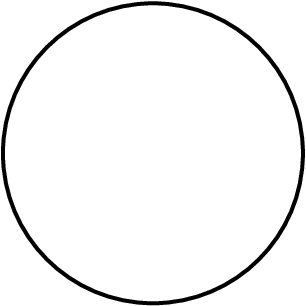} \qquad & \qquad 
\includegraphics[scale=0.50]{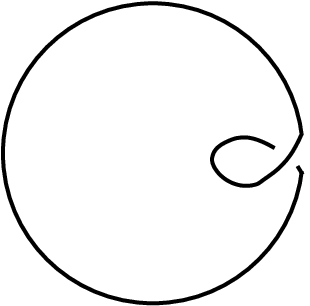} \qquad & \qquad
\includegraphics[scale=0.50]{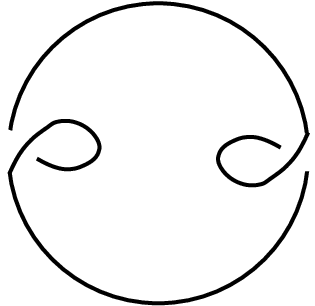}\\[1 ex]
&(a.)\quad \w(C) = 0 \qquad & \qquad (b.)\quad \w(C) = 1 \qquad &
\qquad (c.)\quad
\w(C) = 2
\end{matrix}$$
\caption{Writhe for a selection of unknots.}\label{WUnknot}
\end{figure}

The Frenet-Serret self-linking number $\slk_\tau(C)$ is really a
special case of the more general, and perhaps more familiar, notion of
framed self-linking.  By definition, a framing of
${C \subset \BR^3}$ is a trivialization of the normal bundle $N_C$, up
to homotopy.  Such a trivialization can be specified by a
nowhere-vanishing normal vector field ${n \in \Gamma(C, N_C)}$ along
the knot.  In turn, the normal vector field determines a new curve
$C_n$ obtained by displacing $C$ a small amount in the direction of
$n$, as shown in Figure \ref{Framed}.  

\begin{figure}[h]
\begin{center}
\includegraphics[scale=0.50]{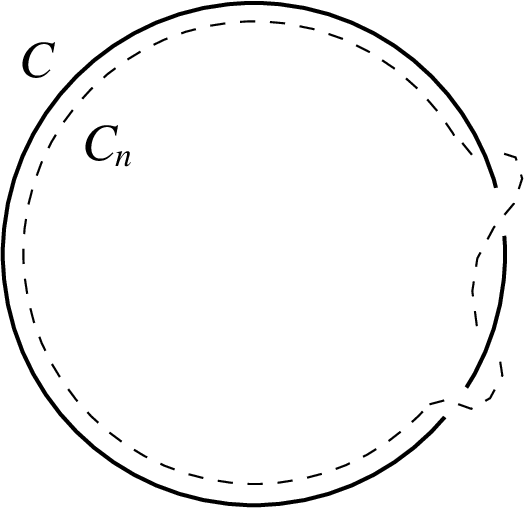}
\caption{A framed unknot, with ${\slk_{\rm f}(C,n) = 2}$.}\label{Framed}
\end{center}
\end{figure}

Given the pair $(C,n)$, the framed self-linking number is defined as
the ordinary linking number of the two disjoint curves $C$ and $C_n$,
\begin{equation}
\slk_{\rm f}(C,n) \,:=\, \lk(C,C_n)\,.
\end{equation}
On the upside, $\slk_{\rm f}(C,n)$ is manifestly an integral isotopy
invariant of the pair $(C,n)$.  On the downside, $\slk_{\rm f}(C,n)$
carries no information about the knot $C$ itself, since the invariant
takes all possible values as the winding number of $n$ about $C$
shifts. 

Had one a canonical choice of framing, the framed self-linking number
$\slk_{\rm f}(C,n)$ could be converted into an honest invariant of
$C$.  No such choice exists for all smooth curves 
simultaneously, but a variety of choices can be made if one restricts
to special classes of curves.  

We have already encountered an example
in the discussion of the Frenet-Serret self-linking $\slk_\tau(C)$,
for which we require ${C\subset \BR^3}$ to have everywhere
non-vanishing curvature.  Not coincidentally, such curves also carry a
canonical Frenet-Serret framing, with unit normal ${n =
  \ddot{X}/||\ddot{X}||}$.  Indeed, a natural guess is that
${\slk_\tau(C) = \slk_{\rm f}(C,n)}$ for the Frenet-Serret normal.

This guess is correct, as can be seen most easily by considering the
behavior of the Frenet-Serret frame in the planar limit
${\Lambda\to\infty}$ from \eqref{FLambda}.  In this limit, the 
Frenet-Serret framing reduces to the blackboard framing in which
the unit normal vector $n$ lies everywhere in the plane of the knot
diagram.  After displacing $C$ by $n$, one finds a planar ribbon
graph, shown in the neighborhood of a positive crossing on the left in Figure
\ref{FSFrame}.  For such a graph, each self-crossing of $C$ corresponds to a
crossing of $C$ by $C_n$ with identical chirality, so automatically
${\w(C) = \lk(C,C_n)}$.  The relation between the writhe and the
Frenet-Serret framing can also be understood in three-dimensional
terms, as indicated to the right in Figure \ref{FSFrame}.

\begin{figure}[t]
$$\begin{matrix}
&\raisebox{-23pt}{\includegraphics[scale=0.50]{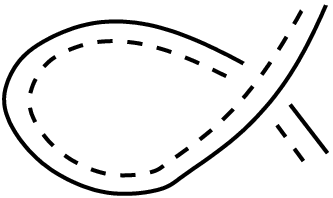}}
\quad&\simeq\quad 
&\raisebox{-40pt}{\includegraphics[scale=0.50]{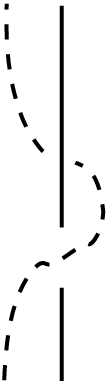}}\\[1.5 ex]
&\w \,=\, +1 \qquad&\qquad &\slk_{\rm f} \,=\, +1
\end{matrix}
$$
\caption{Relation between writhe and Frenet-Serret framing.}\label{FSFrame}
\end{figure}

\subsection{Legendrian knots}

This article serves as the companion to a longer work
\cite{Beasley:2015} in which we develop a new, effectively
supersymmetric formulation of Chern-Simons perturbation theory.  Both
the Frenet-Serret and the framed self-linking invariants were
rediscovered physically
\cite{Guadagnini:91g,Polyakov:1988md,Witten:1989hf} in the setting of
bosonic Chern-Simons theory, and our purpose is to present what one finds with
the addition of ${\CN=2}$ supersymmetry, after all the
baggage from gauge theory is removed.  Details about the gauge
theory are discussed in \cite{Beasley:2015,ZhangR:2016}.

Supersymmetry has two consequences.  We discuss the first now, and we 
defer a discussion of the second to the next subsection.

As a first consequence of supersymmetry, ${C \subset \BR^3}$ must be Legendrian with
respect to the standard contact structure on $\BR^3$, and the
supersymmetric self-linking number will be a Legendrian isotopy
invariant.  For very enjoyable introductions to contact 
topology and the study of Legendrian knots, see
\cite{Etnyre:2003,Etnyre:2006j,Geiges:2006,Geiges:2008}.  At the
moment, we recall only the minimum necessary to state and prove our
Main Theorem.

Throughout, we represent the standard contact structure on $\BR^3$ with
the contact form 
\begin{equation}\label{Contct}
\kappa \,=\, dz \,+\, x\,dy\,-\, y\,dx\,,
\end{equation}
for which the top-form ${\kappa\^d\kappa>0}$ is positive with respect
to the standard orientation on $\BR^3$.  This choice of contact form
respects many symmetries, which are important here and in
\cite{Beasley:2015,ZhangR:2016}.

Manifestly, $\kappa$ is preserved under translations generated by the Reeb field
${\RR=\partial/\partial z}$ as well as rotations in the $xy$-plane.
Though $\kappa$ is not preserved by translations in the $xy$-plane, 
$\kappa$ is preserved by the left-action of the Heisenberg Lie group
${\BH \simeq \BR^3}$ on itself, where the Heisenberg multiplication
${\mu:\BH\times\BH\to\BH}$ is given by 
\begin{equation}\label{Heis}
\mu\big(X_1,X_2\big) \,=\,
\left(x_1 + x_2,\, y_1 + y_2,\, z_1 + z_2 - x_1 y_2 + x_2
  y_1\right),\qquad X_{1,2}\,\in\,\BH\,.
\end{equation}
The origin remains the identity in $\BH$, and the Heisenberg inverse
is ${X^{-1} = \left(-x,-y,-z\right)}$. 
Finally, $\kappa$ is homogeneous of degree two under the parabolic
scaling 
\begin{equation}\label{Parabol}
(x,y,z) \,\longmapsto\, \left(\lambda x,\,\lambda
  y,\,\lambda^2 z\right),\qquad\qquad \lambda \,\in\,\BR_+\,,
\end{equation}
which commutes with Heisenberg multiplication.  As
a result, the parabolic scaling fixes each
contact plane ${H \subset \BR^3}$ in the kernel of $\kappa$.  

A picture  of the family of contact planes ${H =
  \ker\kappa}$ appears in Figure \ref{StandardH}.
The contact planes are approximately horizontal near 
${x=y=0}$, but they twist vertically as one moves
outward from the origin in the $xy$-plane.
\begin{figure}[t]
\begin{center}
\includegraphics[scale=1.0]{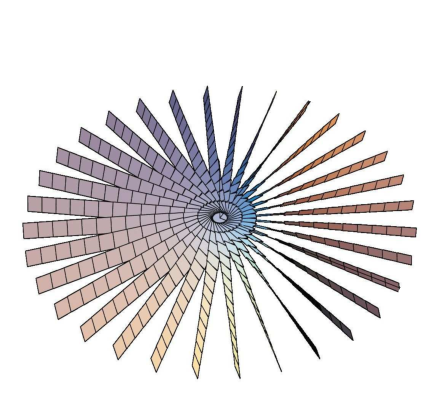}
\caption{The standard radially-symmetric contact structure on
  $\mathbb{R}^{3}$.  (Courtesy of the  {\tt
    Mathematica} routine `CSPlotter' by Matias Dahl.)}\label{StandardH}
\end{center}
\end{figure}

We also require a few facts about Legendrian knots.  By definition, ${C
  \subset \BR^3}$ is Legendrian when the tangent line ${T_pC}$ at any
point ${p\in C}$ lies in the contact plane $H_p$ at that point.
Equivalently, the pullback of $\kappa$ to $C$ vanishes,
\begin{equation}
\kappa\big|_C \,=\, 0 \qquad\Longleftrightarrow\quad C \,\,\hbox{is
  Legendrian}\,, 
\end{equation}
or in terms of a parametrization ${X:S^1\to\BR^3}$,
\begin{equation}\label{LegC}
\frac{dz}{d\theta} \,=\, y \, \frac{dx}{d\theta} \,-\, x \,
\frac{dy}{d\theta}\,,\qquad\qquad X(\theta) \equiv
  \big(x(\theta),\,y(\theta),\,z(\theta)\big)\,.
\end{equation}
Any smooth knot admits a Legendrian representative, so the theory of
Legendrian knots is extremely rich.  Moreover, equivalence by
Legendrian isotopy, ie.~continuous isotopy through a family of Legendrian
knots, strictly refines the usual topological equivalence.  A
given topological knot has infinitely-many inequivalent
Legendrian representatives.

Unlike topological knots, Legendrian knots have canonical plane
projections.  For this reason, Legendrian knots behave in many ways
like plane curves.  
Our interest lies in the so-called Lagrangian projection to
the $xy$-plane,
\begin{equation}
\Pi:\BR^3\,\to\,\BR^2\,,\qquad\qquad \Pi(x,y,z) \,=\, (x,y)\,,
\end{equation}
for which the image $\Pi(C)$ of a Legendrian knot ${C \subset \BR^3}$ is
a smoothly immersed curve.  The smoothness of $\Pi(C)$
is already a non-trivial feature of the Legendrian condition
\eqref{LegC}, since this condition implies that ${\dot{z} = 0}$ at
any point on $C$ where ${\dot{x}=\dot{y}=0}$.  Hence if $X$ is a
regular parametrization of $C$, then ${\Pi \circ X}$ is a
regular parametrization of $\Pi(C)$.  Trivially, $\Pi(C)$ is a
Lagrangian submanifold of $\BR^2$ with respect to the 
symplectic form ${d\kappa = 2\,dx\^dy}$, whence the name.

\begin{figure}[t]
\begin{center}
\includegraphics[scale=0.5]{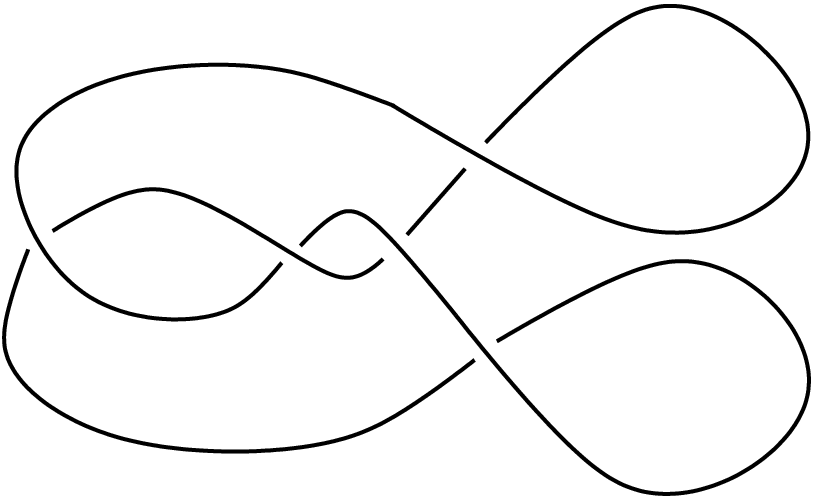}
\caption{Legendrian trefoil knot, with ${\tb(C)=1}$ and
  ${\rot(C)=0}$.}\label{LgTref}
\end{center}
\end{figure}

In Figure \ref{LgTref} we display the Lagrangian projection
of a Legendrian trefoil knot.  To guide the eye, we indicate
over- and under-crossings in the figure.  Unlike for topological knot
diagrams, the crossing information for a Legendrian knot is redundant, since the
spatial configuration of ${C \subset \BR^3}$ can be completely
recovered from the plane curve $\Pi(C)$ by integrating the contact relation in
\eqref{LegC},
\begin{equation}\label{ZTheta}
z(\theta) \,=\, z_0\,+\, \int_0^\theta
d\theta'\left[y(\theta')\,\dot{x}(\theta') \,-\,
  x(\theta')\,\dot{y}(\theta')\right]. 
\end{equation}
Here $z_0$ is the height of $C$ at the basepoint corresponding to
${\theta=0}$.  Due the symmetry of $\kappa$ under translations in $z$,
this constant is both arbitrary and irrelevant. 

Not every immersed curve can be the plane projection of a
Legendrian knot.  For instance, if we take the parameter $\theta$ in
\eqref{ZTheta} to have periodicity $2\pi$, then 
\begin{equation}\label{Period}
0 \,=\, z(2\pi) - z(0) \,=\, \int_0^{2\pi}
d\theta'\left[y(\theta')\,\dot{x}(\theta') \,-\,
  x(\theta')\,\dot{y}(\theta')\right].
\end{equation}
Equivalently by Stokes' Theorem, the oriented $2$-chain $D$ bounded by $\Pi(C)$ must have zero symplectic area,
\begin{equation}
\int_D dx\^dy\,=\, 0\,,
\end{equation}
where each component of $D$ is oriented consistently with ${\partial D
  = C}$.  Also, if ${\theta_1 \neq \theta_2}$ are distinct parameter
values for which ${x(\theta_1) = x(\theta_2)}$ and 
${y(\theta_1)=y(\theta_2)}$, corresponding to the location of a crossing in
$\Pi(C)$, then 
\begin{equation}\label{Embed}
0 \neq z(\theta_2) - z(\theta_1) \,=\, \int_{\theta_1}^{\theta_2}
d\theta'\left[y(\theta')\,\dot{x}(\theta') \,-\,
  x(\theta')\,\dot{y}(\theta')\right].
\end{equation}
The necessary conditions in \eqref{Period} and \eqref{Embed} are sufficient
for the immersed plane curve to lift to an embedded Legendrian knot.
These conditions depend upon the signed areas of the regions enclosed by
$\Pi(C)$, so the Lagrangian projection cannot be manipulated in a
wholly topological fashion \'a la Reidemeister. 

Legendrian knots always carry a canonical framing by the Reeb vector
field ${\RR = \partial/\partial z}$.  Concretely from \eqref{LegC},
a Legendrian curve ${C \subset \BR^3}$ cannot have a vertical tangent, where
${\dot{x}=\dot{y}=0}$ but ${\dot{z}\neq 0}$.  With the choice ${n =
  \RR}$, the framed self-linking number $\slk_{\rm f}(C,n)$ can then be converted
into a Legendrian invariant
\begin{equation}
\tb(C) \,:=\, \slk_{\rm f}(C,\RR)\,\in\,\BZ\,,
\end{equation}
a kind of self-linking number for $C$.

The Thurston-Bennequin invariant $\tb(C)$ can be easily computed from
the Lagrangian projection of $C$.  After a rigid rotation, the 
vertical framing by the Reeb field $\RR$ becomes equivalent to the
planar, blackboard framing of $\Pi(C)$.  But again by Figure \ref{FSFrame}, the
self-linking number in the blackboard framing is exactly the writhe of
the knot diagram.  Thus,
\begin{equation}
\tb(C) \,=\, \w\!\left(\Pi(C)\right).
\end{equation}
Since the writhe is fixed under orientation-reversal,
so too is 
\begin{equation}\label{OrTB}
\tb(-C) = \tb(C)\,.
\end{equation}

The Thurston-Bennequin invariant is one of a pair of classical
Legendrian invariants.  To state the Main Theorem, we also need the other.

Because $C$ is determined by its Lagrangian projection $\Pi(C)$, any
isotopy invariant of immersed plane curves yields a 
Legendrian invariant of $C$.  According to the Whitney-Graustein
Theorem \cite{Whitney:1937}, the unique such invariant of an immersion
${\gamma:S^1\to \BR^2}$ is the rotation number 
\begin{equation}\label{RotGam}
\rot(\gamma) \,=\, \deg\dot{\gamma}\,,\qquad\qquad \dot{\gamma}:S^1
  \to \BR^2 - \{0\}\,,
\end{equation}
defined as the topological degree of the derivative $\dot{\gamma}$.
Equivalently, $\rot(\gamma)$ is the total signed curvature (see
eg.~Exercise 12 in \S $1.5$ of \cite{doCarmo:76}) of the
immersed plane curve,  
\begin{equation}\label{LocRot}
\rot(\gamma) \,=\,
\frac{1}{2\pi}\oint_{S^1}\!\!d\theta\,\,\frac{\dot{\gamma}\times\ddot{\gamma}}{||\dot{\gamma}||^2}\,,
\end{equation}
where we use the shorthand `$\times$' for the scalar cross-product,
\begin{equation}
\dot{\gamma}\times \ddot{\gamma}(\theta) \,\equiv\, \dot{x}(\theta)\,
\ddot{y}(\theta) \,-\, \dot{y}(\theta)\,\ddot{x}(\theta)\,,\qquad\quad
\gamma(\theta) = \left(x(\theta),y(\theta)\right) \in \BR^2\,.
\end{equation}
As will be essential later, the formula for $\rot(\gamma)$ in
\eqref{LocRot} presents the rotation number as a local invariant, in
the sense of being the integral of a locally-defined geometric
quantity along the curve.  With our conventions,
${\rot(\gamma)=1}$ when $\gamma$ is a circle traversed in the
counterclockwise direction.

For the Legendrian knot ${C \subset \BR^3}$, we set 
\begin{equation}
\rot(C) \,:=\, \rot\!\left(\Pi(C)\right) \in\,\BZ\,.
\end{equation}
See Definition $3.5.12$ in \cite{Geiges:2008} for an intrinsically
three-dimensional characterization of $\rot(C)$.  In terms of the
diagram for $\Pi(C)$, the rotation number can be
computed as a signed count of upwards vertical tangencies, as in
Figure \ref{Rotat}.  For the Legendrian trefoil in Figure
\ref{LgTref}, two upwards vertical tangencies occur, but they do so
with opposite signs, so ${\rot(C) = 0}$.

Note that the rotation number depends upon the orientation of the
curve, and under a reversal of orientation, the rotation number
changes sign. So in contrast to the behavior \eqref{OrTB} of the
Thurston-Bennequin invariant,
\begin{equation}
\rot(-C) \,=\, -\rot(C)\,.
\end{equation}
We have not specified an orientation for the trefoil knot in
Figure \ref{LgTref}, but because ${\rot(C) =  0}$, the orientation 
does not matter.

\begin{figure}[t]
$$\begin{matrix}
&\includegraphics[scale=0.50]{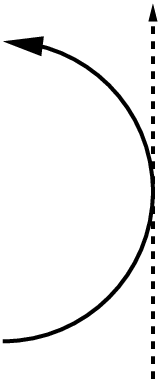}\qquad\quad & \quad\qquad 
&\includegraphics[scale=0.50]{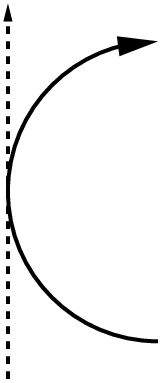}\\
&(a.)\quad\rot=+1 \qquad\quad & \quad\qquad &(b.)\quad\rot=-1
\end{matrix}$$
\caption{Rotation number at upwards tangencies.}\label{Rotat}
\end{figure}

\subsection{Main theorem}\label{Mainthm}

As its second consequence, supersymmetry modifies the angular form
${\psi\in\Omega^2\big(\BR^3-\{0\}\big)}$ which enters the elementary
Gauss linking integral \eqref{Gauss}.  Because the Legendrian condition
on ${C \subset \BR^3}$ does not respect the Euclidean action by
$SO(3)$, we forgo the seemingly-natural requirement that $\psi$ itself be
$SO(3)$-invariant.  In recompense, the supersymmetric version of
$\psi$ will enjoy superior analytic behavior near the origin in
$\BR^3$.

As heuristic motivation for the following, one should imagine that we
alter the angular form $\psi$ in \eqref{Angular} by
concentrating the support for the generator of 
$H^2(S^2;\BZ)$ at the North Pole on the sphere.  This trick
is well-known to aficionados of Chern-Simons perturbation theory, but here
we must take care to preserve the underlying symmetries of the contact
structure on $\BR^3$.

More precisely, we introduce the Gaussian two-form ${\omega_\Lambda \in
  \Omega^2\big(\BR^2\big)}$ on the $xy$-plane,
\begin{equation}\label{OmLamb}
\omega_\Lambda \,=\, \frac{\Lambda}{2\pi}\,
\e{\!-\Lambda (x^2 +
      y^2)/2}\,dx\^dy\,,\qquad\qquad \Lambda > 0\,.
\end{equation}
Here $\Lambda$ is a positive real parameter which sets the width
of the Gaussian, and in the limit ${\Lambda\to\infty}$, the Gaussian
becomes a delta-function concentrated at the origin in $\BR^2$.
Clearly $\omega_\Lambda$ is invariant under rotations, and
$\omega_\Lambda$ is normalized so that for all $\Lambda$,
\begin{equation}\label{NormOm}
\int_{\BR^2} \omega_\Lambda \,=\, 1\,.
\end{equation}
The various factors of two in \eqref{OmLamb} are standard and could be
absorbed into $\Lambda$ if desired.

Though the support of $\omega_\Lambda$ is not compact,
$\omega_\Lambda$ decays very rapidly at infinity.  At least morally,
$\omega_\Lambda$ should be regarded as a generator for the
compactly-supported cohomology ${H^2_c\big(\BR^2;\BZ\big)\simeq\BZ}$ of the
plane, on the same footing as the unit-area form on the sphere.  
Because the normalization condition in
\eqref{NormOm} does not depend on $\Lambda$, neither does the
cohomology class of $\omega_\Lambda$.  Explicitly, a small computation shows
\begin{equation}\label{DelOm}
\frac{\partial\omega_\Lambda}{\partial\Lambda} \,=\, \frac{1}{2\pi}\left[1 \,-\,
  \frac{\Lambda (x^2+y^2)}{2}\right] \e{\!-\Lambda\,(x^2 +
      y^2)/2}\,dx\^dy \,=\, d\alpha_\Lambda\,,
\end{equation}
where
\begin{equation}\label{TransA}
\alpha_\Lambda \,=\, \frac{1}{4\pi}\,\e{\!-\Lambda (x^2 +
      y^2)/2}\left(x\,dy - y\,dx\right) \,\in\,\Omega^1\big(\BR^2\big).
\end{equation}
The transgression form $\alpha_\Lambda$ will reappear in the proof of our
Fundamental Lemma.  In the meantime, note that $\alpha_\Lambda$ is also
$SO(2)$-invariant, as required by the relation to $\omega_\Lambda$ in
\eqref{DelOm}.

We next introduce the planar analogue for the retraction onto
$S^2$ in \eqref{Retrac}.  To preserve the parabolic scaling in
\eqref{Parabol}, we consider the map 
${\varrho_+:\BR^3_+\to\BR^2}$ defined on the upper half-space
$\BR^3_+$ by 
\begin{equation}\label{Parabolic}
\varrho_+(x,y,z) \,=\,
\left(\frac{x}{\sqrt{z}},\, \frac{y}{\sqrt{z}}\right),\qquad\qquad z>0\,.
\end{equation}
Trivially, the image of $\varrho_+$ is preserved under the scaling for which
$x$ and $y$ have weight one and $z$ has weight two.

Using the planar retraction in \eqref{Parabolic}, we pull the Gaussian form
${\omega_\Lambda\in\Omega^2(\BR^2)}$ back to a new two-form 
\begin{equation}\label{PullChi}
\begin{aligned}
\chi_\Lambda \,&=\, \varrho_+^*\omega_\Lambda\,,\\
&=\, \frac{\Lambda}{2\pi z} \e{\!-\Lambda (x^2 +
      y^2)/2 z} \left[dx\^dy \,+\,  \ha \left(x \,dy -
        y\,dx\right)\!\^\frac{dz}{z}\right],\qquad z>0\,.
\end{aligned}
\end{equation}
By construction, $\chi_\Lambda$ is invariant under the parabolic scaling
and the action of $SO(2)$, but not $SO(3)$.  

Of course, $\chi_\Lambda$ strongly resembles the heat kernel for the
Laplacian in two dimensions.  As with the heat kernel, so long as
${x^2+y^2 \neq 0}$, the expression in \eqref{PullChi}  vanishes
smoothly as ${z\to 0}$ from above.  To define $\chi_\Lambda$ on the
entire punctured space ${\BR^3-\{0\}}$, we simply extend by zero,
\begin{equation}
\chi_\Lambda \,=\, 0 \,,\qquad\qquad z \le 0\,.
\end{equation}
With this choice, ${\chi_\Lambda\in\Omega^2(\BR^3-\{0\})}$ is
automatically closed away from $\{0\}$ and generates the
cohomology ${H^2\big(\BR^3-\{0\};\BZ\big)}$.  For instance, over the unit
sphere ${S^2\subset \BR^3}$,
\begin{equation}\label{UnitNorm}
\int_{S^2}\chi_\Lambda \,=\,
\int_{S^2\cap\BR^3_+}\mskip-10mu\chi_\Lambda \,=\,
\int_{\BR^2}\omega_\Lambda \,=\, 1\,.
\end{equation}

We have yet to incorporate the Heisenberg symmetry of the contact
structure.  In the elementary Gauss linking integral, the abelian structure
of $\BR^3$ as a vector space enters implicitly through the definition
of the difference map $\Gamma$ in \eqref{DiffG}.  To preserve instead
the non-abelian symmetry by left-translation in ${\BH\simeq\BR^3}$, we consider a
Heisenberg difference map ${\widehat\Gamma:\BH\times\BH\to\BH}$, given
by 
\begin{equation}\label{HeisG}
\begin{aligned}
\widehat\Gamma(X_1,X_2) \,&=\, \mu(X_1^{-1},\,X_2) \,=\,
\mu(-X_1,X_2)\,,\qquad\qquad X_{1,2}\,\in\,\BH\,,\\
&=\, \left(x_2-x_1,\,y_2-y_1,\,z_2 - z_1 + x_1 y_2 - x_2 y_1\right).
\end{aligned}
\end{equation}
Since $\mu$ is the Heisenberg multiplication, ${\widehat\Gamma(X_1,X_2)
= X_1^{-1}\cdot X_2}$ in the usual shorthand.  This combination of
$X_1$ and $X_2$ is invariant under simultaneous left-multiplication,
\begin{equation}
\widehat\Gamma(g\cdot X_1,g\cdot X_2) =
\widehat\Gamma(X_1,X_2)\,,\qquad\qquad g\,\in\,\BH\,,
\end{equation}
and we have selected the relative signs of $X_1$ and $X_2$ in
\eqref{HeisG} to agree with the convention for the abelian difference
in \eqref{DiffG}.

\begin{prop}[Heisenberg Linking]\label{HLink}
Suppose ${C_i\subset \BR^3}$ for ${i=1,2}$ are disjoint oriented curves, not
necessarily Legendrian, with
respective parametrizations ${X_i:S^1\to\BR^3}$.  Then
\begin{equation}
\lk(C_1,C_2) \,=\, \int_{T^2} \left(X_1\times
  X_2\right)^*\!\widehat\Gamma^*\chi_\Lambda\,,\qquad\qquad \Lambda>0\,.
\end{equation}
\end{prop}\noindent
This proposition follows from the fact that the heat form
$\chi_\Lambda$ is equivalent in cohomology to the global angular form $\psi$,
\begin{equation}
[\chi_\Lambda]\,=\,[\psi] \,\in\, H^2\big(\BR^3-\{0\};\BZ\big)\,.
\end{equation}
Also, the Heisenberg difference $\widehat\Gamma$ in \eqref{HeisG} is
homotopic to the abelian difference $\Gamma$.  To see this, set
\begin{equation}\label{Mut}
\mu_\hbar(X_1,X_2) \,=\, \big(x_1 + x_2,\, y_1 + y_2,\, z_1 + z_2 -
  \hbar\left(x_1 y_2 - x_2 y_1\right)\big)\,,\qquad \hbar\,\in\,[0,1]\,.
\end{equation}
Then 
\begin{equation}
\widehat\Gamma_\hbar(X_1,X_2) \,=\,
\mu_\hbar\!\left(X_1^{-1},\,X_2\right) \,=\,
\big(x_2-x_1,\,y_2-y_1,\,z_2 - z_1 + \hbar\,(x_1 y_2 - x_2
    y_1)\big)
\end{equation}
smoothly interpolates from the abelian to the Heisenberg difference
as the Planck constant $\hbar$ ranges over the interval from
${\hbar=0}$ to ${\hbar=1}$.\qquad$\square$

Though the angular form $\psi$ and the heat form $\chi_\Lambda$ agree
in cohomology, they behave very differently near the origin in
$\BR^3$.  This analytic distinction matters crucially for
approaches to self-linking.  

Let ${C \subset \BR^3}$ be an oriented Legendrian
curve, with regular parametrization ${X:S^1\to\BR^3}$.  By analogy to the naive
Gauss self-linking integral in \eqref{SLKZero}, we consider a new
Heisenberg self-linking integral 
\begin{equation}\label{eq:1}
\slk_\kappa(C) \,:=\,  \underset{\varepsilon\to 0}{\lim}\int_{T^2 -
  \Delta(\varepsilon)} \left(X \times
  X\right)^*\!\widehat\Gamma_\hbar^*\,\chi_\Lambda\,,\qquad\quad
\hbar\,\in\,[0,1]\,,\qquad\quad \Lambda > 0\,.
\end{equation}
The remainder of the article is devoted to the proof of the
following Main Theorem.
\begin{thm}[Legendrian Self-Linking]\label{MT}
The limit defining $\slk_\kappa(C)$ exists.  The value of
$\slk_\kappa(C)$ is independent of $\Lambda$ and depends only upon the
Legendrian isotopy class of $C$.   In terms of the Thurston-Bennequin
invariant $\tb(C)$ and the rotation number $\rot(C)$,
\begin{equation}\label{eq:2}
\slk_\kappa(C)\,=\,
\begin{cases}
\quad\tb(C)-\rot(C)\,,\qquad &\hbar\,\neq\,1\,,\\
\quad\tb(C)\,,\qquad &\hbar\,=\,1\,.
\end{cases}
\end{equation}
\end{thm}

Informally, the Main Theorem states that the framing anomaly for knots
in bosonic Chern-Simons theory is absent in supersymmetric
Chern-Simons theory.  The corresponding statement for Seifert-fibered
three-manifolds was observed previously in \cite{Beasley:2005vf}.  The
Main Theorem is also consistent with results of Fuchs and Tabachnikov
\cite{Fuchs:99a} identifying $\tb(C)$ and $\rot(C)$ as the only order
${\le 1}$ finite-type invariants of Legendrian knots.

The strategy of proof for Theorem \ref{MT} is straightforward.  We
first demonstrate that $\slk_\kappa(C)$ is independent of the
parameter $\Lambda$ which sets the width of the Gaussian in
$\chi_\Lambda$.  For any ${\varepsilon > 0}$, the derivative of 
$\slk_\kappa(C)$ with respect to $\Lambda$ is given in terms of a boundary
integral of the transgression form $\alpha_\Lambda$ in
\eqref{TransA}.  The content of our Fundamental Lemma in Section
\ref{Fundamental} is to demonstrate that the potentially anomalous
boundary contribution from $\alpha_\Lambda$ vanishes in the limit
${\varepsilon\to 0}$, due to the rapid decay of the heat kernel away
from the diagonal.

In Section \ref{MainThm} we directly evaluate the self-linking integral
in the limit ${\Lambda\to\infty}$.  In this regime, the
self-linking integrand is non-negligible only in the neighborhood of
points on $T^2$ which map under the product ${X \times X}$ to pairs of
points ${p,q\in C}$ that become coincident under the Lagrangian 
projection to the $xy$-plane, ie.~${\Pi(p)=\Pi(q)}$.  Such
points on $T^2$ correspond either to the preimage of 
crossings in $\Pi(C)$, or more trivially, to points on the
diagonal ${\Delta\subset T^2}$.  The local contribution from the
crossings leads to the appearance of the Thurston-Bennequin invariant
$\tb(C)$, while the local contribution from the
diagonal $\Delta$ leads to the appearance of the rotation number
$\rot(C)$ for ${\hbar \neq 1}$.  At the special value ${\hbar = 1}$,
the Heisenberg symmetry of the integrand is restored, and
the anomalous contribution from $\Delta$ vanishes.  

The Legendrian condition is used crucially throughout the 
analysis.  Invariance under Legendrian isotopy follows {\sl a
  postiori} from the formula in \eqref{eq:2}.\footnote{A direct computation
of the Legendrian variation $\delta\slk_\kappa(C)$, achieved for
$\delta\slk_0(C)$ in \eqref{VarSLK}, involves some formidable
differential algebra and did not appear practical to these authors.}

Let us emphasize one striking feature of Theorem \ref{MT}, which is
perhaps best appreciated after one works through the localization
computation in Section \ref{Local}.  Namely, since the coefficients of
$\tb(C)$ and $\rot(C)$ in \eqref{eq:2} are integers, so is the value of
$\slk_\kappa(C)$!  The coefficient of $\tb(C)$ is directly related to
our normalization condition \eqref{UnitNorm} on the heat form
$\chi_\Lambda$, required to recover the usual linking number in
Proposition \ref{HLink}.  Thus the coefficient of
$\tb(C)$ is fixed by fiat to unity.  In contrast, the coefficient of $\rot(C)$
is determined by a delicate calculation near the
diagonal ${\Delta\subset T^2}$, so its integrality for all $\hbar$ is
a non-trivial feature of the Legendrian self-linking integral.

From the physical perspective, integrality of $\slk_\kappa(C)$ is
necessary for gauge invariance as well as consistency with
standard lore about non-renormalization and the infrared behavior of
supersymmetric Yang-Mills-Chern-Simons theory.  See \S $2.3$ of
\cite{Beasley:2015} for a discussion of this statement.  However, on the
principle that no integer appears by chance, a simple
topological explanation for the integrality of ${\slk_\kappa(C)\in\BZ}$ would
be nice to have.  

For instance, the difference of classical
Legendrian invariants ${\tb(C)-\rot(C)}$ in \eqref{eq:2} occurs
naturally in contact topology as the transverse self-linking invariant
$\slk(C_+)$ of the canonical positive
transverse push-off $C_+$ of $C$ (Proposition $3.5.36$ in
\cite{Geiges:2008}).  The transverse self-linking invariant $\slk(C_+)$ can be
interpreted in terms of a relative Euler class on a Seifert surface for $C_+$,
so its appearance in Theorem \ref{MT} is surely no accident.

\subsection{Notation and conventions}

For the convenience of the reader, we summarize the notation and
conventions used in the rest of the paper.
\begin{itemize}
\item $\BR^3$ has Euclidean coordinates $(x,y,z)$ and is oriented by
  ${dx\^dy\^dz}$.  
\item The map ${\Pi:\BR^3\to\BR^2}$ is the projection onto the $xy$-plane.
\item ${\kappa = dz + x\,dy - y\,dx}$ is the standard
  radially-symmetric contact form, positive with respect to the
  orientation on $\BR^3$.
\item ${C\subset \BR^3}$ is an oriented Legendrian knot, with regular
  parametrization ${X:S^1\to\BR^3}$.
\item ${\theta \sim \theta+2\pi}$ is an angular coordinate on $S^1$,
  compatible under $X$ with the given orientation on $C$.  We
  abbreviate ${dX/d\theta \equiv \dot{X}(\theta)}$, and so on.
\item The torus ${T^2 = S^1 \times S^1}$ has angular coordinates
  $(\theta_1,\theta_2)$ and is oriented by ${d\theta_1\^d\theta_2}$.
  The diagonal ${\Delta \subset T^2}$ is the subset where 
  ${\theta_1=\theta_2}$.  
\item For ${\varepsilon>0}$, a tubular neighborhood
  $\Delta(\varepsilon)$ of the diagonal ${\Delta\subset T^2}$ is parametrized by
  ${\theta_1=\phi}$ and ${\theta_2=\phi+\eta}$ for
  ${|\eta|<\varepsilon}$.  The cylinder ${T^2-\Delta(\varepsilon)}$
  has boundary circles $S^1_\pm$ on which ${\eta=\pm\varepsilon}$, respectively.
\item ${\gamma:S^1\to\BR^2}$ is an immersed plane curve which is the
  Lagrangian projection of $C$, ie.~${\gamma=\Pi\circ X}$.  
\item We use the abbreviation `$\times$' for the scalar cross-product
  on the plane, as in 
\begin{equation*}
\gamma\times \dot{\gamma}(\theta) \,\equiv\, x(\theta)\,
\dot{y}(\theta) \,-\, y(\theta)\,\dot{x}(\theta)\,,\qquad\quad
\gamma(\theta) = \left(x(\theta),y(\theta)\right) \in \BR^2\,.
\end{equation*} 
\item The contact form $\kappa$ is left-invariant with respect to the
  Heisenberg multiplication 
\begin{equation*}
\begin{aligned}
&\mu:\BH\times\BH\to\BH\,,\qquad\qquad \BH\simeq \BR^3\,,\qquad\qquad
[\hbar=1]\\ 
&\mu\big(X_1,X_2\big) \,=\,
\left(x_1 + x_2,\, y_1 + y_2,\, z_1 + z_2 - x_1 y_2 + x_2
  y_1\right).
\end{aligned}
\end{equation*}
More generally, for other values of $\hbar$ set 
\begin{equation*}
\mu_\hbar(X_1,X_2) \,=\, \big(x_1 + x_2,\, y_1 + y_2,\, z_1 + z_2 -
  \hbar\left(x_1 y_2 - x_2 y_1\right)\big)\,.
\end{equation*}
\item The left-invariant Heisenberg difference
  ${\widehat\Gamma_\hbar:\BH\times\BH\to\BH}$ for
  ${\hbar\in[0,1]}$ is given by 
\begin{equation*}
\begin{aligned}
\widehat\Gamma_\hbar(X_1,X_2) \,&=\, \mu_\hbar(X_1^{-1},\,X_2)\,,\\
&=\, \left(x_2-x_1,\,y_2-y_1,\,z_2 - z_1 + \hbar\left(x_1 y_2 - x_2
    y_1\right)\right). 
\end{aligned}
\end{equation*}
\item The Gaussian fundamental class
  ${\omega_\Lambda\in\Omega^2\big(\BR^2\big)}$ of the $xy$-plane is 
  given by 
\begin{equation*}
\omega_\Lambda \,=\, \frac{\Lambda}{2\pi}\,
\e{\!-\Lambda (x^2 +
      y^2)/2}\,dx\^dy\,,\qquad\qquad \Lambda > 0\,.
\end{equation*}
\item The transgression form
  ${\alpha_\Lambda\in\Omega^1\big(\BR^2\big)}$ satisfies
  ${\partial\omega/\partial\Lambda = d\alpha_\Lambda}$, where
\begin{equation*}
\alpha_\Lambda \,=\, \frac{1}{4\pi}\,\e{\!-\Lambda (x^2 +
      y^2)/2}\left(x\,dy - y\,dx\right).
\end{equation*}
\item The planar retraction ${\varrho_+\!:\BR^3_+\to\BR^2}$ is defined
  on the upper half-space $\BR^3_+$ by 
\begin{equation*}
\varrho_+(x,y,z) \,=\,
\left(\frac{x}{\sqrt{z}},\, \frac{y}{\sqrt{z}}\right),\qquad\qquad z>0\,.
\end{equation*}
\item The heat form ${\chi_\Lambda\in\Omega^2\big(\BR^3-\{0\}\big)}$
  is the pullback 
\begin{equation*}
\chi_\Lambda \,=\, 
\begin{cases}
\qquad\varrho_+^*\omega_\Lambda,\qquad &z>0\,,\\
\qquad 0\,,\qquad &z \le 0\,.
\end{cases}
\end{equation*}
Explicitly,
\begin{equation*}
\varrho_+^*\omega_\Lambda \,=\, \frac{\Lambda}{2\pi z} \e{\!-\Lambda
  (x^2 + y^2)/2 z} \left[dx\^dy \,+\,  \ha \left(x \,dy -
    y\,dx\right)\!\^\frac{dz}{z}\right],\qquad z>0\,.
\end{equation*}
\end{itemize}

\section{Fundamental lemma}\label{Fundamental}

We first demonstrate that the value of the
Legendrian self-linking integral $\slk_\kappa(C)$ does not depend upon
the parameter ${\Lambda 
  > 0}$ which sets the width of the Gaussian in the heat form $\chi_\Lambda$.
\begin{lem}[Fundamental Lemma]\label{FundLM}  
The limit which defines the Legendrian self-linking integral
$\slk_\kappa(C)$ exists, 
\begin{equation}\label{SlkkapC}
\slk_\kappa(C) \,=\,  \underset{\varepsilon\to 0}{\lim}\int_{T^2 -
  \Delta(\varepsilon)} \left(X \times
  X\right)^*\!\widehat\Gamma^*_\hbar\,\chi_\Lambda\,,\qquad\quad\hbar\,\in\,[0,1]\,,\qquad\quad
\Lambda>0\,, 
\end{equation}
and the value of ${\slk_\kappa(C)\in\BR}$ is independent of the parameter
$\Lambda$.
\end{lem}\noindent
Precisely at the special value ${\hbar=1}$, the self-linking integrand
is Heisenberg-invariant.  For this reason, the analysis to prove Lemma \ref{FundLM} will differ slightly depending on whether ${\hbar \neq
  1}$ or ${\hbar = 1}$. 

Throughout, we make two extra topological assumptions about the
Legendrian knot $C$.  Both assumptions hold generically and help to
simplify the proofs.
\begin{enumerate}
\item[1.] The Lagrangian projection $\Pi(C)$ is an immersed plane
  curve $\gamma$ with only double-point singularities, as in
  Figure \ref{LgTref}.\smallskip
\item[2.] The height function $z(\theta)$ on ${C \subset \BR^3}$ is
  Morse, with isolated non-degenerate critical points.  That is,
  $\dot{z}$ vanishes only at isolated points ${p\in C}$, at which
  ${\ddot{z} \neq 0}$.  Because $C$ is Legendrian, the Morse condition
  on $z(\theta)$ is equivalent by \eqref{LegC} to the condition that
  the function ${\gamma\times\dot\gamma}$ vanish only at isolated
  points ${\theta\in S^1}$, at which ${\gamma\times\ddot{\gamma}\neq 0}$. 
\end{enumerate}
The first assumption is standard.  Otherwise, the Morse
condition is used only in the case ${\hbar \neq 1}$ and could possibly
be relaxed, at the cost of further effort.

Before embarking on the proof of our Fundamental Lemma, let us mention an
easy corollary which is handy for the gauge theory analysis in \S $5$
of \cite{Beasley:2015}.   To state the corollary, we require
additional notation.  Let ${t > 0}$ be a positive scaling parameter.
As in the remarks following Proposition \ref{HLink}, we consider a
version of the Heisenberg difference $\widehat\Gamma_{t\hbar}$ with
rescaled Planck constant ${t\hbar}$ for fixed ${\hbar\in[0,1]}$,
\begin{equation}\label{HatGam}
\widehat\Gamma_{t\hbar}(X_1,X_2) \,=\, \left(x_2-x_1,\,y_2-y_1,\,z_2 - z_1 +
  t\hbar\left(x_1 y_2 - x_2 y_1\right)\right).
\end{equation}
For each value of $t$, we associate the contact form 
\begin{equation}\label{Kapt}
\kappa_t \,=\, t^{-1/2}\,dz \,+\, t^{1/2} \left(x \, dy - y \,
  dx\right).
\end{equation}
The relative power of $t$ between the
two terms in \eqref{Kapt} ensures that the contact form is
left-invariant under the Heisenberg multiplication with ${\hbar=t}$ in
\eqref{Mut}.  The overall power of $t$
ensures that the contact condition ${\kappa_t\^d\kappa_t =
  2\,dx\^dy\^dz \neq 0}$ is satisfied trivially for all values ${t > 0}$.

Finally, suppose that ${C\subset \BR^3}$ is a Legendrian knot with
respect to the standard contact form $\kappa_{t=1}$.  Just as we
consider the family of isotopic contact forms in \eqref{Kapt}, we
would like to consider a family of isotopic knots ${C_t \subset
  \BR^3}$, each of which is Legendrian with respect to $\kappa_t$ for
the given value of $t$.  Such a family of knots is determined
if we simply require the Lagrangian projection of $C_t$ to coincide
with that of $C$,
\begin{equation}
\Pi(C_t) \,=\, \Pi(C)\,,\qquad\qquad t\,>\,0\,.
\end{equation}
Either by integrating the contact condition as in \eqref{ZTheta} or just by
scaling, the embedding map ${X_t\!:S^1\to\BR^3}$ for $C_t$ is then
related to the original embedding $X$ for $C$ via 
\begin{equation}\label{BigXt}
X_t(\theta) \,\equiv\left(x_t(\theta),\, y_t(\theta),\,
  z_t(\theta)\right) = 
\left(x(\theta),\,y(\theta),\, t\,z(\theta)\right).
\end{equation}
In particular, the abelian limit ${t\to 0}$ of the Heisenberg
structure corresponds to a limit in which $C_t$ flattens to a
curve in the $xy$-plane, and 
the Lagrangian projection $\Pi(C)$ is realized geometrically.

Given the family of curves $C_t$, we consider the {\sl three}-parameter
self-linking integral  
\begin{equation}\label{SlKtwo}
\slk_\kappa(C_t) \,=\,  \underset{\varepsilon\to 0}{\lim}\int_{T^2 -
  \Delta(\varepsilon)} \left(X^{}_t \times
  X^{}_t\right)^*\!\widehat\Gamma_{t\hbar}^*\,\chi_\Lambda\,,\qquad\quad
\hbar\in[0,1]\,,\qquad\quad t,\Lambda>0\,. 
\end{equation}
Precisely for ${\hbar=1}$, the self-linking integrand is
invariant under the Heisenberg symmetry with multiplication map $\mu_t$.

\begin{remark}[Scaling Identity]  For all ${t,\Lambda>0}$,
\begin{equation}\label{SlKtwos}
\left(X^{}_t \times
  X^{}_t\right)^*\!\widehat\Gamma_{t\hbar}^*\,\chi_\Lambda \,=\,
\left(X\times X\right)^*\!\widehat\Gamma_\hbar^*\,\chi_{\Lambda/t}\,.
\end{equation}
By the Scaling Identity, the behavior of the self-linking integrand  in
limit ${\Lambda\to\infty}$ with fixed $t$ is the same as the behavior
in the limit ${t\to 0}$ with fixed $\Lambda$.  Because
$C_t$ flattens to a plane curve in the latter limit, 
$\Lambda$ plays a similar role to the parameter of the same name in
\eqref{FLambda}.
\end{remark}
\noindent
\textbf{Proof of Scaling Identity}
\vskip2pt
\noindent
The $t$-dependence of the pullback is $\left(X^{}_t \times
  X^{}_t\right)^*\!\widehat\Gamma_{t\hbar}^*\,\chi_\Lambda$ is
actually very simple.  In terms of the finite differences
\begin{equation}
\Delta x \,=\, x_2-x_1\,,\qquad\quad \Delta y \,=\,
y_2-y_1\,,\qquad\quad \widehat\Delta z \,=\, z_2 - z_1 + \hbar\left(x_1 y_2 - x_2 y_1\right),
\end{equation}
the formula for the heat form $\chi_\Lambda$ in \eqref{PullChi} implies
\begin{equation}\label{PullChiD}
\begin{aligned}
&\left(X^{}_t \times
  X^{}_t\right)^*\!\widehat\Gamma_{t\hbar}^*\,\chi_\Lambda \,=\,\\
&\qquad\frac{\Lambda}{2\pi t \widehat\Delta z} \,\e{\!-\Lambda
  (\Delta x^2 + \Delta y^2)/2 t \widehat\Delta z} \left[d\Delta
  x\^d\Delta y \,+ \ha \left(\Delta x \,d\Delta y -
    \Delta y\,d\Delta x\right)\!\^\frac{d\widehat\Delta
    z}{\widehat\Delta z}\right],
\end{aligned}
\end{equation}
provided ${\widehat\Delta z > 0}$.  Otherwise, the pullback of
$\chi_\Lambda$ vanishes.  Evidently, $t$ just multiplies $\widehat\Delta
z$ in \eqref{PullChiD}, and all dependence on $t$ can be absorbed by
rescaling the Gaussian parameter $\Lambda$.\qquad$\square$
\begin{cor}
The value of $\slk_\kappa(C_t)$ in \eqref{SlKtwo} is independent of both $t$ and $\Lambda$.
\end{cor}\noindent
\textbf{Proof}
The corollary follows immediately from Lemma \ref{FundLM} and the
Scaling Identity in \eqref{SlKtwos}.\qquad$\square$

\subsection{Analysis near the diagonal}\label{LocalD}

The non-trivial aspect of our work concerns the local analysis of the
self-linking integrand in the vicinity of the diagonal ${\Delta
  \subset T^2}$.  In practice, this analysis amounts to the Taylor
expansion of expressions such as occur in \eqref{PullChiD}.  Rather than
scatter such expansions willy-nilly throughout the paper, we collect here
the basic ingredients to be used again and again.

Let $(\theta_1,\theta_2)$ be angular coordinates on $T^2$.  To
parametrize the tubular neighborhood ${\Delta(\varepsilon)\subset T^2}$ of the
diagonal, we set 
\begin{equation}\label{EtaM}
\theta_1 \,=\, \phi\,,\qquad\qquad \theta_2 \,=\, \phi+\eta\,.
\end{equation}
Here $\phi$ is an angular coordinate along the diagonal, and
$\Delta(\varepsilon)$ is the subset where ${|\eta|<\varepsilon}$.  Our
local expansions near the diagonal will then be Taylor expansions
in $\eta$, appropriate for the regime ${\varepsilon\ll 1}$.
We must be careful about orientations.  In terms of the
coordinates $(\phi,\eta)$, the orientation form on 
$T^2$ is given by ${d\theta_1\^d\theta_2 = d\phi\^d\eta}$.
Topologically, the configuration space 
${T^2-\Delta(\varepsilon)}$ is a cylinder with oriented boundary
circles 
\begin{equation}
S^1_\pm:\quad \eta \,=\, \pm\varepsilon \mod 2\pi\,.
\end{equation}
As shown in Figure \ref{Cylinder}, the boundary orientation of
$S^1_+$ is positive with respect to the direction of increasing $\phi$ and
the orientation of $S^1_{-}$ is negative, so
\begin{equation}
\partial\!\left(T^2-\Delta(\varepsilon)\right) \,=\, S^1_+ \,-\, S^1_-\,.
\end{equation}

\begin{figure}[t]
\begin{center}
\includegraphics[scale=0.40]{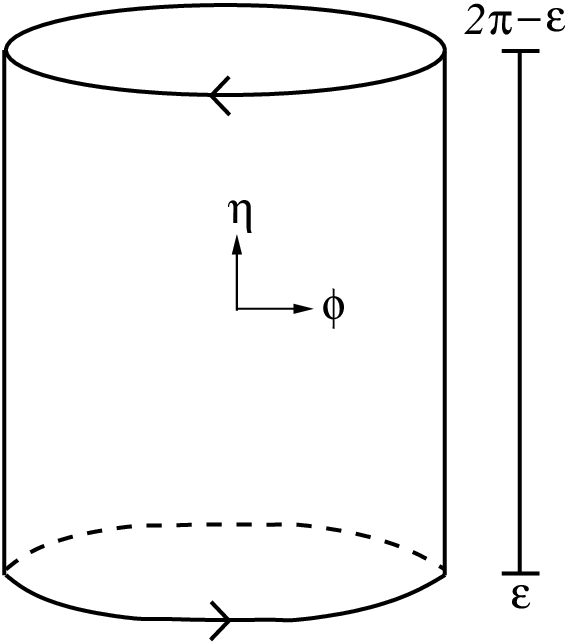}
\caption{Orientations on the cylinder ${T^2 -
    \Delta(\varepsilon)}$.}\label{Cylinder}
\end{center}
\end{figure}

We now expand the various terms appearing in the self-linking
integrand ${(X\times X)^*\widehat\Gamma_\hbar^*\,\chi_\Lambda}$, given by the
expression in \eqref{PullChiD} for ${t=1}$.  We start with 
\begin{equation}
\Delta x \,=\, x(\theta_2) - x(\theta_1) \,=\, x(\phi+\eta) - x(\phi) \,=\, \eta \, \dot{x}(\phi)
\,+\, \CO\big(\varepsilon^2\big)\,,
\end{equation} 
and similarly for $\Delta y$.  Hence 
\begin{equation}\label{Delrsq}
\Delta x^2 + \Delta y^2 \,=\, \eta^2 \left(\dot{x}^2 +
  \dot{y}^2\right) \,+\, \CO\big(\varepsilon^3\big) \,=\,
\eta^2\,||\dot\gamma||^2 
\,+\, \CO\big(\varepsilon^3\big)\,,
\end{equation}
where ${\gamma(\phi)=\left(x(\phi), y(\phi)\right)}$ is the
Lagrangian immersion ${\gamma = \Pi\circ X}$ as before.  
For the related two-form, we can write 
\begin{equation}
\begin{aligned}
d\Delta x\^d\Delta y \,&=\, \ha\,d\Delta\gamma\times
d\Delta\gamma\,,\qquad\qquad\qquad \Delta\gamma \equiv \left(\Delta x,
  \Delta y\right),\\
&=\, \ha \, d(\eta\,\dot{\gamma})\times d(\eta\,\dot{\gamma}) \,+\,
\CO\big(\varepsilon^3\big)\,,\\
&=\, \ha \left(\dot{\gamma}\,d\eta \,+\, \eta \,\ddot{\gamma}\,
  d\phi\right)\times \left(\dot{\gamma}\,d\eta \,+\, \eta \,\ddot{\gamma}\,
  d\phi\right) \,+\, \CO\big(\varepsilon^3\big).
\end{aligned}
\end{equation}
After collecting terms in the product,
\begin{equation}
d\Delta x\^d\Delta y \,=\, -\eta\,
\left(\dot{\gamma}\times\ddot{\gamma}\right) d\phi\^d\eta \,+\,
\CO\big(\varepsilon^2\big)\,d\phi\^d\eta\,.
\end{equation}

Recall the definition 
\begin{equation}
\widehat\Delta z \,=\, z_2 - z_1 + \hbar\left(x_1 y_2 - x_2 y_1\right).
\end{equation}
The Legendrian condition on $C$ is absolutely critical for our results,
because it implies that the local behavior of $\widehat\Delta z$ near the
diagonal is controlled by the geometry of the Lagrangian immersion $\gamma$.
Moreover, the quadratic terms in $\widehat\Delta z$ are required by
the Heisenberg symmetry when ${\hbar=1}$.

The expansion of the abelian difference ${\Delta z = z_2 - z_1}$ is
fixed by the Legendrian condition
\begin{equation}\label{LegCII}
\dot{z} \,=\, y \, \dot{x} \,-\, x \, \dot{y} \,=\, -\gamma\times
\dot{\gamma}\,.
\end{equation}
Thus
\begin{equation}\label{Delz}
\begin{aligned}
z_2 - z_1 \,&=\, z(\phi+\eta) - z(\phi) \,=\, \eta \,
\dot{z}(\phi) \,+\, \ha\,\eta^2\,\ddot{z}(\phi) \,+\,
\frac{1}{6}\,\eta^3\,\dddot{z}(\phi) \,+\,
\CO\big(\varepsilon^4\big)\,,\\
&=\, -\eta \left(\gamma\times\dot{\gamma}\right) \,-\,
\ha\,\eta^2\left(\gamma\times\ddot{\gamma}\right) \,-\,
\frac{1}{6}\,\eta^3\left(\dot\gamma\times\ddot{\gamma} \,+\,
  \gamma\times\dddot{\gamma}\right) \,+\, \CO\big(\varepsilon^4\big)\,.
\end{aligned}
\end{equation}
In passing to the second line of \eqref{Delz}, we repeatedly
differentiate the Legendrian condition on $\dot{z}$ in \eqref{LegCII}.

The attentive reader may wonder why we have expanded ${\Delta z=z_2 - z_1}$ all
the way to cubic order in $\eta$.  The question answers itself once we
expand the remaining quadratic terms in the Heisenberg difference
$\widehat\Delta z$,
\begin{equation}\label{DelzII}
\begin{aligned}
x_1 y_2 - x_2 y_1 \,&=\,
\gamma(\phi)\times\gamma(\phi+\eta)\,,\\
&=\, \eta \left(\gamma\times\dot{\gamma}\right) \,+\,
\ha\,\eta^2\left(\gamma\times\ddot{\gamma}\right) \,+\,
\frac{1}{6}\,\eta^3\left(\gamma\times\dddot{\gamma}\right) \,+\,
\CO\big(\varepsilon^4\big)\,.
\end{aligned}
\end{equation}
So long as ${\hbar\neq 1}$, the expansion of $\widehat\Delta z$ begins
at linear order in $\eta$, 
\begin{equation}\label{DelzOII}
\widehat\Delta z \,\buildrel{\hbar\neq 1}\over=\,
-\eta\left(1-\hbar\right)\left(\gamma\times\dot{\gamma}\right) \,+\,
\CO\big(\varepsilon^2\big)\,,
\end{equation}
as one naively expects.  But precisely at ${\hbar=1}$, cancellations
occur in the sum of \eqref{Delz} and \eqref{DelzII}, and the leading term in the expansion of
$\widehat\Delta z$ near the diagonal begins at cubic order,
\begin{equation}\label{DelzIII}
\widehat\Delta z \,\buildrel{\hbar=1}\over=\, -\frac{1}{6}\eta^3
\left(\dot{\gamma}\times\ddot{\gamma}\right) +\,
\CO\big(\varepsilon^4\big)\,.
\end{equation}
The cancellation in \eqref{DelzIII} is forced by the Heisenberg
symmetry at ${\hbar=1}$.  Clearly, the quantity
${\gamma\times\dot{\gamma}}$ in \eqref{DelzOII} is not invariant under
translations ${\gamma\mapsto\gamma+\gamma_0}$ for constant 
${\gamma_0\in\BR^2}$.  Upon projection to the $xy$-plane, such
translations are generated by the Heisenberg action, so
${\gamma\times\dot{\gamma}}$ is forbidden to appear at ${\hbar=1}$. 

Combining the expansions in \eqref{Delrsq}, \eqref{DelzOII}, and
\eqref{DelzIII}, we see that the argument of the heat kernel is given
in the neighborhood $\Delta(\varepsilon)$ by 
\begin{equation}\label{HeatArg}
\frac{\Delta x^2 + \Delta y^2}{2\widehat\Delta z} \,=\,\Bigg\{\quad
\begin{aligned}
&-\frac{||\dot{\gamma}||^2\,\eta}{2
  \left(1-\hbar\right)\left(\gamma\times\dot{\gamma}\right)} \,+\,
\CO\big(\varepsilon^2\big)\,,\qquad &\hbox{\small$[\hbar\neq 1]$}\\[1 ex]
&-\frac{3\,
  ||\dot{\gamma}||^2}{\left(\dot{\gamma}\times\ddot{\gamma}\right)\eta}
\,+\, \CO(1)\,. &\hbox{\small$[\hbar=1]$}
\end{aligned}
\end{equation}
For generic ${\hbar\neq 1}$, the argument of the heat kernel in
\eqref{HeatArg} vanishes linearly near the diagonal.  However,
at the symmetric point ${\hbar=1}$, the argument instead diverges
as ${\eta\to 0}$.  In Section \ref{MainThm}, this difference will 
ultimately lead to the discontinuity in the value of 
$\slk_\kappa(C)$ at ${\hbar = 1}$.

\begin{remark}[Local Positivity]
The pullback of the heat form $\chi_\Lambda$ vanishes identically
unless ${\widehat\Delta z > 0}$.  On the neighborhood
$\Delta(\varepsilon)$, positivity of $\widehat\Delta z$ becomes
equivalent via \eqref{DelzOII} and \eqref{DelzIII} to the local sign condition 
 \begin{equation}\label{Posit}
\widehat\Delta z\big|_{\Delta(\varepsilon)} > 0 \quad\Longleftrightarrow\quad\Bigg\{
\begin{aligned}
&\eta\left(1-\hbar\right)\left(\gamma\times\dot{\gamma}\right) <
0\,,\qquad &\hbox{\small$[\hbar\neq 1]$}\,,\\
&\eta\left(\dot{\gamma}\times\ddot{\gamma}\right) < 0\,.
&\hbox{\small$[\hbar=1]$}\,.
\end{aligned} 
\end{equation}
Again, the nature of the positivity condition depends upon
whether or not ${\hbar=1}$.  For generic values of $\hbar$, the sign
of $\eta$ is determined by the sign of $\gamma\times\dot{\gamma}$
and hence the sign of the derivative $\dot{z}$.  For
${\hbar=1}$, the sign of $\eta$ is instead fixed by the sign of
$\dot{\gamma}\times\ddot{\gamma}$, proportional to the plane
curvature of $\Pi(C)$.  

\begin{figure}[th]
$$\begin{matrix}
&\includegraphics[scale=0.60]{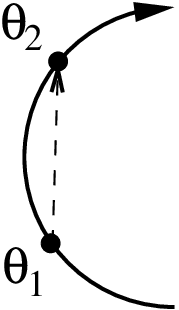}\quad\qquad & \quad\qquad 
&\includegraphics[scale=0.60]{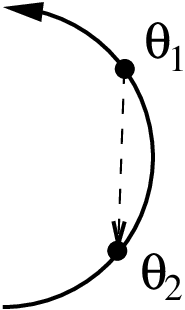}\\[1 ex]
&(a.)\quad\eta>0\,,\,\dot{\gamma}\times\ddot{\gamma}<0 \quad\qquad & \quad\qquad &(b.)\quad\eta<0\,,\,\dot{\gamma}\times\ddot{\gamma}>0
\end{matrix}$$
\caption{Local positivity condition ${\widehat\Delta z>0}$ for ${\hbar=1}$.}\label{Positv}
\end{figure} 

Since the local positivity condition depends upon the
sign of $\eta$, the self-linking integrand always vanishes in the 
neighborhood of one or the other of the boundary circles $S^1_\pm$ in Figure
\ref{Cylinder}, and the geometry of $C$ at any given point determines on
which boundary circle the integrand vanishes.  See Figure \ref{Positv}
for a geometric illustration of the local
positivity condition ${\widehat\Delta z > 0}$ in the symmetric case
${\hbar=1}$.  For clarity, we exaggerate the small separation between
the points $\gamma(\theta_1)$ and $\gamma(\theta_2)$ in the figure.
In both cases, the positivity condition is sensitive to the
orientation of $C$, as a reversal of orientation flips the signs
of $\gamma\times\dot{\gamma}$ and
$\dot{\gamma}\times\ddot{\gamma}$ in \eqref{Posit}.
\end{remark}

Let us complete the expansion for small $\eta$ of the self-linking
integrand.  The second bracketed term in \eqref{PullChiD} involves the angular
one-form 
\begin{equation}
\begin{aligned}
\Delta x \, d\Delta y \,-\, \Delta y \, d\Delta x \,&=\, \Delta\gamma
\times d\Delta\gamma\,,\\
&=\, \left(\eta\,\dot{\gamma}\right) \times
d\left(\eta\,\dot{\gamma}\right) + \CO\big(\varepsilon^3\big)\,,\\
&=\, \eta^2 \left(\dot{\gamma}\times\ddot{\gamma}\right) d\phi
\,+\, \CO\big(\varepsilon^3\big)\,.
\end{aligned}
\end{equation}
From the expansions of ${\widehat\Delta z}$ in \eqref{DelzOII} and
\eqref{DelzIII}, 
\begin{equation}
\left(\Delta x \, d\Delta y - \Delta y \, d\Delta
  x\right)\!\^d\widehat\Delta z=\Bigg\{
\begin{aligned}
&-\left(1-\hbar\right)\left(\gamma\times\dot{\gamma}\right)\left(\dot{\gamma}\times\ddot{\gamma}\right)
\eta^2\, d\phi\^d\eta\,,\quad&\hbox{\small$[\hbar\neq
1]$}\\
&-\ha\left(\dot{\gamma}\times\ddot{\gamma}\right)^2 \eta^4\, d\phi\^d\eta\,, &\hbox{\small$[\hbar=1]$}
\end{aligned}
\end{equation}
to leading order, so
\begin{equation}
\left(\Delta x \,d\Delta y -
    \Delta y\,d\Delta x\right)\!\^\frac{d\widehat\Delta
    z}{\widehat\Delta z}=\Bigg\{
\begin{aligned}
&\eta\left(\dot{\gamma}\times\ddot{\gamma}\right) d\phi\^d\eta\,+\, \CO\big(\varepsilon^2\big)\, d\phi\^d\eta\,,\qquad
&\hbox{\small$[\hbar\neq 1]$}\\
&3\,\eta\left(\dot{\gamma}\times\ddot{\gamma}\right) d\phi\^d\eta\,+\, \CO\big(\varepsilon^2\big)\, d\phi\^d\eta\,. &\hbox{\small$[\hbar=1]$}
\end{aligned}
\end{equation}

After a bit of algebra, one finds that the pullback
\eqref{PullChiD} of $\chi_\Lambda$ behaves near the diagonal ${\Delta
  \subset T^2}$ for ${\hbar\neq 1}$ as  
\begin{equation}\label{PullChiDIIO}
\begin{aligned}
&\left(X \times
  X\right)^*\!\widehat\Gamma_{\hbar}^*\,\chi_\Lambda\Big|_{\Delta(\varepsilon)}\buildrel{\hbar\neq
  1}\over=\\
&\qquad-\sgn(\eta)\,\frac{\Lambda\left(\dot{\gamma}\times\ddot{\gamma}\right)}{4\pi\,
  |1-\hbar|\,
  |\gamma\times\dot{\gamma}|}\,\exp{\!\left[-\frac{\Lambda\,
      ||\dot{\gamma}||^2\,|\eta|}{2\,|1-\hbar|\,
      |\gamma\times\dot{\gamma}|}\right]}\,d\phi\^d\eta\,+\, \cdots\,,
\end{aligned}
\end{equation}
assuming the sign condition
${\eta\left(1-\hbar\right)\left(\gamma\times\dot{\gamma}\right) < 0}$
in \eqref{Posit} holds.  Otherwise, the pullback is equal to zero.
The sign condition ensures that the exponential in \eqref{PullChiDIIO} is
always decaying, and it leads to a non-analytic dependence on the
sign of $\eta$.\footnote{As usual, ${\sgn(\eta)=+1}$ for ${\eta>0}$, and 
  ${\sgn(\eta)=-1}$ for ${\eta<0}$.}  The ellipses in
\eqref{PullChiDIIO} indicate subleading terms which vanish as ${\eta\to
  0}$.

By contrast, at the symmetric value ${\hbar=1}$,
\begin{equation}\label{PullChiDII}
\begin{aligned}
\left(X \times
  X\right)^*\!\widehat\Gamma^*_\hbar\,\chi_\Lambda\Big|_{\Delta(\varepsilon)}
\,&\buildrel{\hbar=1}\over=\,-\frac{3\,\Lambda}{2\pi
  \eta^2} \exp{\!\left[-\frac{3\,\Lambda\,||\dot{\gamma}||^2}{
      \left|\dot{\gamma}\times\ddot{\gamma}\right|}\frac{1}{|\eta|}\right]}\,d\phi\^d\eta
\,+\,\cdots\,,\\
&\,=\,  \frac{3\,\Lambda}{2\pi}
\exp{\!\left[-\frac{3\,\Lambda\,||\dot{\gamma}||^2}{\left|\dot{\gamma}\times\ddot{\gamma}\right|}
    \,|\nu|\right]}\,d\phi\^d\nu \,+\,\cdots\,,
\end{aligned}
\end{equation}
under the sign condition ${\eta
  \left(\dot{\gamma}\times\ddot{\gamma}\right)<0}$ in \eqref{Posit}.
In passing to the second line, we make the substitution ${\nu=1/\eta}$
for clarity.  In this case, the pullback of the heat form $\chi_\Lambda$ vanishes
exponentially as ${|\nu|\to\infty}$, or equivalently ${|\eta|\to 0}$.

\subsection{Proof of the fundamental lemma}\label{Proof}

The proof of our Fundamental Lemma \ref{FundLM} is now an
exercise in calculus.

As a brief formality, we first establish that the singularity in the
pullback of the heat form $\chi_\Lambda$ is integrable, 
so that the defining limit ${\varepsilon\to 0}$ in \eqref{SlkkapC} does
exist.  By assumption, ${||\dot{\gamma}||^2>0}$ is
everywhere non-vanishing, and the functions
$|\gamma\times\dot{\gamma}|$ and $|\dot{\gamma}\times\ddot{\gamma}|$
are bounded from above on $C$.  Integrability in the
region of small ${|\eta|<\varepsilon}$ follows immediately from the local
expressions in \eqref{PullChiDIIO} and \eqref{PullChiDII}, both of
which remain finite as ${\eta\to 0}$.

Otherwise, we must check that the value of $\slk_\kappa(C)$ does not
depend upon the parameter ${\Lambda > 0}$.  This argument will be
equally straightforward but the result is significant; the analogue for
the naive Gauss self-linking integral $\slk_0(C)$ is simply false.  
Our strategy will be to show that the derivative of $\slk_\kappa(C)$
with respect to $\Lambda$ vanishes for all values of ${\Lambda>0}$.
The details differ somewhat depending upon whether ${\hbar\neq
  1}$ or ${\hbar=1}$, but the main idea is the same in both cases.

We compute 
\begin{equation}\label{dslkL}
\begin{aligned}
\frac{d\slk_\kappa(C)}{d\Lambda} \,&=\, \underset{\varepsilon\to
  0}{\lim}\int_{T^2 - \Delta(\varepsilon)} \left(X \times
  X\right)^*\!\widehat\Gamma^*_\hbar\left(\frac{\partial\chi_\Lambda}{\partial\Lambda}\right)\,,\\
&=\,  \underset{\varepsilon\to
  0}{\lim}\int_{[T^2 - \Delta(\varepsilon)]_+} \left(X \times
  X\right)^*\!\widehat\Gamma^*_\hbar\,\varrho_+^*\!\left(\frac{\partial\omega_\Lambda}{\partial\Lambda}\right)\,.
\end{aligned}
\end{equation}
Here $[T^2 - \Delta(\varepsilon)]_+$ indicates the closed
subset of the cylinder where ${\widehat\Delta z \ge 0}$ is positive,
\begin{equation}\label{PosT}
[T^2 - \Delta(\varepsilon)]_+ \,=\, \Big\{
(\theta_1,\theta_2) \,\big|\, \widehat\Delta
z(\theta_1, \theta_2) \ge 0\Big\}\,,
\end{equation}
on which the pullback of the heat form $\chi_\Lambda$ is non-vanishing.
Of course, the positive subset in \eqref{PosT} depends upon the
Legendrian embedding $X$.  We omit this  dependence from the
notation as $X$ is fixed throughout.

By the calculation in \eqref{DelOm},
\begin{equation}
\frac{\partial\omega_\Lambda}{\partial\Lambda}\,=\,d\alpha_\Lambda\,, 
\end{equation}
for the
transgression one-form 
\begin{equation}
\alpha_\Lambda \,=\, \frac{1}{4\pi}\,\e{\!-\Lambda (x^2 +
      y^2)/2}\left(x\,dy - y\,dx\right) \,\in\, \Omega^1\big(\BR^2\big)\,.
\end{equation}
We apply the commutativity of the de Rham operator with
pullback, followed by Stokes' Theorem, to reduce the bulk integral in
the second line of \eqref{dslkL} to a boundary integral,
\begin{equation}\label{Stks}
\int_{[T^2 - \Delta(\varepsilon)]_+}\mskip-5mu \left(X \times
  X\right)^*\!\widehat\Gamma^*_\hbar\,\varrho_+^*\!\left(\frac{\partial\omega_\Lambda}{\partial\Lambda}\right)
\,=\, \int_{\partial[T^2 - \Delta(\varepsilon)]_+}\mskip-5mu \left(X \times
  X\right)^*\!\widehat\Gamma^*_\hbar\,\varrho_+^*\alpha_\Lambda\,.
\end{equation}
Explicitly, the boundary integrand in \eqref{Stks} is given, where
non-zero, by 
\begin{equation}\label{PullAlph}
 \left(X \times
  X\right)^*\!\widehat\Gamma^*_\hbar\,\varrho_+^*\alpha_\Lambda\,=\,
\frac{1}{4\pi\widehat\Delta z}\,\e{\!-\Lambda\left(\Delta x^2
      + \Delta y^2\right)/2\widehat\Delta z} \big(\Delta
  x\,d\Delta y \,-\, \Delta y \, d\Delta x\big)\,.
\end{equation}
This expression vanishes smoothly whenever
${\widehat\Delta z \to 0}$ from above with ${\Delta x^2 + 
  \Delta y^2 \neq 0}$.  

The boundary of the positive subset
${\partial[T^2 - \Delta(\varepsilon)]_+}$ in \eqref{Stks} includes
those curves where ${\widehat\Delta z = 0}$ as well as the
intersection of ${[T^2 - \Delta(\varepsilon)]_+}$ with the
boundary circles $S^1_\pm$ themselves.  Recall that points on
$S^1_\pm$ satisfy ${\eta=\pm\varepsilon}$, respectively.  By the
preceding, only the boundary integral over the intersection
${S^1_\pm\cap[T^2 - \Delta(\varepsilon)]_+}$ is relevant, because 
the boundary integrand in \eqref{PullAlph} vanishes on the locus where
${\widehat\Delta z = 0}$.

Altogether, in terms of the boundary integral on the right in \eqref{Stks},
\begin{equation}\label{dslkLII}
\frac{d\slk_\kappa(C)}{d\Lambda} \,=\,  \underset{\varepsilon\to
  0}{\lim}\left[\int_{S^1_+\cap[T^2 -
    \Delta(\varepsilon)]_+}\mskip-30mu\left(X \times
  X\right)^*\!\widehat\Gamma_\hbar^*\,\varrho_+^*\alpha_\Lambda \,-\, \int_{S^1_-\cap[T^2 -
    \Delta(\varepsilon)]_+}\mskip-30mu\left(X \times
  X\right)^*\!\widehat\Gamma_\hbar^*\,\varrho_+^*\alpha_\Lambda\right].
\end{equation}
The minus sign for the boundary integral over $S^1_{-}$ accounts for
the relative orientation in Figure \ref{Cylinder}.  

Despite the minus sign, no possibility exists for a trivial
cancellation between the two boundary integrals in \eqref{dslkLII} for
any fixed ${\varepsilon>0}$.  According to the local positivity
condition in \eqref{Posit} for respectively ${\hbar\neq 1}$ or ${\hbar=1}$, 
\begin{equation}\label{SSigns}
\begin{cases}
&\left(1-\hbar\right)\left(\gamma\times\dot{\gamma}\right) \hbox{ or }
\dot{\gamma}\times\ddot{\gamma} \le 0 \quad\hbox{on}\quad
S^1_+\cap[T^2 - \Delta(\varepsilon)]_+\,,\\
&\left(1-\hbar\right)\left(\gamma\times\dot{\gamma}\right) \hbox{ or }
\dot{\gamma}\times\ddot{\gamma} \ge 0 \quad\hbox{on}\quad
S^1_-\cap[T^2 - \Delta(\varepsilon)]_+\,.\\
\end{cases}
\end{equation}
The domains of integration over the two boundary circles $S^1_\pm$ in
\eqref{dslkLII} are therefore disjoint away from the degeneracy locus
where ${\left(1-\hbar\right)\left(\gamma\times\dot{\gamma}\right)}$
or ${\dot{\gamma}\times\ddot{\gamma} = 0}$, so no cancellation can
occur. Generically, the degeneracy locus consists of a finite set of
isolated inflection points on the curve.

Let us examine the behavior of the boundary integrand \eqref{PullAlph}
via the expansion near the diagonal from Section
\ref{LocalD}.  

\vspace{2pt}\noindent
\textbf{{\small Symmetric case ${\hbar=1}$}}
\vspace{2pt}

We initially consider the Heisenberg-symmetric case ${\hbar=1}$.
Similar to the bulk integrand in \eqref{PullChiDII}, the boundary
integrand behaves to leading-order at ${\eta = \pm\varepsilon}$ as 
\begin{equation}\label{PullChiDIII}
\left(X \times
  X\right)^*\!\widehat\Gamma^*_\hbar\,\varrho_+^*\alpha_\Lambda\Big|_{S^1_\pm}\,\buildrel{\hbar=1}\over=\,
\mp\frac{3}{2\pi\varepsilon}\,\exp{\!\left[-\frac{3\,\Lambda\,||\dot{\gamma}||^2}{\varepsilon 
    \left|\dot{\gamma}\times\ddot{\gamma}\right|}\right]}\,d\phi
\,+\, \cdots\,,
\end{equation}
where the omitted terms vanish more rapidly as ${\varepsilon\to 0}$.
By a conspiracy of signs, the difference on the
right of \eqref{dslkLII} can be rewritten as the
single integral
\begin{equation}\label{dslkLIV}
\frac{d\slk_\kappa(C)}{d\Lambda} \,\buildrel{\hbar=1}\over=\,  \lim_{\varepsilon\to
  0} \left[-\frac{3}{2\pi\varepsilon}\oint_{S^1}\!\!d\phi\,\exp{\!\left(-\frac{3\,\Lambda\,||\dot{\gamma}||^2}{\varepsilon 
    \left|\dot{\gamma}\times\ddot{\gamma}\right|}\right)}\right]=\,
0\,.
\end{equation}
To deduce the vanishing of the limit ${\varepsilon\to 0}$, we note
that the ratio ${||\dot{\gamma}||^2/|\dot{\gamma}\times\ddot{\gamma}| \ge m}$ is
everywhere bounded from below on $S^1$ by a positive constant ${m>0}$, so the
integrand in \eqref{dslkLIV} is dominated by the
exponentially-small constant 
\begin{equation}
\exp{\!\left(-\frac{3\,\Lambda\,||\dot{\gamma}||^2}{\varepsilon 
    \left|\dot{\gamma}\times\ddot{\gamma}\right|}\right)}
\,\le\,\exp{\!\left(-\frac{3\,\Lambda\,m}{\varepsilon}\right)}\,.
\end{equation}
Since ${\Lambda>0}$ has been arbitrary throughout, $\slk_\kappa(C)$ is
independent of $\Lambda$ for ${\hbar=1}$.\qquad$\square$

\vspace{2pt}\noindent
\textbf{{\small Generic case ${\hbar\neq 1}$}}
\vspace{2pt}

The analysis for generic ${\hbar\neq 1}$ is slightly more delicate.  Here
\begin{equation}\label{PullChiDIV}
\left(X \times
  X\right)^*\!\widehat\Gamma^*_\hbar\,\varrho_+^*\alpha_\Lambda\Big|_{S^1_\pm}\!\buildrel{\hbar\neq
  1}\over
=\,
\frac{\varepsilon \left(\dot{\gamma}\times\ddot{\gamma}\right)}{4\pi\,|1-\hbar|\,|\gamma\times\dot{\gamma}|}\,\exp{\!\left[-\frac{\Lambda\,||\dot{\gamma}||^2\,\varepsilon}{2\,|1-\hbar|\,|\gamma\times\dot{\gamma}|}\right]}\,d\phi
\,+\, \cdots\,,
\end{equation}
so the derivative becomes 
\begin{equation}
\frac{d\slk_\kappa(C)}{d\Lambda} \,\buildrel{\hbar\neq 1}\over=\,
\lim_{\varepsilon\to 0}\Big[{\rm I}_+(\varepsilon) \,-\, {\rm I}_-(\varepsilon)\Big],
\end{equation}
with
\begin{equation}\label{BigI}
{\rm I}_\pm(\varepsilon) \,=\, 
\frac{\varepsilon}{4\pi |1-\hbar|}\int_{S^1_\pm\cap[T^2 -
    \Delta(\varepsilon)]_+}\mskip-20mu
  d\phi\,\,\frac{\dot{\gamma}\times\ddot{\gamma}}{|\gamma\times\dot{\gamma}|}\,
  \exp{\!\left(-\frac{\Lambda\,||\dot{\gamma}||^2\,\varepsilon}{2\,|1-\hbar|\,|\gamma\times\dot{\gamma}|}\right)}\,.
\end{equation}
The functions ${\rm I}_\pm(\varepsilon)$ differ only in the domain of
integration over $S^1$, and our goal will be to show individually 
\begin{equation}\label{LimBigI}
\lim_{\varepsilon\to 0} {\rm I}_\pm(\varepsilon) \,=\, 0\,.
\end{equation}

Were the function $\gamma\times\dot{\gamma}$ to be everywhere non-zero on
$S^1$, the conclusion in \eqref{LimBigI} would be  
immediate, as we would know the integral in \eqref{BigI} to be bounded in
magnitude even for ${\varepsilon=0}$.  The explicit prefactor of $\varepsilon$
then ensures the vanishing of ${\rm I}_\pm(\varepsilon)$ in
the limit ${\varepsilon\to 0}$.  However, ${\dot{z} =
  -\gamma\times\dot{\gamma}}$ always vanishes for at least two points
(the highest and the lowest) on the knot ${C \subset \BR^3}$,
and we must worry about what happens to the integral in \eqref{BigI}
near a zero of ${\gamma\times\dot{\gamma}}$, when $\varepsilon$ is
very small.

Let us make an elementary simplification.  Since
$|\dot{\gamma}\times\ddot{\gamma}|$ is bounded from above and
${||\dot{\gamma}||^2 > 0}$ is bounded from below on $S^1$,
\begin{equation}
|{\rm I}_\pm(\varepsilon)| \,\le\, {\rm J}_\pm(\varepsilon) \,=\, \int_{S^1_\pm\cap[T^2 -
    \Delta(\varepsilon)]_+}\mskip-20mu
  d\phi\,\,\frac{ A\,\varepsilon}{|\gamma\times\dot{\gamma}|}\,
  \exp{\!\left(-\frac{B\,\varepsilon}{|\gamma\times\dot{\gamma}|}\right)}\,,\qquad A,B\,>\, 0\,, 
\end{equation}
for some positive constants $A$ and $B$, into which we also absorb the
dependence on $\Lambda$ and $\hbar$ and the various other numerical
factors in \eqref{BigI}.  To deduce the limit \eqref{LimBigI} for
${\rm I}_\pm(\varepsilon)$, we show that ${\rm
  J}_\pm(\varepsilon)$ vanishes in the same limit.

By assumption, the height function $z(\phi)$ is Morse, with isolated
non-degenerate critical points.  Equivalently, the function 
${(\gamma\times\dot{\gamma})(\phi)}$ vanishes non-degenerately at an
isolated set of points on $S^1$.  By the criteria in \eqref{SSigns},
these points are precisely the endpoints of the intervals which compose each 
integration domain ${S^1_\pm\cap [T^2-\Delta(\varepsilon)]_+}$.
Locally near such an endpoint ${\phi=\phi_0}$,
\begin{equation}
\left(\gamma\times\dot{\gamma}\right)(\phi) \,=\, c_0 \left(\phi - \phi_0\right) \,+\, \CO\big(|\phi-\phi_0|^2\big)\,,\qquad\qquad c_0
\neq 0\,.
\end{equation}

When we examine ${\rm J}_\pm(\varepsilon)$ in the limit ${\varepsilon\to 0}$, 
only the contribution to the integral from a (one-sided) neighborhood of $\phi_0$
can be non-zero, so we simplify further by replacing 
${\rm J}_\pm(\varepsilon)$ by the model 
\begin{equation}\label{BigK}
{\rm K}(\varepsilon) \,=\, \int_{\phi_0}^{\phi_1}\!\!d\phi\,\,\frac{ \varepsilon}{f(\phi)}\,
  \exp{\!\left(-\frac{\varepsilon}{f(\phi)}\right)}\,.
\end{equation}
Here $\phi_1$ is an arbitrary upper cutoff, and $f(\phi)$ is now any
continuous function defined on the interval $[\phi_0,\phi_1]$ such that 
\begin{figure}[th]
\centering
\def\svgwidth{50mm}
\begingroup%
  \makeatletter%
  \providecommand\color[2][]{%
    \errmessage{(Inkscape) Color is used for the text in Inkscape, but the package 'color.sty' is not loaded}%
    \renewcommand\color[2][]{}%
  }%
  \providecommand\transparent[1]{%
    \errmessage{(Inkscape) Transparency is used (non-zero) for the text in Inkscape, but the package 'transparent.sty' is not loaded}%
    \renewcommand\transparent[1]{}%
  }%
  \providecommand\rotatebox[2]{#2}%
  \ifx\svgwidth\undefined%
    \setlength{\unitlength}{273.05145404bp}%
    \ifx\svgscale\undefined%
      \relax%
    \else%
      \setlength{\unitlength}{\unitlength * \real{\svgscale}}%
    \fi%
  \else%
    \setlength{\unitlength}{\svgwidth}%
  \fi%
  \global\let\svgwidth\undefined%
  \global\let\svgscale\undefined%
  \makeatother%
  \begin{picture}(1,0.82561197)%
    \put(0,0){\includegraphics[width=\unitlength]{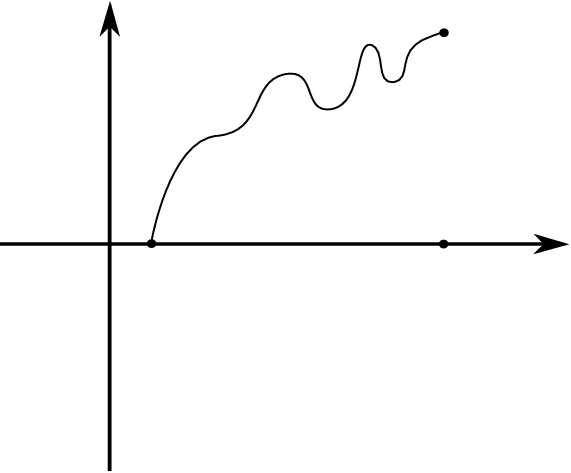}}%
    \put(0.73,0.33){\color[rgb]{0,0,0}\makebox(0,0)[lb]{\smash{$\phi_{1}$}}}%
    \put(0.37939734,0.73308087){\color[rgb]{0,0,0}\makebox(0,0)[lb]{\smash{$f(\phi)$}}}%
    \put(0.22,0.33){\color[rgb]{0,0,0}\makebox(0,0)[lb]{\smash{$\phi_{0}$}}}%
  \end{picture}%
\endgroup%
\caption{A function $f(\phi)$ satisfying the assumptions in
  Lemma \ref{CalcLM}.}\label{Function}
\end{figure}
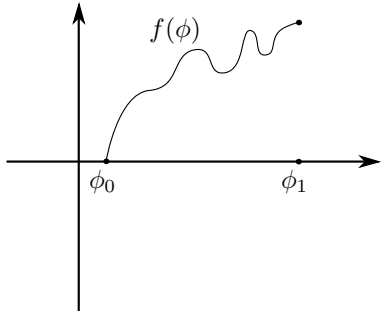 
\begin{equation}\label{Litf}
f(\phi) > 0 \hbox{ for } \phi > \phi_0\,,\qquad f(\phi_0)\,=\,
0\,,\quad \hbox{ and }\quad \lim_{\phi\to\phi_0}\left[\frac{(\phi - \phi_0)}{f(\phi)}\right] >
0 \, \hbox{ exists}\,.
\end{equation}
For all ${\varepsilon > 0}$, the integral defining ${\rm
  K}(\varepsilon)$ exists, since the integrand 
vanishes at the endpoint ${\phi=\phi_0}$.  For convenience, we take
${\phi_0=0}$ and ${\phi_1=1}$ by a suitable choice of parameter.   The
proof of the Fundamental Lemma \ref{FundLM} 
for generic ${\hbar\neq 1}$ reduces to the following claim.
\begin{lem}\label{CalcLM} 
Let ${\rm K}(\varepsilon)$ and $f(\phi)$ be defined as in
  \eqref{BigK} and \eqref{Litf}.  Then 
\begin{equation}
\lim_{\varepsilon\to 0} {\rm K}(\varepsilon) \,=\, \lim_{\varepsilon\to 0}\left[\int_0^1\!\!d\phi\,\,\frac{ \varepsilon}{f(\phi)}\,
  \exp{\!\left(-\frac{\varepsilon}{f(\phi)}\right)}\right] \,=\, 0\,.
\end{equation}
\end{lem}
\noindent
\textbf{Proof}
\vskip2pt
\noindent
We consider a succession of three cases.  

$^{\rm (i)}$\,We start with the
basic example ${f(\phi) = \phi}$, so that 
\begin{equation}
{\rm K}(\varepsilon) \,=\, \varepsilon \int_0^1 \frac{d\phi}{\phi}
\, \exp{\!\left(-\frac{\varepsilon}{\phi}\right)}\,.
\end{equation}
After the substitution ${x =
  \varepsilon/\phi}$,
\begin{equation}
{\rm K}(\varepsilon) \,=\, \varepsilon \int_{\varepsilon}^\infty
\frac{dx}{x}\,\e{-x} \,\le\, \varepsilon \int_{\varepsilon}^1
\frac{dx}{x} \,+\, \varepsilon \int_1^\infty dx\,\e{-x} \,=\,
\varepsilon \left|\ln\varepsilon\right| \,+\, \varepsilon \, \e{-1}\,,
\end{equation}
from which the limit follows.  

$^{\rm (ii)}$\,Next, let $g(\phi)$ and
$h(\phi)$ be continuous functions on the interval $[0,1]$ obeying bounds
\begin{equation}
0 \,<\, m \,\le\, g(\phi)\,,\qquad\qquad |h(\phi)|\,\le\, M\,,
\end{equation} 
for some constants $m$ and $M$.  Set 
\begin{equation}
{\rm K}(\varepsilon) \,=\, \varepsilon \int_0^1
\frac{d\phi}{\phi}\,\,h(\phi)\,\exp{\!\left[-\varepsilon\,\frac{g(\phi)}{\phi}\right]}\,.
\end{equation}
Then
\begin{equation}
{\rm K}(\varepsilon) \,\le\, M\,\varepsilon
\int_0^1\frac{d\phi}{\phi}\,\exp{\left(-\varepsilon\,\frac{m}{\phi}\right)}
\,=\,  M\,\varepsilon
\int_0^{1/m}\frac{d\phi}{\phi}\,\exp{\left(-\frac{\varepsilon}{\phi}\right)}\,.
\end{equation}
The function ${\rm K}(\varepsilon)$ vanishes as ${\varepsilon\to 0}$
by {\rm (i)}.

$^{\rm (iii)}$\,In the general case of interest,
\begin{equation}
{\rm K}(\varepsilon) \,=\, \int_{0}^{1}\!\!d\phi\,\,\frac{
  \varepsilon}{f(\phi)}\, 
  \exp{\!\left[-\frac{\varepsilon}{f(\phi)}\right]} \,=\, \varepsilon
  \int_0^1 \frac{d\phi}{\phi}\left(\frac{\phi}{f(\phi)}\right) \, \exp{\!\left[-\frac{\varepsilon}{\phi}\left(\frac{\phi}{f(\phi)}\right)\right]}\,. 
\end{equation}
Because ${f(\phi) > 0}$ for ${\phi > 0}$ by assumption, the function
${g(\phi) = h(\phi) = 
  \phi/f(\phi)}$ is continuous and
positive for all ${\phi > 0}$.  Since the limit ${\lim_{\phi\to
    0}\left[\phi/f(\phi)\right]>0}$ is also assumed to exist and be
non-zero, ${g(\phi) > 0}$ is continuous and non-vanishing throughout
the unit interval.  Hence ${0 < m \le g(\phi) \le M}$ for some
constants $m$ and $M$, and the general case follows from {\rm (ii)}.\qquad$\square$

\section{Planar limit}\label{MainThm}

According to the Fundamental Lemma \ref{FundLM}, the value of the
self-linking integral $\slk_\kappa(C)$ does not depend upon the
positive parameter ${\Lambda > 0}$ which sets the width of the
Gaussian in the heat form $\chi_\Lambda$.  To evaluate
$\slk_\kappa(C)$, and in the process to show that $\slk_\kappa(C)$ is
invariant under Legendrian isotopy, we now analyze the
self-linking integral \eqref{eq:1} in the limit ${\Lambda\to\infty}$.   The
Legendrian knot ${C\subset \BR^3}$ and its regular parametrization
${X:S^1\to\BR^3}$ remain fixed throughout.

The limit ${\Lambda\to\infty}$ has several interpretations. 
 
In terms of the heat kernel, this limit is the short-time limit, in which the
Gaussian generator $\omega_\Lambda$ for the compactly-supported cohomology
$H^2_c(\BR^2;\BZ)$ concentrates to a form with delta-function
support at the origin. More geometrically, by the Scaling Identity in
\eqref{SlKtwos}, the limit ${\Lambda\to\infty}$ is equivalent to the limit ${t\to 0}$ in
which the contact planes represented by $\kappa_t$ in
\eqref{Kapt} and the Legendrian knot $C_t$ in \eqref{BigXt} flatten to the
$xy$-plane.  Simultaneously, the Planck constant 
$\hbar$ in the Heisenberg multiplication scales to zero, and the abelian
structure of $\BR^3$ is restored.  For this reason, we refer to the
limit ${\Lambda\to\infty}$ as the planar limit.

In the planar limit, the Legendrian self-linking integral simplifies
immensely, as can be understood from the formula for the integrand
\begin{equation}\label{PullChiDD}
\begin{aligned}
&\left(X^{} \times
  X^{}\right)^*\!\widehat\Gamma_{\hbar}^*\,\chi_\Lambda \,=\, \frac{\Lambda}{2\pi\widehat\Delta z} \,\e{\!-\Lambda
  (\Delta x^2 + \Delta y^2)/2 \widehat\Delta z}\,\,\times\,\\
&\qquad\qquad\times\left[d\Delta
  x\^d\Delta y \,+ \ha \left(\Delta x \,d\Delta y -
    \Delta y\,d\Delta x\right)\!\^\frac{d\widehat\Delta
    z}{\widehat\Delta z}\right],\qquad\qquad \widehat\Delta z > 0\,.
\end{aligned}
\end{equation}
Intuitively, the behavior of the integrand is controlled by the
exponential factor in the first line of \eqref{PullChiDD}.  When
$\Lambda$ is sufficiently large, the integrand is negligible away from
the locus where 
\begin{equation}\label{AsymPT}
\left[\Delta x^2 + \Delta y^2\right]\!\Big|_{(\theta_1,\theta_2)}
\,\ll\, \frac{1}{\Lambda}\,,\qquad\qquad (\theta_1,\theta_2) \,\in\, T^2-\Delta(\varepsilon)\,.
\end{equation}
Because $\Delta x$ and $\Delta y$ are given by the differences
\begin{equation}
\Delta x \,=\, x(\theta_2) -
x(\theta_1)\,,\qquad\qquad \Delta y \,=\,
y(\theta_2) - y(\theta_1)\,,
\end{equation}
the asymptotic condition in \eqref{AsymPT} means that the pair 
${\theta_1, \theta_2}$ map under the embedding ${X:S^1\to\BR^3}$ to
points ${p,q\in C}$ which are nearly coincident under the
Lagrangian projection to the $xy$-plane.  Thus the point 
$(\theta_1,\theta_2)$ either lies near the preimage of a crossing
(aka double-point) on
the Lagrangian projection $\Pi(C)$, or
$(\theta_1,\theta_2)$ lies near the diagonal $\Delta$ itself, in the
boundary region that we previously analyzed in Section \ref{LocalD}.

With this observation, our proof of the Main Theorem proceeds in three steps.
\begin{enumerate}
\item[1.] Estimate the contribution to $\slk_\kappa(C)$ from
  each crossing of $\Pi(C)$ when $\Lambda$ is large.\vspace{2pt}
\item[2.] Estimate the contribution to $\slk_\kappa(C)$ from
  the diagonal ${\Delta\subset T^2}$ when
  $\Lambda$ is large.\vspace{2pt}
\item[3.] Bound the contributions from elsewhere on the integration domain, as well as the errors in the preceding
  estimates, by a quantity $\delta$ which can be made arbitrarily
  small as ${\Lambda\to\infty}$.
\end{enumerate}
Conceptually, the local estimates in the first two steps are most
important, because these estimates explain why the Thurston-Bennequin
invariant $\tb(C)$ and the rotation number $\rot(C)$ appear in the
formula \eqref{eq:2} for $\slk_\kappa(C)$.  We therefore 
begin in Section \ref{Local} with simple, informal computations for
the first two steps in the proof.

The technical heart of the proof resides in the third step, when we
carefully bound the errors in the preceding local computations.  This
step is required for a rigorous analysis, but the ideas are standard
and offer no surprises.  For this reason, Sections \ref{Preps} and
\ref{Bounds} could be omitted on the initial reading of the paper.   In Section
\ref{Preps} we introduce various geometric quantities to be used in
the error analysis, and in Section \ref{Bounds} we make the
necessary bounds.

\subsection{Local computations}\label{Local}

We compute the contribution to $\slk_\kappa(C)$ from a
right-handed crossing in the Lagrangian projection.  We depict such a
crossing on the left in Figure \ref{LocalOver}.  We
shall proceed softly, reserving precise inequalities for Section \ref{Bounds}.

\begin{figure}[h]
$$\begin{matrix}
&\includegraphics[scale=0.50]{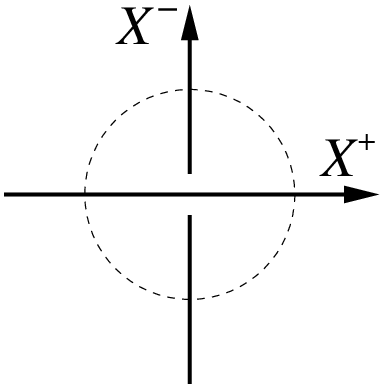}\qquad & \qquad 
&\includegraphics[scale=0.50]{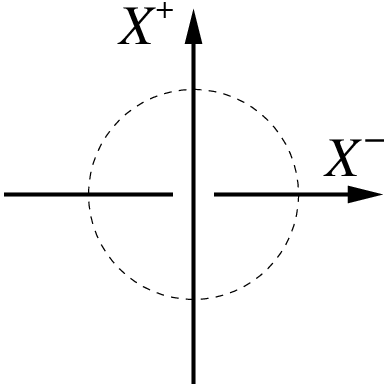}\\
&(a.)\quad\slk_\kappa=+1 \qquad & \qquad &(b.)\quad\slk_\kappa=-1
\end{matrix}$$
\caption{Neighborhoods of right- and left-handed crossings of
  $\Pi(C)$.}\label{LocalOver} 
\end{figure}
  
To first-order over the double-point, the curve $C$ is
approximated by a pair of straight lines.  For our local
computation, we take the lines to be parametrized by maps
${X^\pm\!:\BR\to\BR^3}$, where $X^+$ passes over $X^-$ by convention.  As in the
figure, we take $X^\pm$ to describe lines which are perpendicular and
lie in parallel planes,
\begin{equation}\label{Lines}
X^-(\theta_1) \,=\, \left(0,\theta_1,0\right),\qquad\qquad 
X^+(\theta_2) \,=\, \left(\theta_2,0, \Delta z\right),\qquad\qquad
\theta_{1,2}\in\BR\,.
\end{equation}
Here ${\Delta z > 0}$ is a positive constant, the height of the
overpass.  With this choice, both $X^\pm$ satisfy
the Legendrian condition \eqref{LegC} and so describe a 
Legendrian crossing.  Because we have yet to establish
isotopy-invariance of any kind, our assumptions about even the 
first-order geometry of $C$ require justification.  A significant
portion of the analysis in Section \ref{Bounds} will be devoted exactly to this issue. 

The local contribution to $\slk_\kappa(C)$ from the right-handed crossing at
${\{0\}\in\BR^2}$ is now given by 
\begin{equation}\label{LocO}
\slk_\kappa(C)\big|_{\{0\}} \,=\,
\int_{\BR^2} \left(X^-\!\times
  X^+\right)^*\!\widehat\Gamma_\hbar^*\,\chi_\Lambda\,,\qquad\qquad
\Lambda \gg 1\,,
\end{equation}
where we integrate over all ${(\theta_1,\theta_2)\in\BR^2}$, with the
standard orientation ${d\theta_1\^d\theta_2}$.  Informally, the error
which we make when we extend the range 
of integration from a small region on $T^2$
to the entire plane $\BR^2$ vanishes exponentially as
${\Lambda\to\infty}$, due to the rapid decay of the heat form
$\chi_\Lambda$ away from the origin.  By extending over all of
$\BR^2$, we will be able to evaluate the integral \eqref{LocO} in
closed form.

For the perpendicular lines $X^\pm$ in \eqref{Lines}, the differences
$\Delta x$, $\Delta y$, and $\widehat\Delta z$ in \eqref{PullChiDD} become 
\begin{equation}
\Delta x \,=\, x^+ - x^- \,=\,
\theta_2\,,\qquad\qquad \Delta y \,=\, y^+ - y^-
\,=\, -\theta_1\,, 
\end{equation}
and 
\begin{equation}
\widehat\Delta z \,=\, z^+ - z^- +
\hbar\left(x^- y^+ - x^+ y^-\right) \,=\, \Delta z - \hbar \, \theta_1
\theta_2\,.
\end{equation}
After a small calculation, one finds for the self-linking integrand in
\eqref{PullChiDD} 
\begin{equation}
\left(X^-\!\times
  X^+\right)^*\!\widehat\Gamma_{\hbar}^*\,\chi_\Lambda \,=\,
\frac{\Lambda}{2\pi \Delta z \left(1 \,-\, \hbar\,\theta_1
    \theta_2/\Delta z\right)^2} \, \exp{\!\left[-\frac{\Lambda
      \left(\theta_1^2 + \theta_2^2\right)}{2 \left(\Delta z - \hbar\,
        \theta_1 \theta_2\right)}\right]}\,d\theta_1\^d\theta_2\,,
\end{equation}
assuming the positivity condition ${\widehat\Delta z > 0\,
  \Longleftrightarrow\, \Delta z > \hbar \, \theta_1
  \theta_2}$ (else the integrand vanishes).  

Thus,
\begin{equation}\label{slkC1}
\slk_\kappa(C)\big|_{\{0\}} \,=\,
\int_{\Delta z > \hbar \, \theta_1 \theta_2}\mskip-15mu d\theta_1 d\theta_2 \,\,\frac{\Lambda}{2\pi \Delta z \left(1 \,-\, \hbar\,\theta_1
    \theta_2/\Delta z\right)^2} \, \exp{\!\left[-\frac{\Lambda
      \left(\theta_1^2 + \theta_2^2\right)}{2 \left(\Delta z - \hbar\,
        \theta_1 \theta_2\right)}\right]}\,.
\end{equation}
Since ${\Lambda \gg 1}$ is large, let us rescale the integration
variables to eliminate the overall factor of $\Lambda$ from the argument
of the exponential, 
\begin{equation}\label{slkC0}
\begin{aligned}
&\slk_\kappa(C)\big|_{\{0\}} \,=\,\\
&\qquad\int_{\Lambda\Delta z > \hbar \, \theta_1 \theta_2}\mskip-15mu d\theta_1 d\theta_2 \,\,\frac{1}{2\pi \Delta z \left(1 \,-\, \hbar\,\theta_1
    \theta_2/\Lambda\Delta z\right)^2} \, \exp{\!\left[-\frac{
      \theta_1^2 + \theta_2^2}{2 \,\Delta z \left(1 - \hbar\,
        \theta_1 \theta_2/\Lambda\Delta z\right)}\right]}\,.
\end{aligned}
\end{equation}
After we expand the integrand of \eqref{slkC0} asymptotically in
$1/\Lambda$, the local contribution to $\slk_\kappa(C)$ from the
right-handed crossing can be evaluated as a Gaussian integral over $\BR^2$,
\begin{equation}\label{GausR2}
\begin{aligned}
\slk_\kappa(C)\big|_{\{0\}} \,&=\,
\int_{\BR^2}\!\! d\theta_1 d\theta_2 \,\,\frac{1}{2\pi \Delta z} \,
\exp{\!\left[-\frac{\theta_1^2 + \theta_2^2}{2 \Delta z}\right]}\,+\,
\CO\big(1/\Lambda\big),\\
&=\,1 \,+\, \CO\big(1/\Lambda\big)\,.
\end{aligned}
\end{equation}
Note that all dependence on the homotopy parameter $\hbar$ disappears
as soon as we perform the asymptotic expansion in $\Lambda$.

In principle, the contribution from the right-handed crossing in
Figure \ref{LocalOver} also includes the portion of the integration
domain ${T^2 - \Delta(\varepsilon)}$ where the roles of $\theta_1$ and
$\theta_2$ are swapped in \eqref{Lines}, with $X^+\equiv
X^+(\theta_1)$ and ${X^-\equiv X^-(\theta_2)}$.  In this case,
${\widehat\Delta z = -\Delta z + \hbar \, \theta_1 \theta_2 < 0}$ is
negative near the origin, and the self-linking integrand vanishes
identically by the definition of the heat form $\chi_\Lambda$.

Finally, to evaluate the local contribution from the left-handed
crossing in Figure \ref{LocalOver}, we simply swap the roles of $X^+$
and $X^-$.  Apparently from \eqref{LocO}, this swap is equivalent to
an orientation-reversal on $\BR^2$, so the sign of
$\slk_\kappa(C)|_{\{0\}}$ is reversed.  

Comparing to our conventions
for the writhe in Figure \ref{Writhe}, we conclude that
$\slk_\kappa(C)|_{\{0\}}$ is the local writhe of the given crossing in
$\Pi(C)$.  In total, the local contribution to $\slk_\kappa(C)$
from the crossings is precisely the Thurston-Bennequin invariant of $C$,
\begin{equation}\label{LocTB}
\sum_{a\in\textbf{I}}\,\slk_\kappa(C)\big|_a \,=\,
\w\!\big(\Pi(C)\big) \,=\, \tb(C)\,,
\end{equation}
where $\textbf{I}$ indexes the set of all crossings in the Lagrangian
projection.  The localization computation is also
consistent with Proposition \ref{HLink} regarding
Heisenberg linking, together with the diagrammatic formula for $\lk(C_1,C_2)$ in
\eqref{LkDiag}.

More interesting is the local contribution to $\slk_\kappa(C)$ from the
diagonal ${\Delta\subset T^2}$.   This
contribution does depend (weakly) on the value of $\hbar$, a small remnant of
the topological anomaly.  Integrating over a neighborhood of
the diagonal really means integrating over the two boundary regions on
the cylinder ${T^2-\Delta(\varepsilon)}$, as we have already considered in our
proof of the Fundamental Lemma in Section \ref{Fundamental}.  So we do
not need to perform any new computations to evaluate the contribution
from the diagonal.

We begin with the generic case ${\hbar\neq 1}$, for which the local
expression for the self-linking integrand appears in
\eqref{PullChiDIIO}.  Directly for ${\Lambda \gg 1}$,
\begin{equation}\label{RotI}
\begin{aligned}
&\slk_\kappa(C)\big|_\Delta \,\buildrel{\hbar\neq 1}\over =\,\\
&\qquad-\int_{S^1_+\cap
  [(1-\hbar)(\gamma\times\dot{\gamma})<0]}\mskip-25mu
d\phi\,\left[\int_0^\infty \!\!d\eta \,\,\frac{\Lambda\left(\dot{\gamma}\times\ddot{\gamma}\right)}{4\pi\,
  |1-\hbar|\,
  |\gamma\times\dot{\gamma}|}\,\exp{\!\left(-\frac{\Lambda\,
      ||\dot{\gamma}||^2\,\eta}{2\,|1-\hbar|\,
      |\gamma\times\dot{\gamma}|}\right)}\right] +\,\\
&\qquad-\int_{S^1_-\cap
  [(1-\hbar)(\gamma\times\dot{\gamma})>0]}\mskip-25mu
d\phi\,\left[\int_0^\infty \!\!d\eta \,\,\frac{\Lambda\left(\dot{\gamma}\times\ddot{\gamma}\right)}{4\pi\,
  |1-\hbar|\,
  |\gamma\times\dot{\gamma}|}\,\exp{\!\left(-\frac{\Lambda\,
      ||\dot{\gamma}||^2\,\eta}{2\,|1-\hbar|\,
      |\gamma\times\dot{\gamma}|}\right)}\right]\,.
\end{aligned}
\end{equation}
The integrals in the two lines of \eqref{RotI} describe
the respective local contributions to $\slk_\kappa(C)$ from collar neighborhoods
of the boundary circles $S^1_\pm$ on the cylinder in Figure
\ref{Cylinder}.  Both integrals appear with identical
signs, after one takes into account the relative boundary
orientations on $S^1_\pm$ and the explicit dependence of the integrand
on $\sgn(\eta)$ in \eqref{PullChiDIIO}.  We make a trivial change of
  variables so that the integral in the neighborhood of $S^1_-$ also
  runs over positive, as opposed to negative, values of $\eta$.  By
the positivity condition in \eqref{Posit}, the integral over $S^1_+$
runs over the subset where
${\left(1-\hbar\right)\left(\gamma\times\dot{\gamma}\right)<0}$, and the
integral over $S^1_-$ runs over the complement.  Finally, we extend
the integration range over the normal coordinate $\eta$ to infinity at
the cost of an exponentially small error for large $\Lambda$, and we
set ${\varepsilon = 0}$ at the lower limit of integration for $\eta$.
 
After integrating over $\eta$ in \eqref{RotI},
\begin{equation}
\slk_\kappa(C)\big|_\Delta \,\buildrel{\hbar\neq 1}\over =\,
-\frac{1}{2\pi}\int_{S^1_+\cap[(1-\hbar)(\gamma\times\dot{\gamma})<0]}\mskip-10mu
d\phi\,\,\frac{\dot{\gamma}\times\ddot{\gamma}}{||\dot{\gamma}||^2} 
\,-\,\frac{1}{2\pi}\int_{S^1_-\cap[(1-\hbar)(\gamma\times\dot{\gamma})>0]}\mskip-10mu
d\phi\,\,\frac{\dot{\gamma}\times\ddot{\gamma}}{||\dot{\gamma}||^2}\,,
\end{equation}
or put more succinctly,
\begin{equation}
\slk_\kappa(C)\big|_\Delta \,\buildrel{\hbar\neq 1}\over
=\,-\frac{1}{2\pi}\oint_{S^1}\!\!d\phi\,\frac{\dot{\gamma}\times\ddot{\gamma}}{||\dot{\gamma}||^2}\,.
\end{equation}
So long as ${\hbar\neq 1}$, all dependence on $\hbar$ disappears.
From the geometric expression for the rotation number in
\eqref{LocRot}, we deduce
\begin{equation}\label{RotII}
\slk_\kappa(C)\big|_\Delta \,\buildrel{\hbar\neq 1}\over =\,
-\rot(C)\,.
\end{equation}

On general grounds, the appearance of the rotation number in this
calculation is not so surprising, as one was guaranteed to find the
integral of some local geometric quantity on $C$.  However, the
integrality of the result \eqref{RotII} comes as a minor miracle,
which is far from obvious from the definition of the Legendrian
self-linking integral in \eqref{eq:1}.  Remember, the value of the
naive Gauss self-linking integral $\slk_0(C)$ is not even a
deformation-invariant!

We return to our formula in \eqref{PullChiDII} to evaluate the local
self-linking contribution from the diagonal when ${\hbar=1}$,
\begin{equation}\label{PullChiDDD}
\left(X \times
  X\right)^*\!\widehat\Gamma^*_\hbar\,\chi_\Lambda\Big|_{\Delta(\varepsilon)}
\,\buildrel{\hbar=1}\over=\,-\frac{3\,\Lambda}{2\pi
  \eta^2} \exp{\!\left[-\frac{3\,\Lambda\,||\dot{\gamma}||^2}{
      \left|\dot{\gamma}\times\ddot{\gamma}\right|}\frac{1}{|\eta|}\right]}\,d\phi\^d\eta
\,+\,\cdots\,.
\end{equation}
Unlike the expressions in \eqref{RotI}, which are non-zero
for ${\eta=0}$, the self-linking integrand in \eqref{PullChiDDD}
vanishes exponentially as ${\eta\to 0}$ for any ${\Lambda > 0}$.  By
inspection we conclude 
\begin{equation}
\slk_\kappa(C)\big|_\Delta \,\buildrel{\hbar=1}\over =\, 0\,.
\end{equation}
If the Heisenberg symmetry is preserved, the diagonal does not
contribute to the Legendrian self-linking integral.

At least informally, modulo precise control of the error terms, we
obtain from these local computations the statement in the Main Theorem,
\begin{equation}
\begin{aligned}
\slk_\kappa(C) \,&=\, \lim_{\Lambda\to\infty} \slk_\kappa(C) \,=\,
\sum_{a\in\textbf{I}}\slk_\kappa(C)\big|_a 
\,+\, \slk_\kappa(C)\big|_\Delta\,,\\
&=\,\begin{cases}
\quad\tb(C)-\rot(C)\,,\qquad &\hbar\,\neq\,1\,,\\
\quad\tb(C)\,,\qquad &\hbar\,=\,1\,.
\end{cases}
\end{aligned}
\end{equation}

\subsection{Some preliminaries}\label{Preps}

The informal localization computation in Section \ref{Local} is useful
for developing geometric intuition about the behavior of the Legendrian
self-linking integral.  To prove the Main Theorem, we retrace the same
route philosophically, but we exercise greater care in analyzing
the dependence of the error terms on $\Lambda$, at least when $\Lambda$ is
large.

Before we establish precise inequalities in
Section \ref{Bounds}, we need to introduce a bevy of constants related
to the geometry of the knot ${C \subset \BR^3}$ and its Lagrangian
projection $\Pi(C)$.  These constants are important, as the
required bounds fall out naturally from them.

\vspace{2pt}\noindent
\textbf{\small{Local Neighborhoods of Crossings}}
\vspace{2pt}

The notion of a collar neighborhood for the boundary of ${T^2 -
  \Delta(\varepsilon)}$ is unambiguous, but we also need a proper
notion for the neighborhood of each crossing in $\Pi(C)$.  The trick
will be to choose these neighborhoods to be small enough so that the geometry of
$C$ is controlled over each neighborhood.  Informally, we 
treated this issue by linearizing $C$ in Figure \ref{LocalOver}, but now we
work nonlinearly.

As before, ${\gamma:S^1\to \BR^2}$ is the
regular immersed plane curve which is the Lagrangian projection of the
embedding ${X:S^1\to\BR^3}$,
\begin{align}\label{eq:39}
\gamma\,=\,\Pi\circ X\,,\qquad\qquad \gamma(\theta) \,\equiv\,
  \left(x(\theta),\, y(\theta)\right).
\end{align}
By assumption, $\gamma$ has only a finite number $n$ of simple 
double-point singularities, located at positions
\begin{equation}
\gamma_1,\,\gamma_2,\,\ldots,\,\gamma_n\,\in\,\BR^2\,.
\end{equation}
See Figure \ref{LgTref} for our canonical trefoil example, with ${n=5}$.  Each
crossing $\gamma_a$ for ${a=1,\ldots,n}$ lies under a pair 
of corresponding points ${(p_a, q_a)}$ on the knot $C$, 
\begin{equation}
\gamma_a \,=\, \Pi(p_a) \,=\, \Pi(q_a)\,,\qquad\qquad p_a,q_a \in C\,.
\end{equation}
Let $z_a^{}$ and $z_a'$ be the respective heights of $p_a$
and $q_a$, so that these points have coordinates in $\BR^3$ given by 
\begin{equation}
p_a \,=\, \big(\gamma_a,\,z_a\big),\qquad\qquad q_a \,=\,
\big(\gamma_a^{},\,z'_a\big)\,.
\end{equation}
As in the informal computation, an important
geometric quantity is the absolute difference $\Delta z_a$ in the
heights of $p_a$ and $q_a$ over the crossing,
\begin{equation}
\Delta z_a \,=\, \left|z_a^{} - z_a'\right|\,>\, 0\,.
\end{equation}
See Figure \ref{Lagproj} for a local (nonlinear) picture of $C$ near the points
$p_a$ and $q_a$.

\begin{figure}[t]
\centering
\def\svgwidth{60mm}
\begingroup%
  \makeatletter%
  \providecommand\color[2][]{%
    \errmessage{(Inkscape) Color is used for the text in Inkscape, but the package 'color.sty' is not loaded}%
    \renewcommand\color[2][]{}%
  }%
  \providecommand\transparent[1]{%
    \errmessage{(Inkscape) Transparency is used (non-zero) for the text in Inkscape, but the package 'transparent.sty' is not loaded}%
    \renewcommand\transparent[1]{}%
  }%
  \providecommand\rotatebox[2]{#2}%
  \ifx\svgwidth\undefined%
    \setlength{\unitlength}{235.7bp}%
    \ifx\svgscale\undefined%
      \relax%
    \else%
      \setlength{\unitlength}{\unitlength * \real{\svgscale}}%
    \fi%
  \else%
    \setlength{\unitlength}{\svgwidth}%
  \fi%
  \global\let\svgwidth\undefined%
  \global\let\svgscale\undefined%
  \makeatother%
  \begin{picture}(1,0.78192618)%
    \put(0,0){\includegraphics[width=\unitlength]{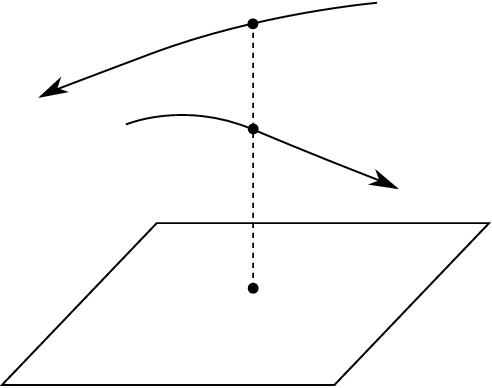}}%
    \put(0.56699418,0.70218189){\color[rgb]{0,0,0}\makebox(0,0)[lb]{\smash{$p_a=(\gamma_a^{},z_a^{})$}}}%
    \put(0.55842263,0.55475193){\color[rgb]{0,0,0}\makebox(0,0)[lb]{\smash{$q_a=(\gamma_a^{},z_a')$}}}%
    \put(0.5687085,0.20503421){\color[rgb]{0,0,0}\makebox(0,0)[lb]{\smash{$\gamma_{a}$}}}%
  \end{picture}%
\endgroup%
\caption{Two points $p_a$ and $q_a$ on $C$ with coincident Lagrangian projections.}\label{Lagproj}
\end{figure}
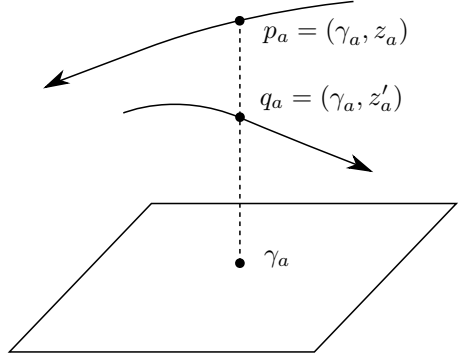

Let ${D\!\left(\gamma_a;\Rh\right)\equiv D_a(\Rh) \subset \BR^2}$ be
the open disc of radius ${\Rh > 0}$ centered at the location
$\gamma_a$ of a given crossing in the plane.  The union ${\cup\,
  D_a(\Rh)}$ of these discs, each with the same radius $\Rh$, provides an open neighborhood for all crossings in
$\Pi(C)$.  We now choose ${\Rh > 0}$ to be sufficiently small so that the 
following statements are true at each crossing.  By continuity of $X$
and compactness of the closure $\bar{D_a(\Rh)}$, such a choice is
always possible.

For ease of notation, we suppress the crossing index `$a$' below.
\begin{enumerate}
\item[1.]  The disc $D(\Rh)$ intersects the immersed plane curve $\gamma$ in
  two arcs, as shown in Figure \ref{Inter}.  We denote these arcs by
  $\gamma^+$ and $\gamma^-$, where `$\pm$' indicate the respective upper
  and lower strands at the crossing.  Over the disc,
  the embedding $X$ restricts to a pair of maps ${X^-(\theta_1) =
    (\gamma^-(\theta_1),\,z^-(\theta_1))}$ and ${X^+(\theta_2) =
    (\gamma^+(\theta_2),\, z^+(\theta_2))}$, with 
  ${z^+ > z^-}$.  Here $X^\pm$ are nonlinear analogues of the
  expressions in \eqref{Lines}.

\begin{figure}[ht]
\centering
\def\svgwidth{50mm}
\begingroup%
  \makeatletter%
  \providecommand\color[2][]{%
    \errmessage{(Inkscape) Color is used for the text in Inkscape, but the package 'color.sty' is not loaded}%
    \renewcommand\color[2][]{}%
  }%
  \providecommand\transparent[1]{%
    \errmessage{(Inkscape) Transparency is used (non-zero) for the text in Inkscape, but the package 'transparent.sty' is not loaded}%
    \renewcommand\transparent[1]{}%
  }%
  \providecommand\rotatebox[2]{#2}%
  \ifx\svgwidth\undefined%
    \setlength{\unitlength}{192.575bp}%
    \ifx\svgscale\undefined%
      \relax%
    \else%
      \setlength{\unitlength}{\unitlength * \real{\svgscale}}%
    \fi%
  \else%
    \setlength{\unitlength}{\svgwidth}%
  \fi%
  \global\let\svgwidth\undefined%
  \global\let\svgscale\undefined%
  \makeatother%
  \begin{picture}(1,0.92072627)%
    \put(0,0){\includegraphics[width=\unitlength]{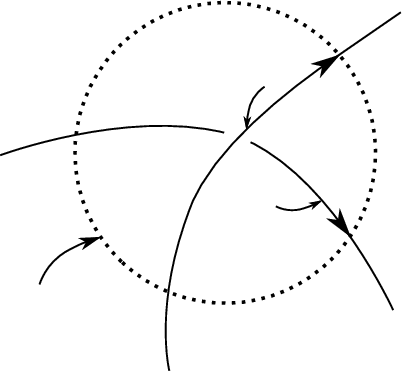}}%
    \put(0.63442225,0.75){\color[rgb]{0,0,0}\makebox(0,0)[lb]{\smash{$\gamma^+$}}}%
    \put(0.64124139,0.43999776){\color[rgb]{0,0,0}\makebox(0,0)[lb]{\smash{$\gamma^-$}}}%
    \put(0.01810753,0.12){\color[rgb]{0,0,0}\makebox(0,0)[lb]{\smash{$D(\Rh)$}}}%
  \end{picture}%
\endgroup%
\caption{Intersection of $D(\Rh)$ and $\gamma$.}\label{Inter}
\end{figure}
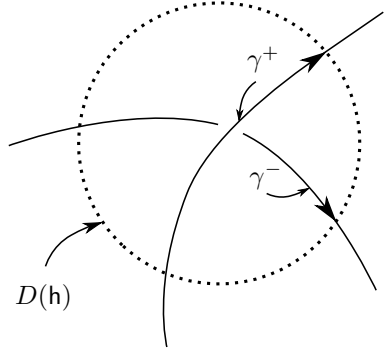 

\item[2.]  With the same arcs $\gamma^\pm$ in mind, let ${\Gamma_{2d}:\BR^2\times\BR^2\to\BR^2}$ be the
  two-dimensional difference map
\begin{equation}
\Gamma_{2d}(u,v) \,=\, v - u\,.
\end{equation}
Consider the composition
\begin{equation}
\varphi \,=\, \Gamma_{2d} \circ \left(\gamma^- \times \gamma^+\right)
= \gamma^+ - \gamma^-\,,
\end{equation}
which maps the region on $T^2$ where
${\gamma^-\times\gamma^+}$ is locally defined to another region on the $uv$-plane.  Then
${\varphi \equiv \left(u(\theta_1,\theta_2), v(\theta_1,\theta_2)\right)}$
is a diffeomorphism to a curvy quadrilateral
region ${\CQ}$ around the origin in the $uv$-plane, as in
Figure \ref{Quadri}.  

Eventually, we will use $\varphi$ to make a
change-of-variables to simplify the self-linking integrand in the
neighborhood of the crossing.

\begin{figure}[h]
\centering
\def\svgwidth{90mm}
\begingroup%
  \makeatletter%
  \providecommand\color[2][]{%
    \errmessage{(Inkscape) Color is used for the text in Inkscape, but the package 'color.sty' is not loaded}%
    \renewcommand\color[2][]{}%
  }%
  \providecommand\transparent[1]{%
    \errmessage{(Inkscape) Transparency is used (non-zero) for the text in Inkscape, but the package 'transparent.sty' is not loaded}%
    \renewcommand\transparent[1]{}%
  }%
  \providecommand\rotatebox[2]{#2}%
  \ifx\svgwidth\undefined%
    \setlength{\unitlength}{390.54202231bp}%
    \ifx\svgscale\undefined%
      \relax%
    \else%
      \setlength{\unitlength}{\unitlength * \real{\svgscale}}%
    \fi%
  \else%
    \setlength{\unitlength}{\svgwidth}%
  \fi%
  \global\let\svgwidth\undefined%
  \global\let\svgscale\undefined%
  \makeatother%
  \begin{picture}(1,0.3593107)%
    \put(0,0){\includegraphics[width=\unitlength]{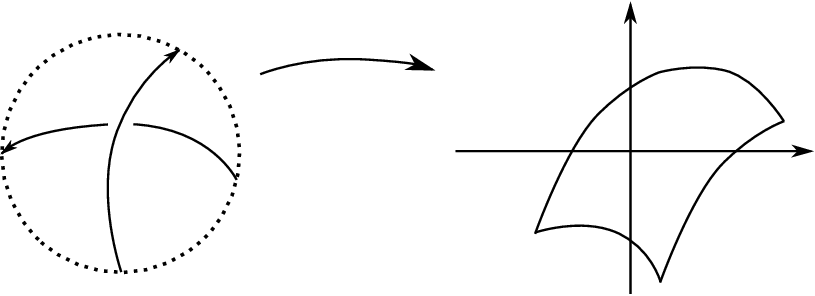}}%
    \put(0.42057511,0.32058002){\color[rgb]{0,0,0}\makebox(0,0)[lb]{\smash{$\varphi$}}}%
    \put(0.89775252,0.22017614){\color[rgb]{0,0,0}\makebox(0,0)[lb]{\smash{${\mathcal Q}$}}}%
    \put(0.96185018,0.12){\color[rgb]{0,0,0}\makebox(0,0)[lb]{\smash{$u$}}}%
    \put(0.796753,0.34307401){\color[rgb]{0,0,0}\makebox(0,0)[lb]{\smash{$v$}}}%
  \end{picture}%
\endgroup%
\caption{Diffeomorphism ${\varphi=\gamma^+-\gamma^-}$ to $\CQ$.}\label{Quadri}
\end{figure}
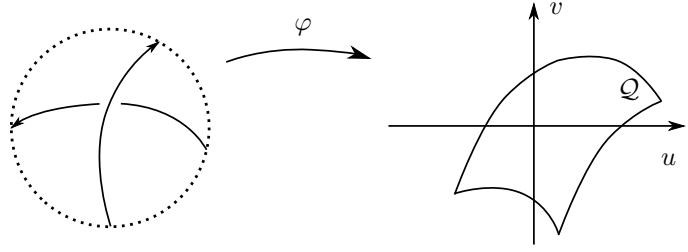 

\item[3.]  Let $J_\varphi$ be the Jacobian for the change-of-variables
  induced by $\varphi$ from $(\theta_1,\theta_2)$ to $(u,v)$,
\begin{equation}
J_\varphi \,=\, \left|\frac{d\gamma^+}{d\theta_2} \times \frac{d\gamma^-}{d\theta_1}\right|\,.
\end{equation}
Since $\varphi$ is a diffeomorphism, ${J_\varphi \neq 0}$ is non-vanishing
throughout the domain of $\varphi$, as illustrated in
Figure \ref{Nonvan}.  We go slightly further and assume that
$J_\varphi$ is uniformly bounded from below by a positive 
constant 
\begin{equation}\label{BoundJ}
0 < \RJ_0 < J_\varphi\,.
\end{equation}

\begin{figure}[ht]
\centering
\def\svgwidth{40mm}
\begingroup%
  \makeatletter%
  \providecommand\color[2][]{%
    \errmessage{(Inkscape) Color is used for the text in Inkscape, but the package 'color.sty' is not loaded}%
    \renewcommand\color[2][]{}%
  }%
  \providecommand\transparent[1]{%
    \errmessage{(Inkscape) Transparency is used (non-zero) for the text in Inkscape, but the package 'transparent.sty' is not loaded}%
    \renewcommand\transparent[1]{}%
  }%
  \providecommand\rotatebox[2]{#2}%
  \ifx\svgwidth\undefined%
    \setlength{\unitlength}{145.92890694bp}%
    \ifx\svgscale\undefined%
      \relax%
    \else%
      \setlength{\unitlength}{\unitlength * \real{\svgscale}}%
    \fi%
  \else%
    \setlength{\unitlength}{\svgwidth}%
  \fi%
  \global\let\svgwidth\undefined%
  \global\let\svgscale\undefined%
  \makeatother%
\begin{picture}(1,1.00007174)%
    \put(0,0){\includegraphics[width=\unitlength]{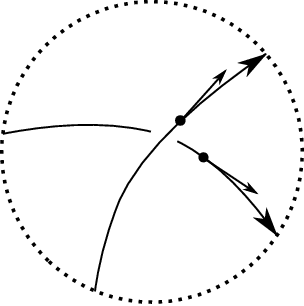}}%
    \put(0.45483778,0.77513141){\color[rgb]{0,0,0}\makebox(0,0)[lb]{\smash{$\frac{d\gamma^+}{d\theta_{2}}$}}}%
    \put(0.46975858,0.39668187){\color[rgb]{0,0,0}\makebox(0,0)[lb]{\smash{$\frac{d\gamma^-}{d\theta_{1}}$}}}%
  \end{picture}%
\endgroup%
\caption{The Jacobian ${J_\varphi \neq 0}$ of $\varphi$.}\label{Nonvan}
\end{figure} 

\item[4.]  For the crossing labelled by `$a$', consider all pairs of points on ${C\subset \BR^3}$ which lie
  in the image of ${X^+_a \times X^-_a}$ over the disc $D_a(\Rh)$.  Then the difference in heights
  ${|z^+_a - z^-_a|}$ for all such pairs lies in the range 
\begin{equation}\label{eq:40}
\left(1-\frac{\Rc}{2}\right) \Delta z_a \,<\, |z^+_a - z^-_a| \,<\, \left(1+\frac{\Rc}{2}\right)\Delta z_a\,.
\end{equation}
Here $\Rc$ is a positive constant, independent of $a$, bounded by 
\begin{equation}\label{Boundc}
0 \,<\, \Rc \,<\, \ha\,.
\end{equation}

This assumption in \eqref{eq:40} implies that the
difference $|z^+_a - z^-_a|$ for all pairs of points on $C$ projecting to
$D_a(\Rh)$ obeys\footnote{The bound in \eqref{eq:41}
  is not sharp given \eqref{eq:40} but will suffice for us.}  
\begin{equation}\label{eq:41}
1-\Rc<\frac{\Delta z_a}{|z^+_a-z^-_a|}<1+\Rc\,,
\end{equation}
for $\Rc$ in the given range.  Informally, the constant $\Rc$ controls
the variation in the vertical separation between the two strands of $C$ over the disc $D_a(\Rh)$, relative to the
separation $\Delta z_a$ over the crossing itself.  The constant `$1/2$'
in the bound \eqref{Boundc} is a convenient choice related to other inequalities 
later.
\end{enumerate}

\noindent
\textbf{\small{Bounds on the Integrand}}
\vspace{2pt}

Our assumption about the radius $\Rh$ of the discs $D_a(\Rh)$ gives us
adequate control on the local geometry of $C$ above each
crossing $\gamma_a$.  We now introduce further constants related to the
magnitude of the self-linking integrand itself.

First, let ${\RZ>0}$ be the height of the knot ${C \subset \BR^3}$.
More formally, $\RZ$ is the maximum vertical separation between any
pair of points on $C$,
\begin{equation}
\RZ \,=\, \max_{\left(\theta_1,\theta_2\right)\in T^2}\Big[\big|z(\theta_2)-z(\theta_1)\big|\Big].
\end{equation}
We now fix a small positive constant ${\delta > 0}$ which will control
the errors.  In reference to the dependence of the heat form $\chi_\Lambda$
on ${r^2=x^2 + y^2}$ and $z$ in \eqref{PullChi}, we note the following lemma.
\begin{lem}\label{BoundLM}
Given $\delta>0$ and sufficiently large $\Lambda$, there exists a
constant ${\Rr_\Lambda>0}$, depending on $\Lambda$, so that for all
${r > \Rr_\Lambda}$, 
\begin{equation}\label{TopB}
\underset{0<z\leq \RZ}{\sup}\left[\frac{\Lambda}{2\pi
  z}\exp\Big(-\frac{\Lambda r^2}{2z}\Big)\right]<\delta\,,
\end{equation}
and 
\begin{align}\label{eq:42a}
\underset{0<z\leq \RZ}{\sup}\left[\frac{\Lambda}{4\pi
  z^2}\exp\Big(-\frac{\Lambda r^2}{2z}\Big)\right]<\delta\,.
\end{align}
\end{lem}
\noindent
\textbf{Proof}
\vskip2pt
\noindent
The proof of the lemma is elementary, but we wish to gain precise
knowledge about how $\Rr_\Lambda$ must depend upon $\RZ$ and
$\Lambda$ for the bounds to hold.  The bounds in \eqref{TopB} and
\eqref{eq:42a} can be treated similarly; we start with the bound in \eqref{TopB}.

We differentiate the function in \eqref{TopB} with respect to $z$,
\begin{equation}\label{eq:43}
\frac{\partial}{\partial z}\Bigg[\frac{\Lambda}{2\pi
  z}\exp\Big(-\frac{\Lambda
  r^2}{2z}\Big)\Bigg]\,=\,\frac{\Lambda}{2\pi z^2}\Big(\frac{\Lambda
  r^2}{2z}-1\Big)\exp\Big(-\frac{\Lambda r^2}{2z}\Big)\,.
\end{equation}
The derivative in \eqref{eq:43} is positive so long as 
\begin{equation}\label{Boundr}
\frac{\Lambda r^2}{2 z} \,>\, 1,\qquad 0 \,<\, z \,\le\,\RZ\,,
\end{equation}
which in turn is equivalent to 
\begin{equation}\label{BoundrII}
r^2 \,>\, \frac{2 \RZ}{\Lambda}\,.
\end{equation}

We will actually require a stronger bound on $r$ in regard to its dependence on $\Lambda$.  We set 
\begin{equation}\label{RrLambda}
\Rr^2_\Lambda \,=\, \frac{2\RZ}{\sqrt{\Lambda}}\,.
\end{equation}
So long as ${\Lambda > 1}$, the condition ${r > \Rr_\Lambda}$
implies the bound in \eqref{BoundrII}, so that 
\begin{equation}
\frac{\partial}{\partial z}\Bigg[\frac{\Lambda}{2\pi
  z}\exp\Big(-\frac{\Lambda r^2}{2z}\Big)\Bigg]\,>\,0\,.
\end{equation}
The function in \eqref{TopB} therefore increases with $z$ for all ${r >
  \Rr_\Lambda}$, and the supremum is achieved at the value ${z = \RZ}$,
\begin{equation}\label{eq:44}
\underset{0<z\leq \RZ}{\sup}\left[\frac{\Lambda}{2\pi
  z}\exp\Big(-\frac{\Lambda r^2}{2z}\Big)\right]\,=\,\frac{\Lambda}{2\pi
  \RZ}\exp\Big(-\frac{\Lambda r^2}{2\RZ}\Big)\,,\qquad\qquad r > \Rr_\Lambda\,.
\end{equation}

For ${r > \Rr_\Lambda}$, the argument of the exponential
in \eqref{eq:44} satisfies
\begin{equation}\label{eq:45}
\frac{\Lambda r^2}{2\RZ}\,>\,\frac{\Lambda \Rr_{\Lambda}^2}{2\RZ}\,=\,\sqrt{\Lambda}\,.
\end{equation}
Had we imposed the weak inequality in \eqref{BoundrII},
the quantity ${\Lambda r^2/2\RZ}$ would have been bounded from
below only by a constant, independent of $\Lambda$, and we would have
no chance to achieve the bound by $\delta$ in \eqref{TopB}.
Instead, with the strong inequality ${r > \Rr_\Lambda}$ in \eqref{RrLambda},
\begin{equation}\label{eq:46}
\frac{\Lambda}{2\pi \RZ}\exp\Big(-\frac{\Lambda r^2}{2\RZ}\Big)\,<\,\frac{\Lambda}{2\pi \RZ}\exp{\!\left(-\sqrt{\Lambda}\right)}\,.
\end{equation}
For any positive ${\delta > 0}$, we now take $\Lambda$ sufficiently large
so that 
\begin{equation}
\frac{\Lambda}{2\pi \RZ}\exp{\!\left(-\sqrt{\Lambda}\right)} \,<\, \delta\,,
\end{equation}
implying via \eqref{eq:44} and \eqref{eq:46} the desired inequality in \eqref{TopB}.

The analysis of the function in \eqref{eq:42a} is identical, up to a
factor of $2$, due to appearance of the same Gaussian factor.  In this case, we take ${\Rr_\Lambda^2 = 4
  \RZ/\sqrt{\Lambda}}$.  To treat both cases of Lemma \ref{BoundLM}
simultaneously, we set
\begin{align}\label{eq:47}
\Rr^2_{\Lambda}=\max\left\{\frac{2\RZ}{\sqrt{\Lambda}},\frac{4\RZ}{\sqrt{\Lambda}}\right\}\,=\,\frac{4\RZ}{\sqrt{\Lambda}}\,.
\end{align}
With this choice for $\Rr_\Lambda$, the lemma follows.\qquad$\square$

We require one additional quantity to make our bounds on the error.
This quantity will be a function of $\Lambda$ which arises in
reference to the magnitude of the Gaussian integrand \eqref{GausR2} at
a crossing.  Briefly, for given 
$\delta$ and crossing index `$a$', we let ${\RR_a(\Lambda) > 0}$ be
the solution to  
\begin{equation}\label{CondRLam}
\frac{\Lambda}{2\pi\Delta z_a}\exp{\!\left[-\frac{\Lambda
      \RR_a(\Lambda)^2}{2\Delta z_a}\right]} \,=\,\delta\,.
\end{equation}
Explicitly,
\begin{equation}\label{BigRLam}
\RR_a(\Lambda)^2\,=\,\frac{2\Delta z_a}{\Lambda}\ln\left(\frac {\Lambda}{2\pi\Delta z_a\delta}\right)\,.
\end{equation}
As a function of $\Lambda$, $\RR_a(\Lambda)$ is decreasing for
sufficiently large $\Lambda$ and vanishes in the limit ${\Lambda\to\infty}$.

For the convenience of the reader, we summarize all these constants in
Table \ref{Constants}.
\begin{table}
\begin{center}
\begin{tabular}{l l}
\qquad$\Lambda$ \quad&\quad Width of Gaussian in $\chi_\Lambda$.\\[1
ex]
\qquad$M, m$ \quad&\quad Any positive constant which only depends upon
${C \subset \BR^3}$.\\
&\quad The values of $M$ and $m$ may differ at various places in the text.\\[1 ex]
\qquad$\delta$ \quad&\quad A fixed, small positive constant.  An
error less than ${M \delta}$ is negligible.\\[1
ex]
\qquad$\Delta z_a$ \quad&\quad Vertical displacement of $C$ over each
crossing ${\gamma_a\in\BR^2}$.\\[1 ex]
\qquad$\Rh$ \quad&\quad Radius of disc ${D_a(\Rh) \subset \BR^2}$
centered at each crossing.\\[1 ex]
\qquad$\RJ_0$ \quad&\quad Lower bound ${0 < \RJ_0 < J_\varphi}$ for
the Jacobian of $\varphi$.\\[1 ex]
\qquad$\Rc$ \quad&\quad Positive constant $<1/2$ for which ${1-\Rc <
  \Delta z_a/|z_a^+-z_a^-| < 1 + \Rc}$.\\[1 ex]
\qquad$\RZ$ \quad&\quad Total height of ${C \subset \BR^3}$.\\[1 ex]
\qquad$\Rr_\Lambda$ \quad&\quad A positive number given by
${\Rr_\Lambda^2 = 4\RZ/\sqrt{\Lambda}}$.  For sufficiently large
$\Lambda$\\
&\quad and ${r > \Rr_\Lambda}$, the inequalities in Lemma
\ref{BoundLM} are true.\\[1 ex]
\qquad$\RR_a(\Lambda)$ \quad&\quad Positive solution to 
${(\Lambda/2\pi\Delta z_a) \exp{\!\left[-\Lambda
      \RR_a(\Lambda)^2/2\Delta z_a\right]}=\delta}$.\\ \\
\end{tabular}
\caption{List of Constants.}
\label{Constants}
\end{center}
\end{table}

\subsection{Bounds on the error}\label{Bounds}

We now establish the necessary bounds to prove the Main Theorem.

The most fundamental bound emerges trivially from the definition of
the constant $\Rr_\Lambda$ in Lemma \ref{BoundLM}.  Let 
$\CU$ and $\CV$ be the subsets of the cylinder ${T^2 -
  \Delta(\varepsilon)}$ defined by 
\begin{equation}\label{BigU}
\CU \,=\, \big\{
(\theta_1,\theta_2)\in T^2-\Delta(\varepsilon) \,\Big|\, [\Delta x^2 + \Delta
  y^2]\big|_{(\theta_1,\theta_2)} \le \Rr^2_\Lambda\big\},
\end{equation}
and 
\begin{equation}
\CV \,=\, \big\{
(\theta_1,\theta_2)\in T^2-\Delta(\varepsilon) \,\Big|\, [\Delta x^2 + \Delta
  y^2]\big|_{(\theta_1,\theta_2)}  > \Rr^2_\Lambda\big\}\,.
\end{equation}
That is, $\CU$ consists of those pairs of points $(\theta_1,\theta_2)$ whose
images in the $xy$-plane under the map ${\Pi\circ\left(X\times
    X\right)}$ lie within the critical distance $\Rr_\Lambda$, and $\CV$
consists of those pairs separated in the $xy$-plane by a distance greater than
$\Rr_\Lambda$.  

Since $\CU$ and $\CV$ are complementary subsets of the cylinder, the
self-linking integral can be written as the sum 
\begin{equation}
\slk_\kappa(C) \,=\, \int_{\CU} \left(X^{} \times
  X^{}\right)^*\!\widehat\Gamma_{\hbar}^*\,\chi_\Lambda \,+\,
\int_{\CV} \left(X^{} \times
  X^{}\right)^*\!\widehat\Gamma_{\hbar}^*\,\chi_\Lambda\,.
\end{equation}
By the very definition of $\Rr_\Lambda$ in Lemma \ref{BoundLM}, the magnitude of the self-linking
integrand \eqref{PullChiDD} on $\CV$ is everywhere bounded by $\delta$.  Automatically,
\begin{equation}
\left|\slk_\kappa(C) \,-\, \int_{\CU} \left(X^{} \times
  X^{}\right)^*\!\widehat\Gamma_{\hbar}^*\,\chi_\Lambda\right| \,=\, 
\left|\int_{\CV} \left(X^{} \times
  X^{}\right)^*\!\widehat\Gamma_{\hbar}^*\,\chi_\Lambda\right| < M\,\delta\,.
\end{equation}
Here $M$ is a constant, depending on the
geometry of the knot ${C\subset \BR^3}$  but independent of $\Lambda$.  For small $\delta$, the value of
$\slk_\kappa(C)$ is well-approximated by the integral over the subset
$\CU$.

Our next task is to characterize the points in the domain of
integration that lie in
$\CU$.  If we formally set
${\Rr_\Lambda=0}$ in \eqref{BigU}, then $\CU$ consists of those pairs
$(\theta_1,\theta_2)$ which become coincident after projection to the
$xy$-plane.  Such pairs either lie along the diagonal ${\Delta \subset
  T^2}$, or they lie in the preimage of a crossing in the Lagrangian
projection of $C$.  When ${\Rr_\Lambda > 0}$ is positive but
sufficiently small, these closed sets fatten, and  $\CU$ is contained within
the disjoint union of a tubular neighborhood $N_\Delta(\Rw)$ of the
diagonal and a collection of open balls ${B_b(\Rh) \subset T^2}$
associated to the crossings,
\begin{equation}\label{BigUII}
\CU \,\subset\, N_\Delta(\Rw) \cup \bigcup_{b=1}^{2n} B_b(\Rh)\,.
\end{equation}
See Figure \ref{NbigU} for a schematic picture of the open set containing
$\CU$ for sufficiently small $\Rr_\Lambda$.  As in the picture, each
crossing in $\Pi(C)$ has two preimages, which are 
exchanged when the roles of $\theta_1$ and $\theta_2$ swap.  Thus if
$\Pi(C)$ has $n$ crossings, the index `$b$' on the balls $B_b(\Rh)$ runs to $2n$.

\begin{figure}[t]
\begin{center}
\includegraphics[scale=0.60]{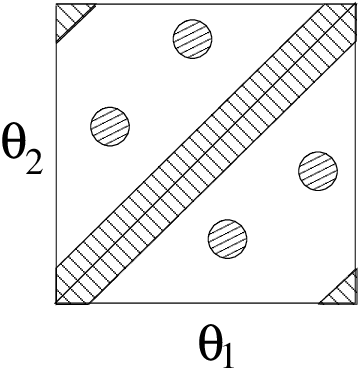}
\caption{Neighborhood of the subset $\CU$ for sufficiently small
  $\Rr_\Lambda>0$.}\label{NbigU}
\end{center}
\end{figure}

In writing $B_b(\Rh)$ for the open ball in $T^2$, we abuse notation
somewhat.  By assumption, the radius of $B_b(\Rh)$ is fixed so that
this ball lies in the preimage of the corresponding disc ${D_a(\Rh) \subset \BR^2}$ under the map
${\gamma^-_a\times\gamma^+_a}$ in Figure \ref{Inter}, 
\begin{equation}
\left(\gamma^-_a\times\gamma^+_a\right)\!\left(B_b(\Rh)\right)
\,\subset\, D_a(\Rh)\,,\qquad\qquad a \equiv b \mod n\,.
\end{equation}
Therefore the radius of $B_b(\Rh)$ is not necessarily equal to $\Rh$, but it is
determined by $\Rh$ independently of $\Lambda$.  Once $\Rh$ is fixed
in terms of the geometry of $C$, we can always take 
$\Lambda$ sufficiently large and ${\Rr_\Lambda \sim \Lambda^{-1/4}}$
sufficiently small so that $B_b(\Rh)$ contains the relevant portion of $\CU$.

Similarly, the tubular neighborhood $N_\Delta(\Rw)$ of the diagonal has
width ${\Rw > 0}$, meaning that points in $N_\Delta(\Rw)$ satisfy
${|\theta_2-\theta_1| < \Rw}$ for the parameters in Figure \ref{NbigU}.  A crucial step will be to fix the value of $\Rw$,
which must be large enough so that $N_\Delta(\Rw)$ contains the
piece of $\CU$ near the diagonal, but also small enough so that the local
analysis from Section \ref{LocalD} is applicable everywhere
in $N_\Delta(\Rw)$.  

According to the next lemma, both conditions on $N_\Delta(\Rw)$ can be
simultaneously satisfied once we set 
\begin{equation}
\Rw \,=\, \frac{\Rr_\Lambda \sqrt{2}}{m}\,,
\end{equation}
with a positive constant ${m > 0}$ defined geometrically by 
\begin{equation}\label{MinG}
m^2 \,=\, \min_{\theta\in S^1}\Big[||\dot{\gamma}(\theta)||^2\Big]\,.
\end{equation}
Since ${\gamma = \Pi\circ X}$ is an immersion, the minimum speed in
\eqref{MinG} is bounded away from zero, which is essential.  The
constant $\sqrt{2}$ is inessential and could be absorbed into the
definition of $m$.
\begin{lem}\label{BoundW}
For sufficiently large $\Lambda$ and ${\Rw = \Rr_\Lambda\sqrt{2}/m}$, the
tubular neighborhood $N_\Delta(\Rw)$ contains the diagonal component
of $\CU$. 
\end{lem}  
\noindent
\textbf{Proof}
\vskip2pt
\noindent
The lemma says that, away from crossings, all points on the projection
$\Pi(C)$ which lie within the distance $\Rr_\Lambda$ of a given point
$\gamma(\theta)$ are contained within the image of the interval
$\left[\theta - \Rr_\Lambda \sqrt{2}/m,\,\theta + \Rr_\Lambda \sqrt{2}/m\right]$
under $\gamma$, for all values of $\theta$.  See Figure \ref{Branch}
for an illustration of this claim.  The minimum speed along
$\gamma$ naturally sets the scale of the required interval, which
decreases with increasing $m$.

\begin{figure}[h]
\begin{center}
\def\svgwidth{55mm}
\begingroup%
  \makeatletter%
  \providecommand\color[2][]{%
    \errmessage{(Inkscape) Color is used for the text in Inkscape, but the package 'color.sty' is not loaded}%
    \renewcommand\color[2][]{}%
  }%
  \providecommand\transparent[1]{%
    \errmessage{(Inkscape) Transparency is used (non-zero) for the text in Inkscape, but the package 'transparent.sty' is not loaded}%
    \renewcommand\transparent[1]{}%
  }%
  \providecommand\rotatebox[2]{#2}%
  \ifx\svgwidth\undefined%
    \setlength{\unitlength}{252.97203623bp}%
    \ifx\svgscale\undefined%
      \relax%
    \else%
      \setlength{\unitlength}{\unitlength * \real{\svgscale}}%
    \fi%
  \else%
    \setlength{\unitlength}{\svgwidth}%
  \fi%
  \global\let\svgwidth\undefined%
  \global\let\svgscale\undefined%
  \makeatother%
  \begin{picture}(1,0.70238512)%
    \put(0,0){\includegraphics[width=\unitlength]{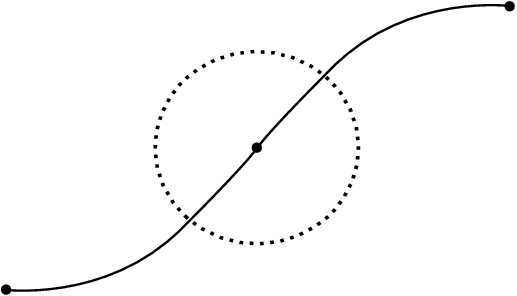}}%
    \put(0.77688627,0.63){\color[rgb]{0,0,0}\makebox(0,0)[lb]{\smash{$\gamma(\theta+\Rr_{\Lambda}\sqrt{2}/m)$}}}%
    \put(0.52654978,0.26){\color[rgb]{0,0,0}\makebox(0,0)[lb]{\smash{$\gamma(\theta)$}}}%
    \put(0.03355366,-0.07){\color[rgb]{0,0,0}\makebox(0,0)[lb]{\smash{$\gamma(\theta-\Rr_{\Lambda}\sqrt{2}/m)$}}}%
  \end{picture}%
\endgroup%
\vspace{5pt}
\caption{Points within the distance $\Rr_\Lambda$ of
  $\gamma(\theta)$.}\label{Branch} 
\end{center}
\end{figure}
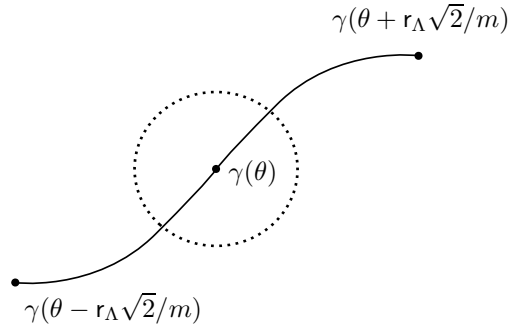 

For any fixed ${\theta_0\in S^1}$, define the function 
\begin{equation}\label{DefF}
F(\theta)=\frac{d}{d\theta}\big|\big|\gamma(\theta)-\gamma(\theta_0)\big|\big|^2\,=\,
2 \,\big(\gamma(\theta) - \gamma(\theta_0)\big)\cdot\dot{\gamma}(\theta)\,.
\end{equation} 
Then ${F(\theta_0)=0}$, and the derivative of $F$ is 
\begin{equation}
\frac{dF}{d\theta}(\theta)\,=\,2\,||\dot{\gamma}||^2 \,+\, 2 \,\big(\gamma(\theta) - \gamma(\theta_0)\big)\cdot\ddot{\gamma}(\theta)\,.
\end{equation}
Since $\gamma$ is an immersion, ${||\dot{\gamma}|| \ge m}$ is bounded
below, and ${||\ddot{\gamma}|| < \infty}$ is bounded above.  So if
${||\gamma(\theta)-\gamma(\theta_0)|| \le \Rr_\Lambda}$ for
sufficiently large $\Lambda$ and hence sufficiently small $\Rr_\Lambda$, then
\begin{equation}\label{BoundF}
\frac{dF}{d\theta}(\theta) > m^2\,.
\end{equation}
Integrating the inequality in \eqref{BoundF}, we obtain 
\begin{equation}
F(\theta) \,>\, m^2 \left|\theta-\theta_0\right|.
\end{equation}
Integrating once more from the definition \eqref{DefF} of $F(\theta)$,
\begin{equation}
\big|\big|\gamma(\theta) - \gamma(\theta_0)\big|\big|^2 \,>\, \ha m^2\,
|\theta - \theta_0|^2\,,
\end{equation}
or
\begin{equation}
\big|\big|\gamma(\theta) - \gamma(\theta_0)\big|\big| \,>\, \frac{m}{\sqrt{2}}\,
|\theta - \theta_0|\,.
\end{equation}
Thus if ${||\gamma(\theta)-\gamma(\theta_0)||=\Rr_\Lambda}$, then
${|\theta-\theta_0| < \Rr_\Lambda\sqrt{2}/m}$, which is the required bound.\qquad$\square$

Because $\CU$ is contained in the union of $N_\Delta(\Rw)$ and ${\cup
  B_b(\Rh)}$, we have a relation between the 
corresponding self-linking integrals,
\begin{equation}
\left|\int_\CU \left(X^{} \times
  X^{}\right)^*\!\widehat\Gamma_{\hbar}^*\,\chi_\Lambda \,-\,
\int_{N_\Delta(\Rw)}\mskip-25mu\left(X^{} \times
  X^{}\right)^*\!\widehat\Gamma_{\hbar}^*\,\chi_\Lambda \,-\,
\sum_{b=1}^{2n} \int_{B_b(\Rh)}\mskip-25mu\left(X^{} \times
  X^{}\right)^*\!\widehat\Gamma_{\hbar}^*\,\chi_\Lambda\right| \,<\, M\,\delta\,.
\end{equation}
This inequality again follows from Lemma \ref{BoundLM}, because 
points contained in either $N_\Delta(\Rw)$ or  $B_b(\Rh)$ but
not in $\CU$ lie in $\CV$, where the magnitude of the self-linking
integrand is bounded by $\delta$.  So for sufficiently large
$\Lambda$, we just need to evaluate the self-linking integral over the
balls $B_b(\Rh)$ and the tubular neighborhood $N_\Delta(\Rw)$ in
Figure \ref{NbigU}.

\vspace{2pt}\noindent
\textbf{\small{Error analysis at a crossing}}
\vspace{2pt}

We first evaluate the self-linking integral over the ball $B_b(\Rh)$.
By the positivity condition on the heat form $\chi_\Lambda$, the
self-linking integrand vanishes in exactly half the
balls.  Reshuffling indices as necessary, we consider only those
$B_a(\Rh)$ for ${a=1,\ldots,n}$ 
on which ${\widehat\Delta z > 0}$ and the integrand is non-zero.

With malice aforethought, we have arranged that the image of
$B_a(\Rh)$ under the product map ${\gamma^-_a\times\gamma^+_a}$ lies
in the disc ${D_a(\Rh)\subset \BR^2}$, where we have control over the
geometry of $C$.  In particular, the map $\varphi_a$ in Figure
\ref{Quadri} restricts to a diffeomorphism from $B_a(\Rh)$ to a curvy 
quadrilateral region $\CQ$ about the origin in the $uv$-plane,
\begin{equation}\label{VarphiII}
\varphi_a^{} \,=\, \gamma^+_a(\theta_2) - \gamma^-_a(\theta_1) \,\equiv\,\left(u(\theta_1,\theta_2),
    v(\theta_1,\theta_2)\right)\,.
\end{equation}  
Our analysis will be performed using the $uv$-coordinates.  In these
coordinates, the self-linking integrand  
\eqref{PullChiDD} simplifies,
\begin{equation}\label{UVChi}
\begin{aligned}
&\frac{\Lambda}{2\pi\widehat\Delta z} \,\e{\!-\Lambda
  (\Delta x^2 + \Delta y^2)/2 \widehat\Delta z} \left[d\Delta
  x\^d\Delta y \,+ \ha \left(\Delta x \,d\Delta y -
    \Delta y\,d\Delta x\right)\!\^\frac{d\widehat\Delta
    z}{\widehat\Delta z}\right]\\
&\qquad\,=\,\frac{\Lambda}{2\pi\widehat\Delta z(u,v)} \,\e{\!-\Lambda
  (u^2 + v^2)/2 \widehat\Delta z(u,v)} \left[du\^dv \,+ \ha \left(u \,dv -
    v\,du\right)\!\^\frac{d\widehat\Delta
    z(u,v)}{\widehat\Delta z(u,v)}\right].
\end{aligned}
\end{equation}
Here $\Delta x$ and $\Delta y$ are identified with
the Cartesian coordinates $u$ and $v$, and $\widehat\Delta z$ is
considered to be a function of $(u,v)$.  All unknown functional
dependence of the self-linking integrand in \eqref{UVChi} is absorbed
into ${\widehat\Delta z(u,v)}$.

The integrand in \eqref{UVChi} is a sum of two terms, 
\begin{equation}
\Psi \,=\, \frac{\Lambda}{2\pi\widehat\Delta z(u,v)} \,\e{\!-\Lambda
  (u^2 + v^2)/2 \widehat\Delta z(u,v)} \, du\^dv\,,
\end{equation}
and 
\begin{equation}\label{BigXi}
\Xi \,=\, \frac{\Lambda}{4\pi\widehat\Delta z(u,v)} \,\e{\!-\Lambda
  (u^2 + v^2)/2 \widehat\Delta z(u,v)} \left(u \,dv -
    v\,du\right)\!\^\frac{d\widehat\Delta
    z(u,v)}{\widehat\Delta z(u,v)}\,.
\end{equation}
Hence after making the change-of-variables in $\CQ$,
\begin{equation}\label{TwoIs}
\int_{B_a(\Rh)} \left(X^{} \times
  X^{}\right)^*\!\widehat\Gamma_{\hbar}^*\,\chi_\Lambda \,=\, \int_\CQ
\Psi \,+\, \int_\CQ \Xi\,.
\end{equation}
The analysis of the two terms on the right in \eqref{TwoIs} is different.
Morally, $\Xi$ is higher-order in $u$ and $v$ so will be irrelevant
when $\Lambda$ is large and ${u,v \ll 1/\Lambda}$.  By
contrast, $\Psi$ is always relevant.  We analyze the integrals of
$\Psi$ and $\Xi$ over $\CQ$ in turn.

\begin{figure}[ht]
$$\begin{matrix}
&\def\svgwidth{45mm}
\begingroup%
  \makeatletter%
  \providecommand\color[2][]{%
    \errmessage{(Inkscape) Color is used for the text in Inkscape, but the package 'color.sty' is not loaded}%
    \renewcommand\color[2][]{}%
  }%
  \providecommand\transparent[1]{%
    \errmessage{(Inkscape) Transparency is used (non-zero) for the text in Inkscape, but the package 'transparent.sty' is not loaded}%
    \renewcommand\transparent[1]{}%
  }%
  \providecommand\rotatebox[2]{#2}%
  \ifx\svgwidth\undefined%
    \setlength{\unitlength}{192.575bp}%
    \ifx\svgscale\undefined%
      \relax%
    \else%
      \setlength{\unitlength}{\unitlength * \real{\svgscale}}%
    \fi%
  \else%
    \setlength{\unitlength}{\svgwidth}%
  \fi%
  \global\let\svgwidth\undefined%
  \global\let\svgscale\undefined%
  \makeatother%
  \begin{picture}(1,0.92238205)%
    \put(0,0){\includegraphics[width=\unitlength]{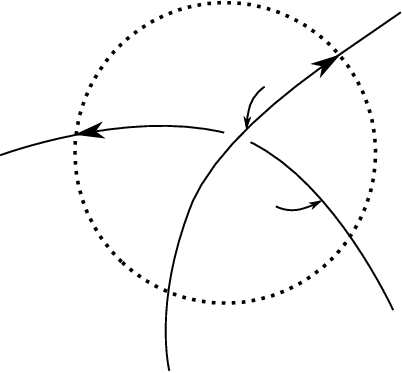}}%
    \put(0.63442225,0.75){\color[rgb]{0,0,0}\makebox(0,0)[lb]{\smash{$\gamma^{+}$}}}%
    \put(0.6,0.42){\color[rgb]{0,0,0}\makebox(0,0)[lb]{\smash{$\gamma^{-}$}}}%
  \end{picture}%
\endgroup%
\qquad & \qquad 
&\def\svgwidth{45mm}
\begingroup%
  \makeatletter%
  \providecommand\color[2][]{%
    \errmessage{(Inkscape) Color is used for the text in Inkscape, but the package 'color.sty' is not loaded}%
    \renewcommand\color[2][]{}%
  }%
  \providecommand\transparent[1]{%
    \errmessage{(Inkscape) Transparency is used (non-zero) for the text in Inkscape, but the package 'transparent.sty' is not loaded}%
    \renewcommand\transparent[1]{}%
  }%
  \providecommand\rotatebox[2]{#2}%
  \ifx\svgwidth\undefined%
    \setlength{\unitlength}{192.575bp}%
    \ifx\svgscale\undefined%
      \relax%
    \else%
      \setlength{\unitlength}{\unitlength * \real{\svgscale}}%
    \fi%
  \else%
    \setlength{\unitlength}{\svgwidth}%
  \fi%
  \global\let\svgwidth\undefined%
  \global\let\svgscale\undefined%
  \makeatother%
  \begin{picture}(1,0.92238205)%
    \put(0,0){\includegraphics[width=\unitlength]{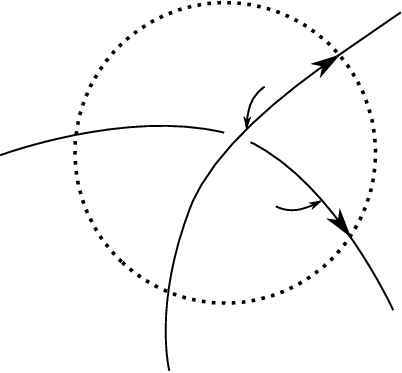}}%
    \put(0.63442225,0.75){\color[rgb]{0,0,0}\makebox(0,0)[lb]{\smash{$\gamma^{+}$}}}%
    \put(0.6,0.42){\color[rgb]{0,0,0}\makebox(0,0)[lb]{\smash{$\gamma^{-}$}}}%
  \end{picture}%
\endgroup%
\\
&(a.)\quad\deg(\varphi_a) = +1 \qquad & \qquad &(b.)\quad\deg(\varphi_a) = -1
\end{matrix}$$
\caption{Sign of the Jacobian for $\varphi_a$.}\label{Jacob} 
\end{figure}
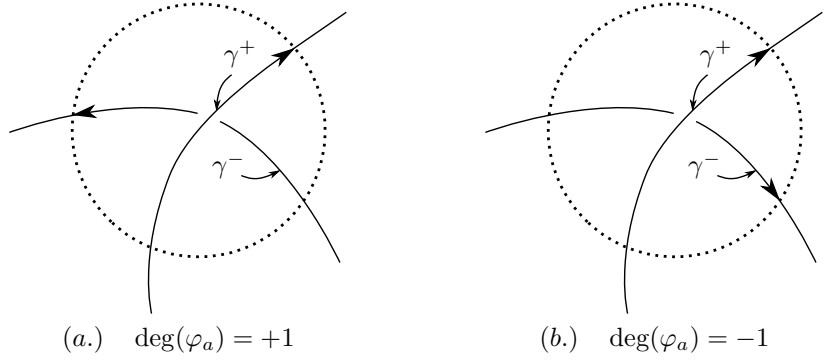

\vspace{2pt}\noindent
\textbf{\small{Integral of $\Psi$}}
\vspace{2pt}

Explicitly, the integral of $\Psi$ is given by a kind of
nonlinear Gaussian,
\begin{equation}
\int_\CQ \Psi \,=\, \deg(\varphi_a) \int_\CQ \frac{\Lambda}{2\pi\widehat\Delta z(u,v)} \,\exp{\!\left[-\frac{\Lambda
  \left(u^2 + v^2\right)}{2\,\widehat\Delta z(u,v)}\right]} \,
du\,dv\,,\qquad\widehat\Delta z(u,v) > 0\,.
\end{equation}
Here ${\deg(\varphi_a) = \pm 1}$ depending upon whether the
diffeomorphism 
${\varphi_a:B_a(\Rh)\to\CQ}$ preserves or reverses orientation.
Equivalently, from the expression in \eqref{VarphiII}, the sign is
determined by the Jacobian in the expansion  
\begin{equation}
du\^dv \,=\,
-\left(\frac{d\gamma^-}{d\theta_1}\times\frac{d\gamma^+}{d\theta_2}\right)
d\theta_1\^d\theta_2\,.
\end{equation}
By inspection of Figure \ref{Jacob}, $\deg(\varphi_a)$ is exactly the local
writhe at the given crossing,
\begin{equation}
\deg(\varphi_a) \,=\, \Rw_a\,.
\end{equation}
To suppress pernicious signs for the remainder, we
assume ${\Rw_a=+1}$. 

For large $\Lambda$, we evaluate the integral of $\Psi$ over $\CQ$ in
two steps.
\begin{itemize}
\item[1.] We replace the unknown function $\widehat\Delta
  z(u,v)$ by the constant displacement $\Delta z_a$ at the crossing, with error
\begin{equation}\label{ItemI}
\left|\int_\CQ\frac{\Lambda}{2\pi\widehat\Delta
  z(u,v)}\exp{\left[-\frac{\Lambda\left(u^2+v^2\right)}{2\,\widehat\Delta
      z(u,v)}\right]} du\,
dv \,-\, \int_\CQ\frac{\Lambda}{2\pi\Delta
  z_a}\exp{\!\left[-\frac{\Lambda\left(u^2+v^2\right)}{2\,\Delta
      z_a}\right]}du\, dv \right|<\, M \delta\,.
\end{equation}
\item[2.] We extend the subsequent range of Gaussian integration from $\CQ$ to
  $\BR^2$ so that the Gaussian integral can be performed
  analytically, with error 
\begin{equation}\label{TrvGaus}
\left|\int_\CQ\frac{\Lambda}{2\pi\Delta
  z_a}\exp{\!\left[-\frac{\Lambda\left(u^2+v^2\right)}{2\,\Delta
      z_a}\right]} du\,dv - \int_{\BR^2}\frac{\Lambda}{2\pi\Delta
  z_a}\exp{\!\left[-\frac{\Lambda\left(u^2+v^2\right)}{2\,\Delta
      z_a}\right]} du\,dv \right|<\,\delta\,.
\end{equation}
\end{itemize}
Of these steps, only the first is non-trivial.  For the second, because the
Gaussian integral over $\BR^2$ is normalized to unity independent of
$\Lambda$, we can always choose $\Lambda$ sufficiently large so that
\begin{align}\label{eq:71}
\int_{\BR^2-\CQ}\frac{\Lambda}{2\pi\Delta
  z_a}\exp{\!\left[-\frac{\Lambda\left(u^2+v^2\right)}{2\,\Delta
      z_a}\right]}du\, dv\,<\,\delta\,.
\end{align}
Informally, we performed both these steps in arriving at \eqref{GausR2}.

For the first step, we use heavily the positive function
$\RR_a(\Lambda)$ which satisfies 
\begin{equation}\label{BigRLamII}
\frac{\Lambda}{2\pi\Delta z_a}\exp{\!\left[-\frac{\Lambda\,
      \RR_a(\Lambda)^2}{2\Delta z_a}\right]} \,=\,\delta\,,
\end{equation}
and vanishes monotonically as ${\Lambda\to\infty}$.  The function
$\RR_a(\Lambda)$ sets the minimum distance from the origin for
which the Gaussian integrand in \eqref{ItemI} becomes negligible.

To be on the safe side in our bounds, we will have to work with a
slightly larger distance ${\RR_a(\Lambda/2) > \RR_a(\Lambda)}$.
Let ${B_0\equiv B_0(\RR_a(\Lambda/2))}$ be 
the ball of radius $\RR_a(\Lambda/2)$ which is centered at the origin
in $\CQ$,
\begin{equation}\label{BZero}
B_0:\quad u^2+v^2 \,<\, \RR_a^2(\Lambda/2)\,.
\end{equation}
We shall prove that when $(u,v)$ lies in $B_0$, the difference
between the nonlinear and the usual Gaussian in \eqref{ItemI} is small,
\begin{equation}\label{SmallBibB}
\left|\int_{B_0}\frac{\Lambda}{2\pi\widehat\Delta
  z(u,v)}\exp{\!\left[-\frac{\Lambda\left(u^2+v^2\right)}{2\widehat\Delta
      z(u,v)}\right]} du\,dv \,-\,
\int_{B_0}\frac{\Lambda}{2\pi\Delta
  z_a}\exp{\!\left[-\frac{\Lambda\left(u^2+v^2\right)}{2\Delta
      z_a}\right]} du\,dv\right| \,<\, \delta\,.
\end{equation}
Otherwise, when $(u,v)$ lies outside $B_0$ in $\CQ$, we show that both
integrals are separately small, with 
\begin{equation}\label{NonBGauss}
\int_{\CQ-B_0}\frac{\Lambda}{2\pi\widehat\Delta
  z(u,v)}\exp{\!\left[-\frac{\Lambda\left(u^2+v^2\right)}{2\widehat\Delta
      z(u,v)}\right]} du\, dv \,< M\,\delta\,,
\end{equation}
and
\begin{equation}\label{BGauss}
\int_{\CQ-B_0}\frac{\Lambda}{2\pi\Delta
  z_a}\exp{\!\left[-\frac{\Lambda\left(u^2+v^2\right)}{2\Delta
  z_a}\right]} du\,dv\,<\, M\,\delta\,.
\end{equation}
Thus the difference must also be small in ${\CQ - B_0}$.  This trick
is the engine of asymptotic analysis.  See Ch.\,$6$ in
\cite{Bender:1999} for further background on this idea.

We begin by establishing some easy bounds when  $(u,v)$ lies inside
the ball ${B_0 \subset \CQ}$.  From the definition of
$\RR_a(\Lambda/2)$,
\begin{equation}
u^2 + v^2 < \RR_a^2(\Lambda/2) \,=\, \frac{4\Delta z_a}{\Lambda}\ln\frac {\Lambda}{4\pi\Delta z_a\delta}\,,
\end{equation}
so the argument of the Gaussian is bounded by 
\begin{equation}\label{eq:65}
\frac{\Lambda (u^2+v^2)}{2\Delta z_a}<2\ln\frac{\Lambda}{4\pi\Delta
  z_a\delta}\,.
\end{equation}
On the other hand, consider the difference ${\widehat\Delta z(u,v) -
\Delta z_a}$.  As a function of $(u,v)$, the difference vanishes at
${u=v=0}$ and is differentiable there, so  
\begin{equation}
\big|\widehat\Delta z(u,v) -
\Delta z_a\big| \,<\, M \sqrt{u^2 + v^2} \,<\, M \, \RR_a(\Lambda/2)\,,
\end{equation}
for some constant $M$ depending on $C$.  Immediately, since
$\RR_a^2(\Lambda/2)$ scales like ${\ln\Lambda/\Lambda}$, the relative
fluctuations in height about $\Delta z_a$ satisfy 
\begin{align}\label{eq:66}
\frac{\big|\widehat\Delta z(u,v)-\Delta z_a\big|}{\Delta z_a}\,<\, \frac{M\,\RR_a(\Lambda/2)}{\Delta z_a}\,\sim\,\left(\frac{\ln\Lambda}{\Lambda}\right)^{1/2}\,.
\end{align}
Directly from \eqref{eq:65} and \eqref{eq:66},
\begin{align}\label{eq:67}
\frac{\Lambda (u^2+v^2)}{2\,\Delta z_a}\cdot\frac{\big|\widehat\Delta
  z(u,v)-\Delta z_a\big|}{\Delta z_a}\,<\,
M\,\frac{(\ln\Lambda)^{3/2}}{\Lambda^{1/2}}\,. 
\end{align}
The constant $M$ in \eqref{eq:67} is not necessarily the same  as the
constant $M$ in \eqref{eq:66}!

Given the relative similarity between the integrands in
\eqref{SmallBibB}, we consider their ratio 
\begin{equation}
q \,=\, \frac{\Lambda}{2\pi\widehat\Delta
        z(u,v)}
      \exp{\!\left[-\frac{\Lambda \left(u^2+v^2\right)}{2\widehat\Delta
          z(u,v)}\right]}\Biggr/\frac{\Lambda}{2\pi\Delta
        z_a}\exp{\!\left[-\frac{\Lambda \left(u^2+v^2\right)}{2 \Delta
          z_a}\right]}\,.
\end{equation}
With some algebra, this ratio can be recast as 
\begin{equation}
\begin{aligned}
q\,&=\,\frac{\Delta z_a}{\widehat\Delta z(u,v)}\exp{\!\left[\frac{\Lambda\left(u^2+v^2\right)}{2}\cdot\frac{\widehat\Delta z(u,v)-\Delta z_a}{\widehat\Delta z(u,v)\Delta z_a}\right]}\,,\\
&=\left[1\,+\,\frac{\widehat\Delta z(u,v)-\Delta z_a}{\Delta
    z_a}\right]^{-1}\cdot\exp{\!\left[\frac{\Lambda\left(u^2+v^2\right)}{2\Delta z_a}\cdot\frac{\widehat\Delta z(u,v)-\Delta z_a}{\Delta z_a}\cdot\frac{\Delta z_a}{\widehat\Delta z(u,v)}\right]}\,.
\end{aligned}
\end{equation}
By the estimates in \eqref{eq:66} and \eqref{eq:67}, the prefactor in
$q$ approaches unity and the argument of the exponential vanishes as
${\Lambda\to\infty}$.  Hence we can choose $\Lambda$ sufficiently large so that 
\begin{equation}\label{eq:68}
|q-1|\,<\,\delta\,.
\end{equation}
With this control over the fractional error, the difference between
the nonlinear and the usual Gaussian in
\eqref{SmallBibB} is bounded by 
\begin{equation}\label{estimate1}
\begin{aligned}
&\int_{B_0}\left|\frac{\Lambda}{2\pi\widehat\Delta
  z(u,v)}\exp{\!\left[-\frac{\Lambda\left(u^2+v^2\right)}{2\widehat\Delta
  z(u,v)}\right]}\,-\,\frac{\Lambda}{2\pi\Delta
  z_a}\exp{\!\left[-\frac{\Lambda\left(u^2+v^2\right)}{2\Delta
  z_a}\right]}\right| du\, dv\\
&\qquad\qquad\le\int_{B_0}\frac{\Lambda\,|q-1|}{2\pi\Delta
  z_a}\exp{\!\left[-\frac{\Lambda\left(u^2+v^2\right)}{2\Delta
      z_a}\right]} du\, dv\\
&\qquad\qquad<\,\int_{B_0}\frac{\Lambda\,\delta}{2\pi\Delta
  z_a}\exp{\!\left[-\frac{\Lambda\left(u^2+v^2\right)}{2\Delta
      z_a}\right]} du\, dv \,<\, \delta\,.
\end{aligned}
\end{equation}

We are left to examine what happens when $(u,v)$ lies outside the ball
${B_0\subset \CQ}$ of radius ${\RR_a(\Lambda/2)}$, meaning 
\begin{equation}
u^2+v^2\ge\RR_a^2(\Lambda/2)\,.
\end{equation}
To start, the bound on the Gaussian in \eqref{BGauss} is trivial
because the integrand is bounded by $\delta$ for all points $(u,v)$
outside the ball of radius $\RR_a(\Lambda)$, and ${\RR_a(\Lambda/2) >
  \RR_a(\Lambda)}$.  So the real task is to establish the bound for
the nonlinear Gaussian in \eqref{NonBGauss}.

Consider the following function on $\CQ$,
\begin{equation}\label{NewLam}
\widehat\Lambda(u,v) \,=\, \Lambda \cdot \frac{\Delta z_a}{\widehat\Delta z(u,v)}\,.
\end{equation}
Conceptually, we interpret $\widehat\Lambda$ as a fluctuating,
position-dependent version of the parameter $\Lambda$, so that the 
width of the nonlinear Gaussian varies from point-to-point on $\CQ$.
By the estimate in \eqref{eq:41}, the relative fluctuation factor is bounded
from below everywhere on $\CQ$ by 
\begin{equation}
\ha \,<\, 1 \,-\, \Rc \,<\, \frac{\Delta z_a}{\widehat\Delta z(u,v)}\,.
\end{equation}
Consequently, by the definition of $\widehat\Lambda$,
\begin{equation}\label{BHatLam}
\frac{\Lambda}{2} \,<\, \widehat\Lambda(u,v)\,.
\end{equation}

Associated to the local parameter $\widehat\Lambda(u,v)$ we have
a local scale $\RR_a(\widehat\Lambda(u,v))$, also a 
function of $u$ and $v$.  Since $\RR_a$ is monotonically decreasing
for large $\Lambda$, the lower bound on $\widehat\Lambda$ in
\eqref{BHatLam} means that 
\begin{equation}
\RR_a^2(\widehat\Lambda(u,v))\,<\,\RR_a^2(\Lambda/2) \,<\, u^2+v^2\,.
\end{equation}
Thus, again by the definition of $\RR_a(\Lambda)$ in \eqref{BigRLamII},
\begin{align}\label{eq:64}
\frac{\widehat\Lambda}{2\pi\Delta z_a}\exp{\!\left[-\frac{\widehat\Lambda\left(u^2+v^2\right)}{2\,\Delta z_a}\right]}\,<\,\delta\,,
\end{align}
or by substitution from \eqref{NewLam},
\begin{equation}\label{NonLBd}
\frac{\Lambda}{2\pi\widehat\Delta z(u,v)}\exp{\!\left[-\frac{\Lambda\left(u^2+v^2\right)}{2\,\widehat\Delta z(u,v)}\right]}\,<\,\delta\,.
\end{equation} 

The bound on the nonlinear Gaussian in \eqref{NonLBd} is exactly what
we need to control the integral over ${\CQ - B_0}$, so that 
\begin{equation}\label{estimate2}
\int_{\CQ-B_0}\frac{\Lambda}{2\pi\widehat\Delta
  z(u,v)}\exp{\!\left[-\frac{\Lambda\left(u^2+v^2\right)}{2\widehat\Delta
      z(u,v)}\right]} du\, dv\,<\,M\,\delta\,.
\end{equation}
Combining with the trivial bound in \eqref{TrvGaus}, we deduce that
the desired integral of $\Psi$ over $\CQ$ can be well-approximated for
large $\Lambda$ by the
naive Gaussian integral,
\begin{equation}\label{eq:69}
\left|\int_\CQ\Psi \,-\, \deg(\varphi_a) \int_{\BR^2}\frac{\Lambda}{2\pi\Delta
  z_a}\exp{\!\left[-\frac{\Lambda\left(u^2+v^2\right)}{2\Delta
  z_a}\right]} du\, dv\right|
  < \, M\,\delta\,.
\end{equation}

\vspace{2pt}\noindent
\textbf{\small{Integral of $\Xi$}}
\vspace{2pt}

We are not finished with our error analysis at the crossing, because
we still must consider the integral of $\Xi$ over $\CQ$ in
\eqref{TwoIs}.  We will show that the contribution of $\Xi$ is negligible for large $\Lambda$,
\begin{equation}\label{SmallXi}
\left|\int_\CQ\Xi\,\right| \,<\, M\,\delta\,.
\end{equation}
Explicitly, from the formula in \eqref{BigXi}, the integral of $\Xi$ is given in the $(u,v)$-coordinates
by 
\begin{equation}\label{IntXi}
\int_\CQ \Xi \,=\, -\int_\CQ\frac{\Lambda}{4\pi\widehat\Delta z(u,v)^2} \,\e{\!-\Lambda
  (u^2 + v^2)/2 \widehat\Delta z(u,v)} \left(u
  \frac{\partial\widehat\Delta z}{\partial u} \,+\, v
  \frac{\partial\widehat\Delta z}{\partial v}\right) du \, dv\,.
\end{equation}

Again, we consider the cases that $(u,v)$ lies inside the ball $B_0$
and outside the ball $B_0$ separately.  When $(u,v)$ lies outside the
ball ${B_0\subset\CQ}$, then by the definition of $\RR_a(\Lambda)$
in \eqref{BigRLamII}, 
\begin{equation}\label{BdXi}
\left|\frac{\Lambda}{4\pi\widehat\Delta z(u,v)^2} \,\e{\!-\Lambda
  (u^2 + v^2)/2 \widehat\Delta z(u,v)} \left(u
  \frac{\partial\widehat\Delta z}{\partial u} \,+\, v
  \frac{\partial\widehat\Delta z}{\partial v}\right)\right| <\, M\,\delta\,.
\end{equation}
Here we note that the extra factors of ${1/\widehat\Delta z}$ and
${\left(u\,\partial/\partial u + v\,\partial/\partial
    v\right)\widehat\Delta z}$ in \eqref{BdXi} are smooth functions
bounded independently of $\Lambda$ on $\CQ$.  These functions do not
alter the bound by $\delta$ but are absorbed into the constant $M$.

Otherwise, for points inside $B_0$, we have a bound 
\begin{equation}\label{eq:83}
\int_{B_0} \frac{\Lambda}{4\pi\widehat\Delta
  z(u,v)^2} \exp{\!\left[-\frac{\Lambda\left(u^2+v^2\right)}{2\widehat\Delta
      z(u,v)}\right]} du\,dv\, < M\,,
\end{equation}
which follows by the same arguments used to produce the estimate in
\eqref{estimate1}.  Also, since ${\left(u\,\partial/\partial u + v\,\partial/\partial
    v\right)\widehat\Delta z}$ is bounded in $B_0$ and ${|u|,|v| \le
  \RR_a(\Lambda/2) \to 0}$ as $\Lambda\to\infty$, we can always choose
$\Lambda$ so that 
\begin{equation}\label{Smallu}
\left|u
  \frac{\partial\widehat\Delta z}{\partial u} \,+\, v
  \frac{\partial\widehat\Delta z}{\partial v}\right|
<\,\delta\,,
\end{equation} 
for all $(u,v)$ in $B_0$.  Combining the bounds in \eqref{BdXi},
\eqref{eq:83}, and \eqref{Smallu} for outside and inside $B_0$, we
obtain the conclusion in \eqref{SmallXi}.  

In summary, these bounds establish the informal localization formula
in \eqref{LocTB} for any crossing of $\Pi(C)$.\qquad$\square$

\vspace{2pt}\noindent
\textbf{\small{Error analysis near the diagonal}}
\vspace{2pt}

Our final goal is to evaluate the self-linking integral over the
tubular neighborhood $N_\Delta(\Rw)$ of the diagonal ${\Delta\subset
  T^2}$, where ${\Rw\sim\Lambda^{-1/4}}$ is the width set in Lemma
\ref{BoundW}.  For the informal localization computation in
\eqref{RotI}, we used the leading term in the Taylor expansion
of the self-linking integrand near the
diagonal to approximate the integral.  Depending upon the value of the
parameter $\hbar$, we abbreviate this leading term by 
\begin{equation}\label{Case1}
\Phi_\hbar \,\buildrel{\hbar\neq 1}\over=\, -\sgn(\eta)\,\frac{\Lambda\left(\dot{\gamma}\times\ddot{\gamma}\right)}{4\pi\,
  |1-\hbar|\,
  |\gamma\times\dot{\gamma}|}\,\exp{\!\left[-\frac{\Lambda\,
      ||\dot{\gamma}||^2\,|\eta|}{2\,|1-\hbar|\,
      |\gamma\times\dot{\gamma}|}\right]}\,d\phi\^d\eta\,,
\end{equation}
or
\begin{equation}\label{Case2}
\Phi_\hbar \,\buildrel{\hbar=1}\over =\,-\frac{3\,\Lambda}{2\pi
  \eta^2} \exp{\!\left[-\frac{3\,\Lambda\,||\dot{\gamma}||^2}{
      \left|\dot{\gamma}\times\ddot{\gamma}\right|}\frac{1}{|\eta|}\right]}\,d\phi\^d\eta\,.
\end{equation}
In both cases we assume that the local positivity condition ${\widehat\Delta z
  > 0}$ is satisfied, as in \eqref{Posit}.  Otherwise, ${\Phi_\hbar
  \equiv 0}$.  

To justify our localization computation, we must
demonstrate for sufficiently large $\Lambda$ the bound  
\begin{equation}\label{DiagBd}
\left|\int_{N_\Delta(\Rw)}\mskip-25mu\left(X^{} \times
  X^{}\right)^*\!\widehat\Gamma_{\hbar}^*\,\chi_\Lambda \,-\,
\int_{N_\Delta(\Rw)}\mskip-15mu\Phi_\hbar\,\right| <M\,\delta\,.
\end{equation}
The behavior of $\Phi_\hbar$ for small $\eta$ depends very much on
whether $\hbar$ is equal to one or not, so we treat the cases in
\eqref{Case1} and \eqref{Case2} separately.  Because the generic case
${\hbar\neq 1}$ is the more involved, and the more interesting, we
begin with it.

\vspace{2pt}\noindent
\textbf{{\small Generic case ${\hbar\neq 1}$}}
\vspace{2pt}

In principle, the error in the leading approximation to
$\left(X^{} \times
  X^{}\right)^*\!\widehat\Gamma_{\hbar}^*\,\chi_\Lambda$ for small
$\eta$ is controlled by the magnitude of the next-order term in the
Taylor expansion.  We will need this correction term for our analysis.
Briefly, by the same computations leading to \eqref{HeatArg} in
Section \ref{LocalD}, the argument of the heat kernel admits the
second-order expansion 
\begin{equation}\label{HeatArgII}
\frac{\Delta x^2 + \Delta y^2}{2\,\widehat\Delta z}
\,\buildrel{\hbar\neq 1}\over =\,
-\frac{||\dot{\gamma}||^2 \,\eta}{2 \left(1-\hbar\right)
  \left(\gamma\times\dot{\gamma}\right)}\left[1 \,+
  \left(\frac{\dot{\gamma}\cdot\ddot{\gamma}}{||\dot{\gamma}||^2} -
    \ha \frac{\gamma\times\ddot{\gamma}}{\gamma\times\dot{\gamma}}\right)\eta
  \,+\, \CO\big(\eta^2\big)\right],
\end{equation}
where ${\dot{\gamma}\cdot\ddot{\gamma} \equiv \dot{x}\,\ddot{x} +
  \dot{y}\,\ddot{y}}$.  Similarly for the pullback of the heat form
$\chi_\Lambda$ itself,
\begin{equation}\label{SecondOrd}
\begin{aligned}
&\left(X^{} \times
  X^{}\right)^*\!\widehat\Gamma_{\hbar}^*\,\chi_\Lambda
\,\buildrel{\hbar\neq 1}\over =\,\\
&-\frac{\Lambda\,\sgn(\eta)\,d\phi\^d\eta}{4\pi\,
  |1-\hbar|\,
  |\gamma\times\dot{\gamma}|}
\left[\left(\dot{\gamma}\times\ddot{\gamma}\right) \,+\, \left(\ha
  \left(\dot{\gamma}\times\dddot{\gamma}\right) -
  \frac{\left(\gamma\times\ddot{\gamma}\right)\left(\dot{\gamma}\times\ddot{\gamma}\right)}{\left(\gamma\times\dot{\gamma}\right)}\right)\eta \,+\,
  \CO\big(\eta^2\big)\right]\times\\
&\times\,\exp{\!\left[-\frac{\Lambda\,||\dot{\gamma}||^2 \,|\eta|}{2
      \left|1-\hbar\right|\left|\gamma\times\dot{\gamma}\right|}\left(1 \,+
  \left(\frac{\dot{\gamma}\cdot\ddot{\gamma}}{||\dot{\gamma}||^2} -
    \ha \frac{\gamma\times\ddot{\gamma}}{\gamma\times\dot{\gamma}}\right)\eta
  \,+\, \CO\big(\eta^2\big)
    \right)\right]}\,.
\end{aligned}
\end{equation}
We omit the computation leading to \eqref{SecondOrd}, since the
details of this formula will not be so important.  The expansion merely
confirms that both the prefactor and the argument of the exponential
for $\Phi_\hbar$ in \eqref{Case1} receive further corrections at the
next order in $\eta$, as determined by the geometry of the
projection $\Pi(C)$.

Validity of the leading approximation $\Phi_\hbar$ requires that the
correction terms in \eqref{SecondOrd} be small.  At least informally,
for the argument of the heat kernel in \eqref{HeatArgII} we require 
\begin{equation}\label{SmallC}
\left|\,\left(\frac{\dot{\gamma}\cdot\ddot{\gamma}}{||\dot{\gamma}||^2}
    - \ha
    \frac{\gamma\times\ddot{\gamma}}{\gamma\times\dot{\gamma}}\right)\eta\,\right|\ll
1\,.
\end{equation}
By assumption, ${|\eta| < \Rw \sim \Lambda^{-1/4}}$ is always small on
$N_\Delta(\Rw)$, and ${||\dot{\gamma}||^2 > 0}$ is bounded from below.
The condition in \eqref{SmallC} is therefore only violated at points where
${\gamma\times\dot{\gamma} = 0}$.  At these points, the Taylor
expansion in \eqref{SecondOrd} breaks down.

Despite the failure of the Taylor expansion at points where
${\gamma\times\dot{\gamma}=0}$, these points cause no difficulty.
Recall that points where ${\gamma\times\dot{\gamma} = 0}$ correspond
to critical points of the height function $z(\theta)$ on $C$.  By the Morse
assumption which follows Lemma \ref{FundLM}, 
the function ${\left(\gamma\times\dot{\gamma}\right)(\phi)}$ vanishes  
at only a finite number of isolated critical points
$\{\phi_1,\ldots,\phi_{2k}\}$ on $S^1$.   According to the
analysis at the end of Section \ref{Proof}, these points
are precisely the endpoints of the disjoint collection of intervals
${S^1_\pm \cap \left[T^2-\Delta(\varepsilon)\right]}$ in
\eqref{SSigns}.  At the endpoints, $\Phi_\hbar$ vanishes, and the
exact integrand $\left(X^{} \times
  X^{}\right)^*\!\widehat\Gamma_{\hbar}^*\,\chi_\Lambda$ in
\eqref{PullChiDD} is exponentially small (since $\widehat\Delta z$ is small).
Consequently, the troublesome points for the Taylor 
expansion can just be removed from the domain
of integration, with negligible error.

\begin{figure}[t]
\begin{center}
\includegraphics[scale=0.45]{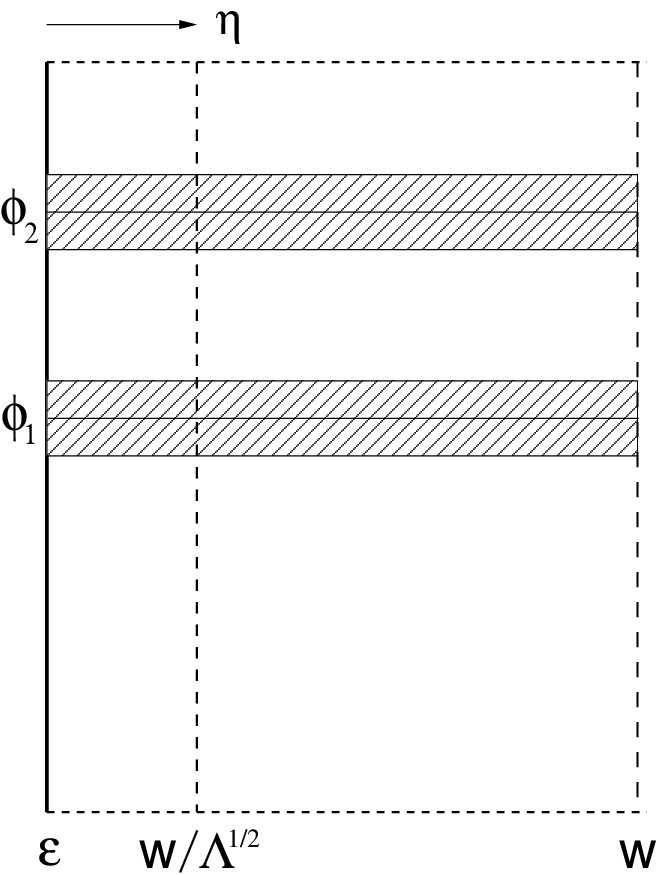}
\caption{Domain $N_\Delta(\Rw;\ell)$ on which
  ${|\gamma\times\dot{\gamma}|>m>0}$.}\label{Edge}
\end{center}
\end{figure}

Technically, about each critical point $\phi_c$,
\begin{equation}
\left(\gamma\times\dot{\gamma}\right)(\phi_c) \,=\, 0\,,\qquad\qquad
c= 1,\,\ldots,\,2k\,,
\end{equation}
we consider a small neighborhood $(\phi_c - \ell, \phi_c + \ell)$ with fixed
width ${\ell > 0}$.  Let ${\textrm{I}_c(\ell) \subset N_\Delta(\Rw)}$ be the
corresponding closed strip
\begin{equation}\label{Strip}
\textrm{I}_c(\ell) \,=\, \big\{ (\phi,\eta) \,\big|\, |\phi - \phi_c|
\le \ell,\, |\eta| \le \Rw\big\}\,,
\end{equation} 
and set 
\begin{equation}
N_\Delta(\Rw;\ell) \,=\, N_\Delta(\Rw) \,-\, \bigcup_{c=1}^{2k} \,
\textrm{I}_c(\ell)\,.
\end{equation}
See Figure \ref{Edge} for a sketch of $N_\Delta(\Rw;\ell)$ near one
boundary of the cylinder ${T^2 - \Delta(\varepsilon)}$.  The 
shaded regions indicate the strips of width $\ell$ which have been
excised about a pair of zeroes $\phi_1$ and $\phi_2$ of the function
$\gamma\times\dot{\gamma}$.

We choose the width ${\ell>0}$ of each strip to be small enough 
so that 
\begin{equation}\label{WidI}
\left|\sum_{c=1}^{2k}\,\int_{\textrm{I}_c(\ell)} \left(X^{} \times
  X^{}\right)^*\!\widehat\Gamma_{\hbar}^*\,\chi_\Lambda \right| \,<\,
\delta\,,\qquad\qquad
\left|\sum_{c=1}^{2k}\,\int_{\textrm{I}_c(\ell)}\Phi_\hbar\,\right|
\,<\, \delta\,.
\end{equation}
Both integrands in \eqref{WidI} are bounded at the
points $\phi_c$, so the integrals over $\textrm{I}_c(\ell)$ can be
made as small as desired by the choice of $\ell$.  Moreover, both
integrands are decreasing functions of
$\Lambda$ for sufficiently large $\Lambda$, so the width $\ell$ can be chosen
independently of $\Lambda$, our crucial requirement.

By definition, the function $(\gamma\times\dot{\gamma})(\phi)$ is now
bounded away from zero everywhere on the new domain $N_\Delta(\Rw;\ell)$,
\begin{equation}\label{BdTayL}
\left|\gamma\times\dot{\gamma}\right| \,>\,
m \,>\, 0 \quad\hbox{ on }\quad N_\Delta(\Rw;\ell)\,.
\end{equation}
Here $m$ is a constant which depends upon the curve $C$ and the parameter
${\ell}$, but not on $\Lambda$.  Because the respective contributions
\eqref{WidI} from the excised strips are small by
assumption, we are free to replace $N_\Delta(\Rw)$ by
$N_\Delta(\Rw;\ell)$ in the inequality \eqref{DiagBd} to be proven.
On the other hand, due to the lower bound in \eqref{BdTayL}, we will also have
uniform control of error terms such as \eqref{SmallC} in the
Taylor approximation on $N_\Delta(\Rw;\ell)$.

The remainder of the discussion proceeds in rough correspondence to the
asymptotic analysis near a crossing.  By analogy to the ball $B_0$ in
\eqref{BZero}, we introduce a smaller  
tubular neighborhood ${N_0 \subset N_\Delta(\Rw;\ell)}$ defined by
\begin{equation}\label{NZero}
N_0:\quad |\eta| \,<\, \frac{\Rw}{\sqrt{\Lambda}} \,\sim\, \Lambda^{-3/4}\,.
\end{equation}
We indicate the coaxial configuration schematically in Figure \ref{Edge}.
For points inside the small tube $N_0$, we will show that the
difference between the integrals of $\left(X^{} \times
  X^{}\right)^*\!\widehat\Gamma_{\hbar}^*\,\chi_\Lambda$ and $\Phi_\hbar$ is small,
\begin{equation}\label{NNought}
\left|\int_{N_0}\mskip-10mu\left(X^{} \times
  X^{}\right)^*\!\widehat\Gamma_{\hbar}^*\,\chi_\Lambda \,-\,
\int_{N_0}\mskip-5mu\Phi_\hbar\,\right| <M\,\delta\,.
\end{equation}
For points outside $N_0$ but inside $N_\Delta(\Rw;\ell)$, we
will show that both integrals are separately small, with 
\begin{equation}\label{None}
\left|\int_{N_\Delta(\Rw;\ell) - N_0} \left(X^{} \times
  X^{}\right)^*\!\widehat\Gamma_{\hbar}^*\,\chi_\Lambda\,\right| \,<\, M\,\delta\,,
\end{equation}
and
\begin{equation}\label{None2}
\left|\int_{N_\Delta(\Rw;\ell)-N_0} \Phi_\hbar \,\right| \,<\, M\,\delta\,.
\end{equation}
The extra factor of $1/\sqrt{\Lambda}$ in the definition of $N_0$ is simply
what is needed to ensure the inequality in \eqref{NNought}.

We first consider the points inside $N_0$.  From the second-order
expansion in \eqref{SecondOrd}, the ratio of the 
self-linking integrand to its approximation $\Phi_\hbar$ satisfies 
\begin{equation}\label{Smallq2}
\begin{aligned}
q \,=\, \frac{\left(X^{} \times
  X^{}\right)^*\!\widehat\Gamma_{\hbar}^*\,\chi_\Lambda}{\Phi_\hbar}
\,&=\, \left[1 \,+ \left(\ha
  \frac{\dot{\gamma}\times\dddot{\gamma}}{\dot{\gamma}\times\ddot{\gamma}}
  - \frac{\gamma\times\ddot{\gamma}}{\gamma\times\dot{\gamma}}\right)\eta\,+\, \CO\big(\eta^2\big)\right]\times\\
&\times\,\exp{\!\left[-\frac{\Lambda\,||\dot{\gamma}||^2 \,|\eta|}{2
      \left|1-\hbar\right|\left|\gamma\times\dot{\gamma}\right|}\left(\frac{\dot{\gamma}\cdot\ddot{\gamma}}{||\dot{\gamma}||^2} -
    \ha
    \frac{\gamma\times\ddot{\gamma}}{\gamma\times\dot{\gamma}}\right) \eta
  \,+\, \CO\big(\eta^3\big)\right]}\,.
\end{aligned}
\end{equation}
Again, to deal with the term ${1/(\dot{\gamma}\times\ddot{\gamma})}$ in
the prefactor of \eqref{Smallq2}, we
assume that any zeroes of $\dot{\gamma}\times\ddot{\gamma}$ are isolated, and we
remove small neighborhoods as necessary about those zeroes so
that the functions which multiply $\eta$ in both the prefactor and the
argument of the exponential in \eqref{Smallq2} are bounded,
independently of $\Lambda$.  

For any
point in the small tube $N_0$, the
argument of the exponential in \eqref{Smallq2} is bounded in magnitude
by  
\begin{equation}\label{VNzero}
\frac{\Lambda\,||\dot{\gamma}||^2 \,\eta^2}{2
      \left|1-\hbar\right|\left|\gamma\times\dot{\gamma}\right|}\cdot\left|\frac{\dot{\gamma}\cdot\ddot{\gamma}}{||\dot{\gamma}||^2} -
    \ha
    \frac{\gamma\times\ddot{\gamma}}{\gamma\times\dot{\gamma}}\right|
< \Lambda\, M \, \eta^2 \,<\, M \, \Rw^2\,,
\end{equation}
where we apply the conditions  ${|\gamma\times\dot{\gamma}|>m>0}$ as
well as ${|\eta| < \Rw /\sqrt{\Lambda}}$ in $N_0$.  Because
${\Rw \sim \Lambda^{-1/4}}$, this inequality means that
the argument of the exponential vanishes, and the prefactor
approaches unity, in the limit ${\Lambda\to\infty}$.  Thus for
sufficiently large $\Lambda$, the fractional error is small,
\begin{equation}
\left|q-1\right| \,<\, \delta\,.
\end{equation}
By the same idea in \eqref{estimate1},
\begin{equation}
\left|\int_{N_0}\mskip-10mu\left(X^{} \times
  X^{}\right)^*\!\widehat\Gamma_{\hbar}^*\,\chi_\Lambda \,-\,
\int_{N_0}\mskip-5mu\Phi_\hbar\,\right| \,\le\,
\left|q-1\right|\cdot\left|\int_{N_0} \Phi_\hbar\right| \,<\, M \,\delta\,,
\end{equation}
since we already know the integral of $\Phi_\hbar$ to be bounded and
independent of $\Lambda$ by the local computation in Section \ref{Local}.

The inequalities for points outside
$N_0$ are even easier.  

From the explicit expression for $\Phi_\hbar$ in \eqref{Case1},
\begin{equation}\label{BoundPhiH}
\big|\Phi_\hbar\big| \,<\, A \, \Lambda
\exp{\!\Big[-B\,\Lambda\,|\eta|\Big]}\,,\qquad\qquad A,B\,>\,0\,,
\end{equation}
for some positive constants $A$ and $B$.  So in the allowed range
${\Lambda^{-1/2}\,\Rw \le |\eta| \le \Rw}$ on the complement of $N_0$,
\begin{equation}
\begin{aligned}
\big|\Phi_\hbar\big| \,&<\, A \, \Lambda \exp{\!\left[-B \,
    \Lambda^{1/2} \, \Rw \right]}\,,\qquad\qquad \Rw \,=\, m \,
\Lambda^{-1/4}\,,\\
&=\, A \, \Lambda \exp{\!\left[-m\, B \,
    \Lambda^{1/4}\right]}\,.
\end{aligned}
\end{equation}
By taking $\Lambda$ sufficiently large, we can make the magnitude of
$\Phi_\hbar$ as small as desired on 
the complement of $N_0$ inside $N_\Delta(\Rw;\ell)$, from which the
bound in \eqref{None2} follows.  

To establish a similar bound for the pullback of $\chi_\Lambda$ in
\eqref{SecondOrd}, observe that the argument of the exponential obeys
\begin{equation}
\frac{||\dot{\gamma}||^2 \,|\eta|}{2
      \left|1-\hbar\right|\left|\gamma\times\dot{\gamma}\right|}\left[1 \,+
  \left(\frac{\dot{\gamma}\cdot\ddot{\gamma}}{||\dot{\gamma}||^2} -
    \ha
    \frac{\gamma\times\ddot{\gamma}}{\gamma\times\dot{\gamma}}\right)\eta\right]
>\, B\,|\eta|\,,\qquad\qquad B>0\,,
\end{equation}
provided that $\Lambda$ is sufficiently large and $\eta$ sufficiently
small.  Here ${B>0}$ is a suitable positive constant.  Then according
to the expansion in \eqref{SecondOrd},
\begin{equation}
\big|\left(X^{} \times
  X^{}\right)^*\!\widehat\Gamma_{\hbar}^*\,\chi_\Lambda\big|\,<\, A \,\Lambda\,
\exp{\!\Big[-B\,\Lambda\,|\eta|\Big]}\,,\qquad\qquad A,B\,>\,0\,,
\end{equation}
exactly as for the preceding bound on $\Phi_\hbar$ in
\eqref{BoundPhiH}.  The claim in \eqref{None} now follows by an
identical argument.

In total, the three inequalities in \eqref{NNought}, \eqref{None},
and \eqref{None2} finish the proof of the localization formula for
$\slk_\kappa(C)|_\Delta$ in the generic case ${\hbar\neq 1}$.\qquad$\square$

\vspace{2pt}\noindent
\textbf{{\small Symmetric case ${\hbar=1}$}}
\vspace{2pt}

For the Heisenberg-symmetric value ${\hbar=1}$, the localization
formula from Section \ref{Local} states ${\slk_\kappa(C)|_\Delta =
  0}$.  Consistent with this result, we establish the basic bound in
\eqref{DiagBd} by showing individually 
\begin{equation}\label{Sam1}
\left|\int_{N_\Delta(\Rw)} \left(X^{} \times
  X^{}\right)^*\!\widehat\Gamma_{\hbar}^*\,\chi_\Lambda\,\right| \,<\,
\delta\,,
\end{equation}
and
\begin{equation}\label{Sam2}
\left|\int_{N_\Delta(\Rw)} \Phi_\hbar \,\right| \,<\,
\delta\,.
\end{equation}

Our workhorse is the next-order expansion of the self-linking
integrand, which behaves differently for ${\hbar=1}$.  For the
argument of the heat kernel, the calculations in Section \ref{LocalD}
yield 
\begin{equation}\label{Oz1}
\frac{\Delta x^2 + \Delta y^2}{2\,\widehat\Delta z}
\,\buildrel{\hbar=1}\over =\, -\frac{3}{\eta}\cdot
\left[\frac{||\dot{\gamma}||^2
    \,+ \left(\dot{\gamma}\cdot\ddot{\gamma}\right) \eta \,+\,
    \CO\big(\eta^2\big)}{\left(\dot{\gamma}\times\ddot{\gamma}\right) \,+\, \ha
    \left(\dot{\gamma}\times\dddot{\gamma}\right) \eta \,+\, \CO\big(\eta^2\big)}\right].
\end{equation}
Similarly,
\begin{equation}\label{Sam3}
\begin{aligned}
\left(X^{} \times
  X^{}\right)^*\!\widehat\Gamma_{\hbar}^*\,\chi_\Lambda
\,&\buildrel{\hbar=1}\over =\,
-\frac{3 \Lambda\,d\phi\^d\eta}{2\pi\eta^2} \left[1 \,+\,
  \CO\big(\eta^2\big)\right]\times\,\\
&\times\,\exp{\!\left[-\frac{3\,\Lambda}{|\eta|}\cdot\left|\frac{||\dot{\gamma}||^2
    \,+ \left(\dot{\gamma}\cdot\ddot{\gamma}\right) \eta \,+\,
    \CO\big(\eta^2\big)}{\left(\dot{\gamma}\times\ddot{\gamma}\right) \,+\, \ha
    \left(\dot{\gamma}\times\dddot{\gamma}\right) \eta \,+\, \CO\big(\eta^2\big)}\right|\right]}\,.
\end{aligned}
\end{equation}
For sake of brevity, we omit the calculation leading to
\eqref{Sam3}.  The details of this formula are not
important.\footnote{Curiously, the order-$\eta$ correction to the
  prefactor in \eqref{Sam3} vanishes when ${\hbar=1}$.}

For the approximation $\Phi_\hbar$, recall the formula
\begin{equation}\label{SamPhi}
\Phi_\hbar \,\buildrel{\hbar=1}\over =\,-\frac{3 \Lambda}{2\pi
  \eta^2} \exp{\!\left[-\frac{3\,\Lambda\,||\dot{\gamma}||^2}{
      \left|\dot{\gamma}\times\ddot{\gamma}\right|}\frac{1}{|\eta|}\right]}\,d\phi\^d\eta\,.
\end{equation}
Then $\Phi_\hbar$ vanishes smoothly for ${\eta=0}$ and otherwise satisfies
\begin{equation}\label{Oz2}
\big|\Phi_\hbar\big| \,<\, \frac{M \Lambda}{\eta^2} \,
\exp{\!\left[-\frac{A\,\Lambda}{|\eta|}\right]}\,,\qquad\qquad A, M\,>\,0\,,
\end{equation}
for some positive constants $A$ and $M$.  For $\Lambda$ sufficiently
large, $\Phi_\hbar$ can be made as small as desired everywhere on
$N_\Delta(\Rw)$, and the
inequality in \eqref{Sam2} holds.

To treat the pullback of $\chi_\Lambda$ in the same fashion, note that
the denominator in \eqref{Oz1} obeys
\begin{equation}
\left|\left(\dot{\gamma}\times\ddot{\gamma}\right) \eta \,+\, \ha
\left(\dot{\gamma}\times\dddot{\gamma}\right) \eta^2\right| \,<\, A \, |\eta|
\,+\, B \, |\eta|^2 \,<\, 2 \, A \, |\eta|\,,
\end{equation}
provided $|\eta|$ is sufficiently small (with ${B |\eta|<A}$), as holds when $\Lambda$ is
sufficiently large.  Also in this regime, the numerator in \eqref{Oz1}
is bounded from below by 
\begin{equation}
\Big|\,||\dot{\gamma}||^2 \,+\,
  \left(\dot{\gamma}\cdot\ddot{\gamma}\right) \eta\,\Big| \,>\, m
\,>\, 0\,.
\end{equation}
Hence on the tubular neighborhood $N_\Delta(\Rw)$,
\begin{equation}
\big|\!\left(X^{} \times
  X^{}\right)^*\!\widehat\Gamma_{\hbar}^*\,\chi_\Lambda\big| \,<\,
\frac{M \Lambda}{\eta^2} \, \exp{\!\left[-\frac{3\,m\,\Lambda}{2\,A\,|\eta|}\right]}\,.
\end{equation}
This inequality has the same shape as that for $\Phi_\hbar$ in
\eqref{Oz2}, from which we reach the conclusion in \eqref{Sam1}.

The proof of Theorem \ref{MT} is complete.\qquad$\square$

\appendix
\section*{Acknowledgments}

We thank the anonymous referee for many helpful suggestions.  The work
of RZ was partially supported by the Ling Ma Fund of Northeastern University.   The work
of CB is supported in part under National Science Foundation Grant
No.~PHY-1620637.  Any opinions, findings, and conclusions or
recommendations expressed in this material are those of the authors
and do not necessarily reflect the views of the National Science
Foundation.

\bibliographystyle{unsrt}

\end{document}